%% file: main.tex
\documentclass[a4paper, 10pt]{amsart}
\pdfoutput=1
\usepackage[utf8]{inputenc}
\usepackage[T1]{fontenc}
\usepackage[english]{babel}

\usepackage{amsmath, amssymb}
\usepackage{algorithm}
\usepackage{algpseudocode}

\usepackage{graphicx}

\usepackage{subcaption}
\usepackage{tikz}

\newcommand{\R}{\mathbb{R}}
\newcommand{\N}{\mathbb{N}}

\newcommand{\T}{\mathcal{T}}
\newcommand{\bS}{\mathbb{S}}

\newcommand{\HH}{\mathcal{H}}
\newcommand{\DD}{\mathcal{D}}
\newcommand{\M}{\mathcal{M}}

\newcommand{\dd}{\mathrm{d}}
\newcommand{\di}{\mathrm{d}}
\newcommand{\D}{\mathrm{D}}
\newcommand{\E}{\mathcal{E}}
\newcommand{\SBV}{\mathrm{SBV}}
\newcommand{\BV}{\mathrm{BV}}

\newcommand{\SBD}{\mathrm{SBD}}

\newcommand{\GSBV}{\mathrm{GSBV}}
\newcommand{\GSBD}{\mathrm{GSBD}}
\newcommand{\U}{\mathcal{U}}
\newcommand{\F}{\mathcal{F}}
\newcommand{\K}{\mathcal{K}}

\newcommand{\strain}{\boldsymbol{\epsilon}}
\newcommand{\stiffness}{\boldsymbol{C}}
\newcommand{\stiff}{\boldsymbol{c}}
\newcommand{\1}{\mathbbm{1}}
\newcommand{\G}{\mathcal{G}}
\newcommand{\p}{\mathrm{\Pi}}
\newcommand{\VV}{\mathcal{X}}
\newcommand{\UU}{\mathcal{X}}
\newcommand{\RR}{\mathcal{R}}
\newcommand{\TOL}{\mathtt{TOL}}
\newcommand{\vv}{\mathbf{v}}
\newcommand{\rr}{\mathbf{r}}

\newcommand{\cal}{\mathcal}
\newcommand{\Om}{\Omega}
\newcommand{\om}{\omega}
\newcommand{\eps}{\varepsilon}
\newcommand{\wto}{\rightharpoonup}

\newcommand{\half}{\frac{1}{2}}
\newcommand{\tforall}{\text{for all }}

\newcommand{\coloneq}{\mathrel{\mathop:}=}

\newcommand{\abs}[1]{\lvert #1 \rvert}
\newcommand{\bigabs}[1]{\bigl\lvert #1 \bigr\rvert}
\newcommand{\Bigabs}[1]{\Bigl\lvert #1 \Bigr\rvert}
\newcommand{\biggabs}[1]{\biggl\lvert #1 \biggr\rvert}

\newcommand{\norm}[1]{\lVert #1 \rVert}
\newcommand{\bignorm}[1]{\bigl\lVert #1 \bigr\rVert}
\newcommand{\Bignorm}[1]{\Bigl\lVert #1 \Bigr\rVert}
\newcommand{\biggnorm}[1]{\biggl\lVert #1 \biggr\rVert}

\newcommand{\scprod}[2]{\langle #1 , #2 \rangle}

\newcommand{\llbracket}{[\!\!\:[}
\newcommand{\rrbracket}{]\!\!\:]}

\DeclareMathOperator{\diag}{diag}

\DeclareMathOperator*{\argmin}{arg\,min}
\DeclareMathOperator{\dist}{dist}

\DeclareMathOperator*{\Glim}{\Gamma-lim}
\DeclareMathOperator*{\Gliminf}{\Gamma-lim\,inf}
\DeclareMathOperator*{\Glimsup}{\Gamma-lim\,sup}

\DeclareMathOperator{\diver}{div}

 \theoremstyle{plain}
 \newtheorem{theorem}{Theorem}[section]
 \newtheorem{lemma}[theorem]{Lemma}
 
 \newtheorem{proposition}[theorem]{Proposition}

 \theoremstyle{definition}
 \newtheorem{definition}[theorem]{Definition}

 \theoremstyle{remark}
 \newtheorem{remark}[theorem]{Remark}

\numberwithin{equation}{section}

\graphicspath{{numerics/figures/}{shellmodel/figures/}}

\allowdisplaybreaks

\begin{document}

\title[Dimension-Reduction for Brittle Fractures on Thin Shells]{A Dimension-Reduction Model for Brittle Fractures on Thin Shells with Mesh Adaptivity}

\author[S. Almi]{Stefano Almi}
\address{Faculty of Mathematics, University of Vienna, Oskar-Morgenstern-Platz~1\\1090 Vienna, Austria}
\email{stefano.almi@univie.ac.at}
\author[S. Belz]{Sandro Belz}
\address{Department of Mathematics, Technical University Munich, Boltzmannstr.~3\\85748 Garching (Munich), Germany}
\email{sandro.belz@ma.tum.de}
\author[S. Micheletti]{Stefano Micheletti}
\author[S. Perotto]{Simona Perotto}
\address{MOX, Department of Mathematics, Politecnico di Milano, Piazza Leonardo da Vinci 32\\20133 Milano, Italy}
\email{stefano.micheletti@polimi.it, simona.perotto@polimi.it}

\begin{abstract}
In this paper we derive a new two-dimensional brittle fracture model for thin shells via dimension reduction, where the admissible displacements are only normal to the shell surface.
The main steps include to endow the shell with a small thickness, to express the three-dimensional energy in terms of the variational model of brittle fracture in linear elasticity, and
to study the $\Gamma$-limit of the functional as the thickness tends to zero.

The numerical discretization is tackled by first approximating the fracture through a phase field, following an Ambrosio-Tortorelli like approach, and then resorting to
an alternating minimization procedure, where the irreversibility of the crack propagation is rigorously imposed via an inequality constraint.
The minimization is enriched with an anisotropic mesh adaptation driven by an a posteriori error estimator, which allows us to sharply track the whole
crack path by optimizing the shape, the size, and the orientation of the mesh elements.

Finally, the overall algorithm is successfully assessed on two Riemannian settings and proves not to bias the crack propagation.
\end{abstract}


\maketitle

\input{introduction}
\input{shellmodel/shellmodel}
\input{numerics/numerics}

\appendix
\input{appendix}

\section*{Acknowledgement}
S.A. wishes to thank the Technical University of Munich, where he worked during the preparation of this paper, with partial support from the SFB project TRR109 \emph{Shearlet approximation of brittle fracture evolutions}.

\noindent
S.B. acknowledges the support of the DFG through the International Research Training Group IGDK 1754 \emph{Optimization and Numerical Analysis for Partial Differential Equations with Nonsmooth Structures'}. Furthermore, S.B. appreciates the hospitality of MOX, Politecnico di Milano during several visits to S.M. and S.P.

\noindent
Finally, S.M. and S.P. gratefully acknowledge the partial financial support by the INdAM-GNCS 2020 Projects.

\bibliography{./Bibliography.bib}
\bibliographystyle{abbrv}

\end{document}

%% file: introduction.tex
\section{Introduction}
The problem of finding reasonable two-dimensional  models of elasticity for plates and shells dates back to more than one hundred years ago with contributions of J. Bernoulli, L. Euler, G. R. Kirchhoff, T. von Kármán, and many others (see, e.g., the Kirchhoff-Love plate theory and the Föppl-von-Kármán equations in \cite{Foe1907,Kar1910,Kir1850,Lov1888/89}).

In recent works, a two dimensional model is usually obtained as a limit of a three dimensional one: the target surface (shell or plate) is endowed with a fictitious thickness $\rho>0$ and the limit as $\rho\to 0$ is studied. Considering the variational framework of elasticity, such a limit is computed in terms of $\Gamma$-convergence (see \cite{Dal1993}).
In the context of linearized elasticity, a comprehensive work by Ph.G. Ciarlet about two-dimensional models can be found in \cite{Cia1997} for thin plates and in \cite{Cia2000} for thin shells. In these monographs, the convergence of the solution to the three-dimensional model is considered, avoiding the notion of $\Gamma$-convergence. A justification of the above results in terms of $\Gamma$-convergence has been provided successively in \cite{Gen1999}. Related works in the case of non-linear elasticity can be found, for instance, in \cite{FriJamMorMue2003,FriJamMue2002,FriJamMue2006}.

In this paper, we develop and analyze a new two-dimensional model of brittle fractures on thin shells, moving from the variational theory of brittle fractures in linearly elastic materials (see \cite{FraMar1998}). Accordingly, the total energy of a body  $U \subset \R^3$ subject to a displacement $u\colon U \to \R^3$  is given by
\begin{equation}
  \label{eq:intro1}
 \half \int_{U} \hat\stiffness \hat\strain(u): \hat\strain(u) \,\dd x + \kappa \HH^2(J_u),
\end{equation}
where $\hat\stiffness$ is the stiffness tensor, $\hat\strain(u)$ stands for the symmetric gradient of $u$, $J_u$ is the jump set of $u$, $\HH^2$ denotes the two-dimensional Hausdorff measure, and $\kappa > 0$ is the toughness of the material. Because of compactness issues, the natural domain of definition of functional \eqref{eq:intro1} is $\SBD(U)$ or $\GSBD(U)$, the space of (generalized) special  functions of bounded deformation. We refer to \cite{AmbCosDal1997,ChaCri2019,Dal2013} for further details on these spaces. In this setting, we can find a dimension reduction result in \cite{BabHen2016}, where the authors investigate thin films bonded to a stiff substrate.
In case of nonlinear or anti-planar elasticity, where the bulk energy in \eqref{eq:intro1} is expressed in terms of the full gradient $\nabla u$, the domain of the energy functional simplifies to~$\SBV(U)$ or $\GSBV(U)$ (for details on the theory of these spaces see \cite{AmbFusPal2000}). Such an approach has been used to investigate dimension reduction problems in \cite{Bab2006,Bab2008,BraFon2001}. However,  all the cited works are obtained for a planar setting, i.e., the target two-dimensional surface is a subset of~$\R^2$.

The main contribution of this paper is the derivation of a brittle fracture model for general surfaces. As in \eqref{eq:intro1}, we stick to linearized elasticity.  Analogously to the anti-plane shear setting, which has been the first one tackled in the variational formulation of fractures (see \cite{DalToa2002}), we only consider displacement fields normal to the surface. The advantage of this choice is that the displacement field can be described by a scalar function, since its direction is fixed, so that we can still adopt the space~$\GSBV$. We defer the general case to future work.

In more detail, in Section~\ref{sec:two-dimensional-model} we introduce the geometric setting by considering a two-dimensional surface~$\phi(\om)\subset \R^3$, where $\om\subset\R^{2}$ is open, bounded, with Lipschitz boundary, and $\phi\colon \om\to \R^3$ is an immersion. We endow this surface with a thickness $\rho>0$, so that our reference configuration becomes $\Phi(\Om_{\rho})$, with $\Om_{\rho}\coloneq \om \times (-\frac\rho{2}, \frac{\rho}{2})$ and $\Phi$ a suitable extension of $\phi$.
We start with a strong formulation of brittle fracture, where a state of the system is described by a pair displacement-fracture $(u, K)$ for $K\subseteq \Phi(\Om_{\rho})$ closed and $u\in C^{1}(\Phi(\Om_\rho)\setminus K;\R^{3})$. In this setting, we express the functional~\eqref{eq:intro1} in curvilinear coordinates on $\Om_{\rho}$. After a second change of variables, we remove the dependence of the integration domain on the thickness, passing from~$\Om_{\rho}$ to~$\Om_{1}$. Then, we restrict the admissible displacements to those which are normal to the surface. As a standard approach in free-discontinuity problems,\cite{AmbFusPal2000} the functional is relaxed to~$\GSBV(\Om_1)$. Section~\ref{sec:shellmodel-model} is devoted to the $\Gamma$-convergence analysis as the thickness tends to zero. The limit functional will be defined for $u\in \GSBV(\Om_1)$ independent of~$x_3$ by
\begin{equation}
  \label{eq:intro2}
   \half \int_{\Om_1} b \abs{u}^2 \,\dd x + \frac{\mu}{2} \int_{\Om_1} \nabla u^\top A \nabla u  \,\dd x
  + \kappa \int_{J_u} \sqrt{\nu_u^\top A \nu_u \sqrt{a}} \,\dd \HH^2,
\end{equation}
where~$A$ is a symmetric positive definite matrix related to the metric tensor of~$\phi(\omega)$,~$b$ is a function of the stiffness~$\hat\stiffness$ and of the curvature of the surface, $\mu>0$ is the second Lam\'e coefficient, and~$\nu_u$ is the approximate unit normal to~$J_u$.
In contrast to the Euclidean setting, the geometry of the surface and the magnitude of the displacement~$\abs{u}$ directly contribute to the energy of the elastic shell due to curvature effects. Moreover, all the quantities in~\eqref{eq:intro2} are independent of~$x_3$, so that the integrals could be written on~$\om$.

Section~\ref{cha:numerics} introduces the regularized reduced model based on a phase-field approximation of~\eqref{eq:intro2} in the sense of L.~Ambrosio and V.M.~Tortorelli (see \cite{AmbTor1990,AmbTor1992})
\begin{align*}
  \F_\eps (u,v) \coloneq  \half & \int_\om b \abs{u}^2 \,\dd x + \frac{\mu}{2} \int_\om (v^2 + \eta_\eps) \nabla u^\top A \nabla u \,\dd x \\
  & + \kappa \int_\om \bigg[ \frac{1}{4\eps}  (1 - v)^2 \sqrt{a} +  \eps \nabla v^\top A \nabla v \bigg] \,\dd x
\end{align*}
for $u\in H^1(\om)$, $v\in H^1(\om;[0,1])$. Loosely speaking,~$v$ is a regularization of the crack set such that where~$v$ is close to one the material is sound, while where $v\ll 1$ a fracture is detected.

The minimization of the functional~$\F_{\eps}$ is used to simulate the fracture process driven by a time dependent boundary condition~$g$.
Following \cite{AlmBelNeg2019}, according to
a quasi-static approximation, at each time~$t_i$ a new state $(u(t_i), v(t_i))$ of the thin shell is computed as the limit as $j\to \infty$ of the alternating minimization
\begin{align}
  \label{eq:intro-minu}
  u_{j} &\coloneq \argmin \bigl\{ \F_\eps (u, v_{j-1}):  u\in H^1(\om),  u= g (t_{i}) \text{ on } \partial\om \bigr\}, \\[1mm]
  \label{eq:intro-minv}
  v_{j} &\coloneq \argmin \biggl\{\F_\eps (u_{j} ,v) + \frac{\alpha}{2\tau} \norm{v-v(t_{i-1})}^2_{L^2(\om)} : v\in H^1(\om), v \le v (t_{i-1}) \biggr\},
\end{align}
where $\alpha>0$ is a fixed parameter and $\tau>0$ is the time increment. In particular, the new state $(u(t_i), v(t_i))$ is a critical point of $\F_{\eps}(u, v) +  \tfrac{\alpha}{2\tau} \norm{v-v(t_{i-1})}^2_{L^2(\om)}$. We refer to Definitions \ref{def:discrete-crit-point} and \ref{def:cont-crit-point} and Proposition \ref{prop:critical-point} for further details.

We notice that the inequality constraint in \eqref{eq:intro-minv} takes care of the irreversibility condition (similar as in \cite{Gia2005,KneNeg2017,KneRosZan2013}), i.e., no healing of the crack is allowed.
As in \cite{AlmBelNeg2019}, the presence of an $L^{2}$-penalization in \eqref{eq:intro-minv} ensures the convergence to a unilateral gradient flow in the time continuous limit. Instead, to approximate a quasi-static evolution of the crack as in \cite{AlmBel2019,AlmNeg2019,BouFraMar2000,BurOrtSue2010,KneNeg2017,KneRosZan2013},
we choose~$\alpha$ small enough.

Following \cite{ArtForMicPer2015}, we couple the alternating minimization with an anisotropic mesh adaptation procedure. The rationale is that the phase field~$v$ is close to one in large portions of the domain, while it exhibits very steep gradients to reach zero in a thin neighborhood of the crack. For this reason, the mesh needs to be very fine only across the crack. As an alternative, to ensure accuracy, one should resort to a very fine uniform grid.
This might be prohibitive from a computational point of view, whereas an adaptive mesh significantly contains the computational effort of the algorithm. Moreover, compared to isotropic adapted meshes (see \cite{BurOrtSue2010,BurOrtSue2011}),  anisotropic grids further improve the efficiency of the numerical scheme, since the triangles can be stretched along the crack.

Since the alternating minimization \eqref{eq:intro-minu}--\eqref{eq:intro-minv} is discretized in a
finite element setting (as in \cite{ArtForMicPer2015,BurOrtSue2010}), in Section~\ref{sec:residual-estimate} we derive an anisotropic a posteriori error estimator to
measure the distance from an exact critical point. This estimator drives the generation of the new anisotropic adapted mesh relying on a metric based strategy proposed in \cite{ForPer2003,MicPer2008,MicPer2011}, as detailed in Section~\ref{sec:mesh-construction}. Compared to the numerical approaches of \cite{Bou2007a,BouFraMar2000,BurOrtSue2010,ArtForMicPer2015,ArtForMicPer2015a}, the main novelty is that we now take care of the inequality constraint in~\eqref{eq:intro-minv}.
This implies that the Euler-Lagrange conditions satisfied by a critical point $(u, v)$ of $\F_{\eps}(u, v) +  \tfrac{\alpha}{2\tau} \norm{v-v(t_{i-1})}^2_{L^2(\om)}$ are expressed by a variational inequality rather than an equality, in contrast to \cite{ArtForMicPer2015} where a penalization of the
irreversibility condition is adopted and to \cite{Bou2007a,BouFraMar2000} where~$v$ is set to~$0$ where~$v(t_{i-1})$ is below a certain threshold.

Finally, in Section~\ref{sec:num-examples} we assess the proposed model and the anisotropic discretization on two non-Euclidean settings, i.e., a piece of a cylinder and a piece of a sphere. This verification allows us to establish the reliability of the new dimensionally reduced brittle fracture model and of the anisotropic mesh adaptation procedure, which does not bias the evolution of the crack path.


%% file: shellmodel/shellmodel.tex
\section{The Two-Dimensional Model}
\label{sec:two-dimensional-model}
Before providing the technical details, we clarify some basic notation.

Given an open subset $U \subset \R^n$, we denote the space of \emph{functions of bounded variation} by~$\BV(U)$ and the space of \emph{special functions of bounded variation} by~$\SBV(U)$. The set of \emph{generalized special functions of bounded variation} is indicated by~$\GSBV(U)$. Furthermore, we define the following function spaces:
\begin{align*}
	\SBV^2 (U) &\coloneq \bigl\{ u \in \SBV (U) \colon \nabla u \in L^2 (U), \HH^{n-1} (S_u) < \infty \bigr\} \,, \\
	\GSBV^2 (U) &\coloneq \bigl\{ u \in \GSBV (U) \colon \nabla u \in L^2 (U), \HH^{n-1} (S_u) < \infty  \bigr\} \,,
\end{align*}
where~$\nabla{u}$ denotes the approximate gradient of~$u$,~$S_{u}$ is the discontinuity set of~$u$, and~$\HH^{n-1}$ stands for the $(n-1)$-dimensional Hausdorff measure.  We refer to \cite{AmbFusPal2000,EvaGar1992} for all the definitions and details on the theory of functions of bounded variation. We recall here that, for $u \in \GSBV(U)$, the set~$S_{u}$ is $\HH^{n-1}$-rectifiable. We will denote by~$\nu_u$ the approximate unit normal to~$S_u$, whereas, for a generic rectifiable set~$K$, we denote by~$\nu_K$ the associated approximate unit normal.
We further notice that~$\GSBV^2(U)$, unlike~$\GSBV(U)$, is a vector space (see \cite{DalFraToa2005}).

Throughout the paper we systematically use the Einstein summation convention, where Greek indices take values~$1$ and~$2$, and Latin indices run form $1$ to $3$.

\subsection{Geometric Setting}
\label{sec:shellmodel-geometric-setting}
\input{shellmodel/setting}

\subsection{Dimension Reduction}
\label{sec:shellmodel-model}
\input{shellmodel/proof}


%% file: shellmodel/setting.tex
Let $\omega \subset \R^2$ be an open and bounded set, and let  $\phi \in C^2(\bar\om ; \R^3)$ be an injective immersion, i.e., the tangent vectors $a_\alpha = \partial_\alpha \phi$ are linearly independent.
Defining the vector $a_3 \coloneq \frac{a_1 \times a_2}{\norm {a_1 \times a_2}}$, normal to the surface~$\phi (\omega)$, we obtain the basis $\{a_1,a_2,a_3\}$ of~$\R^3$. In Figure~\ref{fig:surface} we find an illustration of this configuration.
The contravariant basis~$\{a^i\}$ is defined by $a^i \cdot a_j= \delta^i_j$, where~$\delta^i_j$ denotes the Kronecker delta, $a_3 = a^3$.
The covariant components of the metric tensor are given by $a_{\alpha \beta} \coloneq a_\alpha \cdot a_\beta$.  We set $(a^{\alpha \beta}) \coloneq (a_{\alpha\beta})^{-1}$ which is its contravariant component matrix. Note that $a^{\alpha \beta} = a^\alpha \cdot a^\beta$. Moreover, we simply define $a\coloneq \det(a_{ij})$.

\begin{figure}
	\input{shellmodel/figures/surface.tex}
	\caption{\label{fig:surface}Geometric setting of the surface.}
\end{figure}
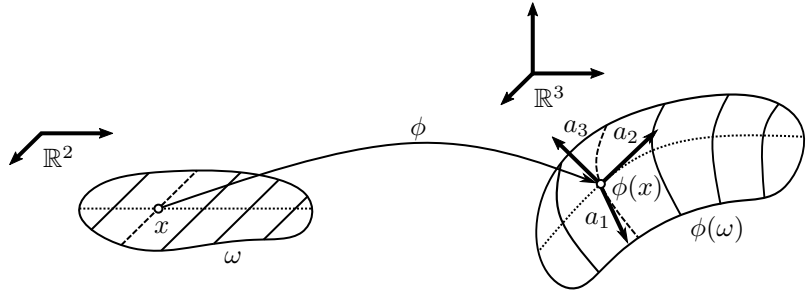

The covariant components~$b_{\alpha \beta}$, the mixed components~$b_\alpha^\beta$ of the curvature tensor,
and the Christoffel symbols~$\Gamma^\sigma_{\alpha\beta}$ are defined by
\begin{equation}
  \label{eq:curvature}
  b_{\alpha \beta} \coloneq a_3 \partial_{\alpha} a_\beta\,, \qquad b^\alpha_\beta \coloneq a^{\alpha \sigma}  b_{\sigma \beta} \,, \qquad \Gamma^\sigma_{\alpha\beta} \coloneq a^\sigma \partial_\alpha a_\beta\,,
\end{equation}
respectively. Notice that we omit the dependence on spatial variable when not explicitly needed.

\begin{remark}
  By the assumptions on $\phi$, we obtain that there exist two positive constants $c$ and $C$, both independent of $x \in \om$, such that
  \begin{equation}
    \label{eq:metric-norm-equivalence}
    c \abs{\zeta}^2  < a_{\alpha \beta} \zeta^{\alpha} \zeta^{\beta} < C \abs{\zeta}^2 \quad \tforall \zeta \in \R^2 \,.
  \end{equation}
  We further make use of the continuity of $\phi$ on the compact set $\bar \om$ to obtain upper and lower bounds for all the quantities in \eqref{eq:curvature}.

  In this work we only deal with manifolds that are covered by one single chart~$\phi$. To deal with more complex manifolds, e.g., compact manifolds, such as a sphere or a torus, we have to resort to more than one chart, each one satisfying \eqref{eq:metric-norm-equivalence}, and then to glue them properly.
\end{remark}

We now modify the surface $\phi (\omega)$ by adding a thickness, $\rho>0$, as illustrated in Figure~\ref{fig:thick_surface}.
Thus, we define $\Om_\rho \coloneq \omega \times \bigl(-\frac{\rho}{2}, \frac{\rho}{2}\bigr)$ and the map
$\Phi \colon \Om_\rho \to \R^3$ by
\begin{equation}\label{Phiona}
  \Phi (x) \coloneq \phi (x_1, x_2) + x_3 a_3 \quad \tforall x = (x_1,x_2,x_3) \in \Om_\rho \,,
\end{equation}
with $\phi (\omega) = \Phi(\omega \times \{ 0 \})$, that is,~$\phi(\om)$ is the middle surface of $\Phi(\Om_\rho)$.
We recall that in view of Theorem~3.1-1 in \cite{Cia2000} it is not restrictive to assume that $\Phi$ is a diffeomorphism.

Concerning the notation related to $\Phi(\Om_\rho)$, symbols with or without a hat are associated with the original Cartesian ($\Phi (\Om_\rho)$) or curvilinear ($\Om_\rho$) coordinate system, respectively.
In particular, it is understood that $x\in \Om_\rho$ with $\hat x = \Phi(x)$ when related in the same statement.
We define the covariant basis $g_i  \coloneq \partial_i \Phi$ and the corresponding metric tensor $g_{ij}  \coloneq g_i \cdot g_j$. By \eqref{Phiona}, we obtain
\begin{equation*}
  g_\alpha = a_\alpha + x_3 \partial_\alpha a_3 \quad \text{and} \quad g_3 = a_3 = a^3 = g^3 \,.
\end{equation*}
The contravariant basis $\lbrace g^i\rbrace$ denotes the dual basis of the covariant basis,
i.e., $g_i \cdot g^j = \delta_i^j$.
It follows that the inverse of $(g_{ij})$ is given by $g^{ij} \coloneq g^i \cdot g^j$. Additionally, we define $g \coloneq \det (g_{ij})$.
For the mapping $\Phi$, we also introduce the corresponding Christoffel symbols, denoted by $\Lambda_{ij}^{k} \coloneq g^k \cdot \partial_{i} g_j$, such that the symmetry condition, $\Lambda_{ij}^{k} = \Lambda_{ji}^{k}$, holds.

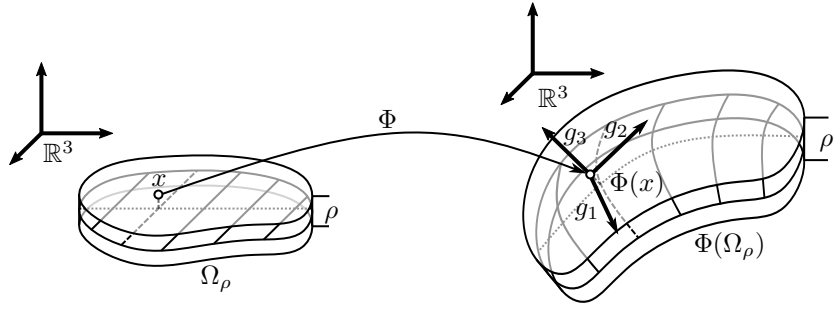
\begin{figure}
	\input{shellmodel/figures/thick_surface.tex}
	\caption{\label{fig:thick_surface}Geometric setting of the thickened surface.}
\end{figure}

\subsection{The Reference Model}\label{sec:mathmodel}
In order to derive the two-dimensional model, we start from the brittle fracture energy from G.A.~Francfort and J.-J.~Marigo\cite{FraMar1998} in the original Cartesian coordinates, given by
\begin{equation}
  \label{cartesianEnergy}
  E(\hat u, \hat K_\rho) \coloneq \half \int_{\Phi(\Om_\rho) \setminus \hat K_\rho} \hat \stiffness \hat \strain(\hat u) : \hat \strain (\hat u) \,\dd \hat x + \kappa \HH^2 (\hat K_\rho) \,,
\end{equation}
for $\hat u \in C^1(\Phi(\Om_\rho)\setminus \hat K_\rho) ; \R^3 )$ describing the displacement field and for $\hat K_\rho \subset \Phi(\Om_\rho)$ a closed and $\HH^2$-rectifiable set describing the fracture. The constant $\kappa >0$ denotes the toughness, which is a material dependent constant. The stiffness tensor~$\hat \stiffness$ is given by
\begin{equation*}
  \hat \stiffness^{ijkl} = \lambda \delta^{ij} \delta^{kl} + \mu (\delta^{ik} \delta^{jl} + \delta^{il} \delta^{jk})
\end{equation*}
with Lamé coefficients $\lambda \geq 0$ and $\mu > 0$. The symbol~$:$ in \eqref{cartesianEnergy} denotes the usual tensor product
\begin{equation*}
  \hat \stiffness \hat \strain(\hat u) : \hat \strain (\hat u) = \hat \stiffness^{ijkl} \hat \strain_{ij}(\hat u)  \hat \strain_{kl}(\hat u) \,.
\end{equation*}
Furthermore, $\hat \strain(\hat u)$ denotes the strain given by the symmetric gradient
\begin{equation*}
  \hat \strain (\hat u) \coloneq \half \bigl( \nabla \hat u + (\nabla \hat u)^\top \bigr) \qquad \hat \strain_{ij} (\hat u) \coloneq \half \bigl( \partial_i \hat u_j + \partial_j \hat u_i \bigr)\,.
\end{equation*}
We remark that the following symmetries hold:
\begin{equation*}
  \hat \stiffness^{ijkl} = \hat \stiffness^{jikl} = \hat \stiffness^{klij} \quad \text{and} \quad  \hat \strain_{ij}(\hat u) = \hat \strain_{ji} (\hat u) \,.
\end{equation*}

Following the strategy of \cite{Cia2000}, we express \eqref{cartesianEnergy}  in terms of curvilinear coordinates. For this purpose, we express the vector field $\hat u$ in terms of the covariant basis, by defining $u_i \colon \Om_\rho \to \R$ such that
\begin{equation}
  \label{vectorFieldTransformation}
  \hat u (\hat x)  = u_i(x) g^i(x) \quad \text{or equivalently} \quad  u_j(x) = \hat u(\hat x) \cdot g_j(x)  \,.
\end{equation}
For $K_\rho \coloneq \Phi^{-1}(\hat K_\rho)$, $u\in  C^1(\Omega_\rho \setminus K_\rho; \R^3)$ and $\hat u\in C^1(\Phi(\Om_\rho \setminus K_\rho); \R^3)$ related by~\eqref{vectorFieldTransformation}, we get
\begin{equation}
  \label{curvedEnergy}
  E(\hat u , \hat K_\rho) = \half \int_{\Om_\rho\setminus K_\rho} \stiffness\strain(u) : \strain(u) \sqrt{g} \,\dd x + \kappa \int_{K_\rho} \sqrt{[\nu_{K_\rho}]_i g^{ij} [\nu_{K_\rho}]_j } \sqrt{g} \,\dd \HH^2 \,,
\end{equation}
where~$[\nu_{K_\rho}]_k$ is the $k$-th component of the unit normal to the surface~$K_\rho$, $\strain(u)$ stands for the strain in the curvilinear setting
\begin{equation*}
  \strain_{ij} (u) \coloneq\half \bigl( \partial_i u_j  + \partial_j u_i ) - u_k \Lambda^k_{ij} \,,
\end{equation*}
and $\stiffness$ is the elasticity tensor in the curvilinear framework
\begin{equation*}
  \stiffness^{ijkl} \coloneq \lambda g^{ij} g^{kl} + \mu (g^{ik} g^{jl} + g^{il} g^{jk}) \,.
\end{equation*}

A simple scaling in the variable~$x_3$ provides an integration domain independent of~$\rho$, namely,
\begin{equation*}
  \pi_\rho \colon
  \left\{
    \begin{aligned}
      \Om &\to \Om_\rho \\
      x &\mapsto (x_1,x_2,\rho x_3)
    \end{aligned}
  \right.
  \quad \text{with } \Om \coloneq \Om_1 =  \omega \times \biggl( -\half, \half \biggr) \,.
\end{equation*}
For any closed set $K_\rho \subset \Om_\rho$, we let $K \coloneq \pi_\rho^{-1} (K_\rho)$. For any scalar, vector, or tensor field~$q$, we add a subscript~$\rho$ to denote the composition with~$\pi_\rho$, i.e., $q_\rho \coloneq q \circ \pi_\rho$.
In particular, for all $u \in C^1(\Om_\rho \setminus K_\rho; \R^3)$ we define $u_\rho \coloneq u\circ \pi_\rho$ and, for $w \in C^1(\Om \setminus K; \R^3)$,
\begin{equation}
  \label{eq:curved-rescaled-strain}
  \begin{split}
    \strain_{\alpha \beta,\rho} (w) &\coloneq \half (\partial_{\alpha} w_\beta + \partial_\beta w_\alpha) - w_k \Lambda_{\alpha \beta, \rho}^k \\
    \strain_{\alpha 3,\rho} (w) &\coloneq \half \biggl( \partial_{\alpha} w_3 + \frac{1}{\rho} \partial_3 w_\alpha \biggr) - w_k \Lambda_{\alpha 3, \rho}^k \\
    \strain_{3 3,\rho} (w) &\coloneq \frac{1}{\rho}  \partial_3 w_3 - w_k \Lambda_{33, \rho}^k \,.
  \end{split}
\end{equation}
One can easily check that $\strain_\rho (u_\rho) = \strain(u) \circ \pi_\rho$, so that the energy functional \eqref{curvedEnergy} can be written as
\begin{align} \label{eq:scaled-curved-energy}
  E(\hat u , \hat K) = \frac{\rho}{2} &  \int_{\Om\setminus K} \stiffness_\rho \strain_\rho (u_{\rho}) : \strain_\rho (u_{\rho}) \sqrt{g_\rho} \,\dd x  \\
  & + \kappa \rho \int_{K} \sqrt{[D^{\rho} \nu_{K}]_i g_\rho^{ij} [D^{\rho} \nu_{K}]_j } \sqrt{g_\rho} \,\dd\HH^2 \,.\nonumber
\end{align}
where $D^{\rho} \coloneq \diag( 1, 1, 1/\rho)$.

Hereafter, we restrict the model to the case of displacements that are normal to the middle surface, i.e., of the form $u = (0, 0, u_3)$, so that \eqref{vectorFieldTransformation} is equivalent to $\hat u = u_3 g^3 = u_3 a^3$. Hence, the whole problem can be expressed in terms of a scalar function $u$ and, with a slight abuse of notation, we set $\strain (u) \coloneq \strain (0,0,u)$ for all $u \in C^1(\Om_\rho \setminus K_\rho)$.

Since $\Lambda_{i3}^3 = a^3 \partial_i a_3 = 0$, by \eqref{eq:curved-rescaled-strain} we obtain, for all $u_\rho \in C^1(\Om \setminus K)$,
\begin{equation}
  \label{scaledStrain}
  \strain_{\alpha \beta, \rho} (u_\rho) = - \Lambda^3_{\alpha \beta, \rho} u_\rho\,, \qquad \strain_{\alpha 3, \rho} (u_\rho) = \half \partial_{\alpha} u_\rho\,, \qquad \strain_{3 3, \rho} (u_\rho) = \frac1\rho \partial_{3} u_\rho \,.
\end{equation}

Finally, we recall Theorems~3.2-1 and~3.3-1 in \cite{Cia2000}, which state some important convergence results of the geometric quantities in \eqref{eq:scaled-curved-energy}, for $\rho \to 0$.
\begin{proposition}
  \label{scaledFuncProp}
  With the definitions above there holds the following:
  \begin{align}
    g_\rho
    &= a + O(\rho) \,, \notag \\
    \notag
    g_\rho^{\alpha \beta}
    &= a^{\alpha \beta} 
      + O(\rho), \quad  g_\rho^{\alpha 3} = 0, \quad  g_\rho^{33} = 1,\\
    \label{scaledChristoffel}
    \Lambda_{\alpha \beta,\rho}^3
     &= b_{\alpha \beta} + O(\rho) \,,
  \end{align}
 where we recall that $g \coloneq \det(g_{ij})$ and $a \coloneq \det(a_{ij})$.
  The convergence rates, as $\rho \to 0$, are uniform, i.e., they do not depend on $x \in \Om$.
  Furthermore, there exist $c,C>0$ such that, for every $\rho > 0$ sufficiently small,
  \begin{equation}
    \label{eq:metric-tensor-bound}
    c \abs{\zeta}^2 \leq g_{ij,\rho} \zeta^i \zeta^{j} \leq C \abs{\zeta}^2 \quad \text{and} \quad c \abs{\zeta}^2 \leq g^{ij}_\rho \zeta_i \zeta_{j} \leq C \abs{\zeta}^2 \quad  \tforall \zeta \in \R^3 \,.
  \end{equation}
\end{proposition}

\begin{proposition}\label{elasticityFormulas}
  The following relations hold:
  \begin{align*}
    \stiffness_\rho^{\alpha \beta \sigma \tau} &= \lambda a^{\alpha \beta} a^{\sigma \tau} + \mu (a^{\alpha \sigma} a^{\beta \tau} + a^{\alpha \tau} a^{\beta \sigma}) + O(\rho) \,, &\stiffness_\rho^{\alpha \beta \sigma 3} &=0 \,, \\
    \stiffness_\rho^{\alpha 3 \beta 3} &= \mu a^{\alpha \beta} +  O(\rho)\,, & \stiffness_\rho^{\alpha \beta 33} &= \lambda a^{\alpha\beta} + O(\rho) \,, \\
    \stiffness_\rho^{\alpha 333} &= 0 \,, & \stiffness_\rho^{3333} &= \lambda + 2\mu \,.
  \end{align*}
  The convergence rates as $\rho \to 0$ are uniform, i.e., they do not depend on $x \in \Om$.  Furthermore, there exist some constants $c,C>0$ such that, for $\rho>0$ sufficiently small,
  \begin{equation}
    \label{eq:positive-definite-stiffness}
    c \abs{\mathrm{M}}^2 \leq \stiffness^{ijkl}_\rho \mathrm{M}_{ij} \mathrm{M}_{kl} \leq C \abs{\mathrm{M}}^2 \quad \tforall  \mathrm{M} \in \R^{3\times 3 } \text{ symmetric} \,,
  \end{equation}
  where $\abs{\cdot}$ stands for the Frobenius norm.
\end{proposition}

As a consequence of Proposition~\ref{scaledFuncProp}, we can rewrite \eqref{eq:scaled-curved-energy} as
\begin{align} \label{eq:scaled-curved-energy-2}
  E(\hat u , \hat K) = \displaystyle \frac{\rho}{2} & \int_{\Om\setminus K} \stiffness_\rho \strain_\rho (u_{\rho}) : \strain_\rho (u_{\rho}) \sqrt{g_\rho} \,\dd x  \\
 & + \displaystyle\kappa \rho \int_{K} \sqrt{[\nu_{K}]_\alpha g_\rho^{\alpha \beta} [\nu_{K}]_\beta  + \frac{1}{\rho^2} [\nu_{K}]_3^2 } \sqrt{g_\rho} \,\dd\HH^2 \,. \nonumber
\end{align}


%% file: shellmodel/figures/surface.tex
\begingroup%
  \makeatletter%
  \providecommand\color[2][]{%
    \errmessage{(Inkscape) Color is used for the text in Inkscape, but the package 'color.sty' is not loaded}%
    \renewcommand\color[2][]{}%
  }%
  \providecommand\transparent[1]{%
    \errmessage{(Inkscape) Transparency is used (non-zero) for the text in Inkscape, but the package 'transparent.sty' is not loaded}%
    \renewcommand\transparent[1]{}%
  }%
  \providecommand\rotatebox[2]{#2}%
  \newcommand*\fsize{\dimexpr\f@size pt\relax}%
  \newcommand*\lineheight[1]{\fontsize{\fsize}{#1\fsize}\selectfont}%
  \ifx\svgwidth\undefined%
    \setlength{\unitlength}{340.15748031bp}%
    \ifx\svgscale\undefined%
      \relax%
    \else%
      \setlength{\unitlength}{\unitlength * \real{\svgscale}}%
    \fi%
  \else%
    \setlength{\unitlength}{\svgwidth}%
  \fi%
  \global\let\svgwidth\undefined%
  \global\let\svgscale\undefined%
  \makeatother%
  \begin{picture}(1,0.33333333)%
    \lineheight{1}%
    \setlength\tabcolsep{0pt}%
    \put(0,0){\includegraphics[width=\unitlength,page=1]{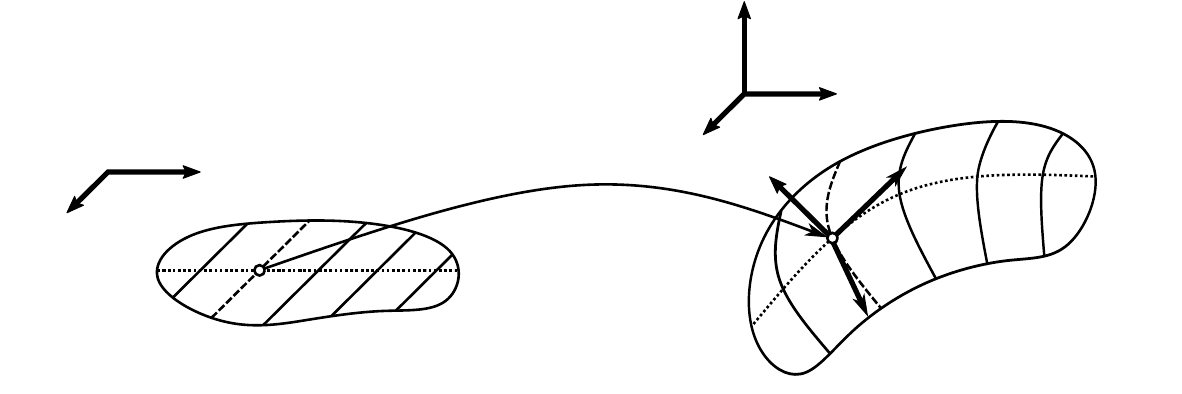}}%
    \put(0.71832509,0.18180854){\color[rgb]{0,0,0}\makebox(0,0)[lt]{\lineheight{1.25}\smash{\begin{tabular}[t]{l}$a_2$\end{tabular}}}}%
    \put(0.68721674,0.0869864){\color[rgb]{0,0,0}\makebox(0,0)[lt]{\lineheight{1.25}\smash{\begin{tabular}[t]{l}$a_1$\end{tabular}}}}%
    \put(0.66385472,0.18922639){\color[rgb]{0,0,0}\makebox(0,0)[lt]{\lineheight{1.25}\smash{\begin{tabular}[t]{l}$a_3$\end{tabular}}}}%
    \put(0.29228918,0.04479412){\color[rgb]{0,0,0}\makebox(0,0)[lt]{\lineheight{1.25}\smash{\begin{tabular}[t]{l}$\omega$\end{tabular}}}}%
    \put(0.21491836,0.0781105){\color[rgb]{0,0,0}\makebox(0,0)[lt]{\lineheight{1.25}\smash{\begin{tabular}[t]{l}$x$\end{tabular}}}}%
    \put(0.49571816,0.18585485){\color[rgb]{0,0,0}\makebox(0,0)[lt]{\lineheight{0}\smash{\begin{tabular}[t]{l}$\phi$\\\end{tabular}}}}%
    \put(0.63165863,0.2165192){\color[rgb]{0,0,0}\makebox(0,0)[lt]{\lineheight{0}\smash{\begin{tabular}[t]{l}$\R^3$\end{tabular}}}}%
    \put(0.71648076,0.119428){\color[rgb]{0,0,0}\makebox(0,0)[lt]{\lineheight{0}\smash{\begin{tabular}[t]{l}$\phi(x)$\end{tabular}}}}%
    \put(0.80285721,0.07406549){\color[rgb]{0,0,0}\makebox(0,0)[lt]{\lineheight{0}\smash{\begin{tabular}[t]{l}$\phi(\omega)$\\\end{tabular}}}}%
    \put(0.09326003,0.15015014){\color[rgb]{0,0,0}\makebox(0,0)[lt]{\lineheight{1.25}\smash{\begin{tabular}[t]{l}$\R^2$\end{tabular}}}}%
  \end{picture}%
\endgroup%

%% file: shellmodel/figures/thick_surface.tex
\begingroup%
  \makeatletter%
  \providecommand\color[2][]{%
    \errmessage{(Inkscape) Color is used for the text in Inkscape, but the package 'color.sty' is not loaded}%
    \renewcommand\color[2][]{}%
  }%
  \providecommand\transparent[1]{%
    \errmessage{(Inkscape) Transparency is used (non-zero) for the text in Inkscape, but the package 'transparent.sty' is not loaded}%
    \renewcommand\transparent[1]{}%
  }%
  \providecommand\rotatebox[2]{#2}%
  \newcommand*\fsize{\dimexpr\f@size pt\relax}%
  \newcommand*\lineheight[1]{\fontsize{\fsize}{#1\fsize}\selectfont}%
  \ifx\svgwidth\undefined%
    \setlength{\unitlength}{340.15748031bp}%
    \ifx\svgscale\undefined%
      \relax%
    \else%
      \setlength{\unitlength}{\unitlength * \real{\svgscale}}%
    \fi%
  \else%
    \setlength{\unitlength}{\svgwidth}%
  \fi%
  \global\let\svgwidth\undefined%
  \global\let\svgscale\undefined%
  \makeatother%
  \begin{picture}(1,0.33333333)%
    \lineheight{1}%
    \setlength\tabcolsep{0pt}%
    \put(0,0){\includegraphics[width=\unitlength,page=1]{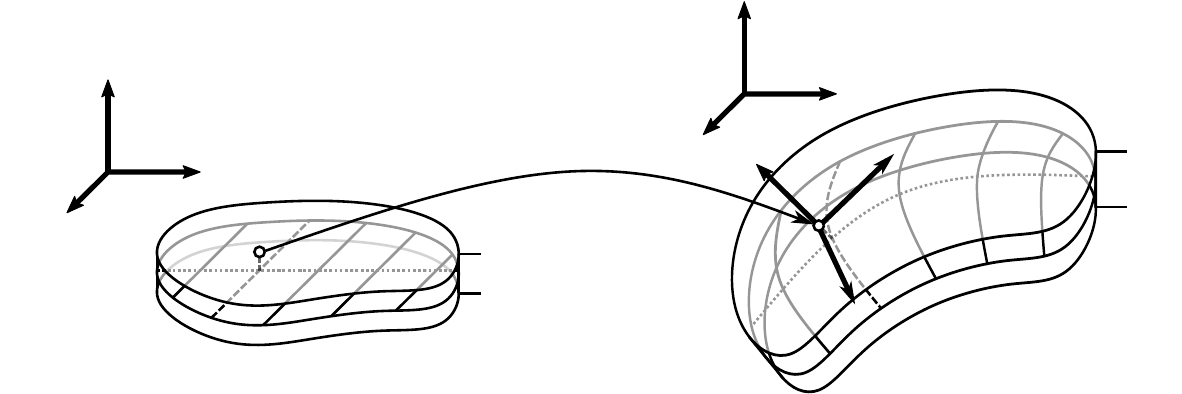}}%
    \put(0.21203383,0.129403){\color[rgb]{0,0,0}\makebox(0,0)[lt]{\lineheight{1.25}\smash{\begin{tabular}[t]{l}$x$\end{tabular}}}}%
    \put(0.70840324,0.187872){\color[rgb]{0,0,0}\makebox(0,0)[lt]{\lineheight{1.25}\smash{\begin{tabular}[t]{l}$g_2$\end{tabular}}}}%
    \put(0.67619241,0.09745951){\color[rgb]{0,0,0}\makebox(0,0)[lt]{\lineheight{1.25}\smash{\begin{tabular}[t]{l}$g_1$\end{tabular}}}}%
    \put(0.66385477,0.18036785){\color[rgb]{0,0,0}\makebox(0,0)[lt]{\lineheight{1.25}\smash{\begin{tabular}[t]{l}$g_3$\end{tabular}}}}%
    \put(0.94461498,0.17405192){\color[rgb]{0,0,0}\makebox(0,0)[lt]{\lineheight{1.25}\smash{\begin{tabular}[t]{l}$\rho$\end{tabular}}}}%
    \put(0.40185086,0.09464616){\color[rgb]{0,0,0}\makebox(0,0)[lt]{\lineheight{1.25}\smash{\begin{tabular}[t]{l}$\rho$\end{tabular}}}}%
    \put(0.09326002,0.1557404){\color[rgb]{0,0,0}\makebox(0,0)[lt]{\lineheight{1.25}\smash{\begin{tabular}[t]{l}$\R^3$\end{tabular}}}}%
    \put(0.2659867,0.02111181){\color[rgb]{0,0,0}\makebox(0,0)[lt]{\lineheight{1.25}\smash{\begin{tabular}[t]{l}$\Omega_\rho$\end{tabular}}}}%
    \put(0.8066754,0.05690343){\color[rgb]{0,0,0}\makebox(0,0)[lt]{\lineheight{1.25}\smash{\begin{tabular}[t]{l}$\Phi(\Omega_\rho)$\end{tabular}}}}%
    \put(0.71446978,0.12426419){\color[rgb]{0,0,0}\makebox(0,0)[lt]{\lineheight{1.25}\smash{\begin{tabular}[t]{l}$\Phi(x)$\end{tabular}}}}%
    \put(0.46034861,0.19410636){\color[rgb]{0,0,0}\makebox(0,0)[lt]{\lineheight{0}\smash{\begin{tabular}[t]{l}$\Phi$\\\end{tabular}}}}%
    \put(0.63606828,0.21651937){\color[rgb]{0,0,0}\makebox(0,0)[lt]{\lineheight{0}\smash{\begin{tabular}[t]{l}$\R^3$\end{tabular}}}}%
  \end{picture}%
\endgroup%

%% file: shellmodel/proof.tex
With a view to the limit for $\rho \to 0$, we rescale the energy in \eqref{eq:scaled-curved-energy-2} by $\rho^{-1}$ and observe that, as long as $\rho > 0$, such a scaling does not change the ``three-dimensional'' minimizer of the functional.

It is a standard, in the theory of free discontinuity problems, to relax the functional \eqref{eq:scaled-curved-energy-2} from $C^1(\Om \setminus K)$ to the space $\GSBV(\Om)$ and to replace the set $K$ with the discontinuity set~$S_{u}$.
Hence, for all $u \in \GSBV(\Om)$ and for all $\rho>0$, we define the functional
\begin{equation*}
  G_\rho (u) \coloneq \half \int_{\Om} \stiffness_\rho \strain_\rho (u) \colon  \strain_\rho (u) \sqrt{g_\rho } \,\dd x
  + \kappa \int_{S_{u}} \sqrt{[\nu_{u}]_\alpha g_\rho^{\alpha \beta} [\nu_{u}]_\beta  + \frac{1}{\rho^2} [\nu_{u}]_3^2 } \sqrt{g_\rho} \,\dd \HH^2.
\end{equation*}
The current goal is the computation of the $\Gamma$-limit of the sequence of functionals $G_\rho$ as $\rho \to 0$. For this purpose, we introduce the function space:
\begin{align*}
  \U &\coloneq \bigl\{ u\in \GSBV^2 (\Om) \colon \partial_3 u = 0, [\nu_u]_3 = 0 \bigr\}.
\end{align*}
\begin{remark}
  \label{r-nox3}
  Conditions $\partial_3 u = 0$ and $[\nu_u]_3 = 0$ imply that $u\in \U$ is independent of $x_3$. This can be easily checked for~$u\in \U\cap \SBV^2(\Om)$, since the third component of the distributional derivative~$\D u$ is zero, so that~$u$ is constant with respect to~$x_3$. By a truncation argument, this can be extended to every~$u\in \U$. Therefore, we can identify~$\U$ with~$\GSBV(\om)$.

\end{remark}

As stated in Theorem~\ref{thm:dim-reduction} below, the $\Gamma$-limit of~$G_\rho$ turns out to be
\begin{align*}
  G_0(u) \coloneq\displaystyle \half &  \int_{\Om} \stiff^{\alpha \beta \sigma \tau} b_{\alpha \beta} b_{\sigma \tau}
  \abs{u}^2 \sqrt{a} \,\dd x
  \\
  &
  + \frac{\mu}{2} \int_{\Om} a^{\alpha \beta} \partial_{\alpha} u \partial_{\beta} u \sqrt{a} \,\dd x  + \kappa \int_{S_u} \sqrt{[\nu_u]_\alpha a^{\alpha \beta} [\nu_u]_\beta} \, \sqrt{a} \,\dd \HH^2
\end{align*}
where
\begin{equation*}
  \stiff^{\alpha \beta \sigma \tau} \coloneq \frac{2 \lambda \mu}{\lambda + 2\mu} a^{\alpha \beta} a^{\sigma \tau} + \mu \bigl(a^{\alpha \sigma} a^{\beta \tau} + a^{\alpha \tau} a^{\beta \sigma} \bigr).
\end{equation*}

\begin{remark}
  Analogously to \eqref{eq:positive-definite-stiffness}, there exist two constants $c,C>0$, such that
  \begin{equation*}
    c \abs{\mathrm{M}}^{2} \leq c^{\alpha \beta \sigma \tau} \mathrm{M}_{\alpha \beta} \mathrm{M}_{\sigma \tau} \leq C \abs{\mathrm{M}}^{2} \quad\tforall \mathrm{M}\in \R^{2\times 2} \text{ symmetric.}
  \end{equation*}
  This implies that, when $G_0 (u) < \infty$, we have $b_{\alpha \beta} u \in L^2(\Om)$.
\end{remark}

We are now ready to state the result describing the two dimensional model in terms of a $\Gamma$-convergence argument as the thickness $\rho$ of $\Om_\rho$ tends to zero.
\begin{theorem}
  \label{thm:dim-reduction}
  Let $\mathcal \G_\rho\colon L^1(\Om) \to \R$ be defined by
  \begin{equation*}
    \G_\rho (u) = \left\{
      \begin{aligned}
        &G_\rho (u) && \text{for } u\in \GSBV^2 (\Om) \\
        &{+}\infty && \text{otherwise},
      \end{aligned} \right.
  \end{equation*}
  and $\G_0 \colon L^1(\Om) \to \R$ by
  \begin{equation*}
    \G_0 (u) = \left\{
      \begin{aligned}
        & G_0(u) && \text{for } u\in \U  \\
        &{+} \infty && \text{otherwise}.
      \end{aligned} \right.
  \end{equation*}
  Then, $\mathcal G_\rho$ $\Gamma$\nobreakdash-converges to $\mathcal G_0$ w.r.t. the $L^{1}$-topology
  as $\rho \to 0$.
\end{theorem}

\begin{proof}
	The proof follows directly from Proposition~\ref{liminfIneq} and  Proposition~\ref{prop:dim-red-limsup} below.
\end{proof}

In order to prove Proposition 2.3, we further need the next auxiliary lemma.

\begin{lemma}\label{Lemma1}
  Let $\lbrace\rho_\ell\rbrace$, with $\rho_\ell > 0$, be a null sequence.
  Let $u_\ell$ be such that $u_\ell \to u$ in $L^1(\Omega)$ as $\ell \to \infty$ and
  \begin{equation}\label{eq:bound-G}
    \sup_{\ell \in \N} \G_{\rho_\ell} (u_\ell) < \infty.
  \end{equation}
  Then, $u\in \U$ and, up to a subsequence, $\strain_{\alpha \beta, \rho_\ell}(u_\ell) \wto -b_{\alpha \beta} u$  and $\partial_\alpha u_\ell \wto \partial_\alpha u$ weakly in $L^2 (\Omega)$. Furthermore,
  \begin{equation*}
    \lim_{\ell\to\infty} \int_{S_{u_\ell}} | [\nu_{u_\ell}]_3 | \,\dd \HH^{2} =0.
  \end{equation*}
\end{lemma}
\begin{proof}
  Throughout the proof, $C>0$ denotes a generic constant, independent of~$x \in \Om$ and of~$\rho_\ell$.

  Since $\G_{\rho_\ell} (u_\ell)$ is bounded, we have that $u_\ell \in \GSBV^2 (\Om)$. From \eqref{scaledStrain}, we have that, for sufficiently large~$\ell$,
  \begin{align}\label{uniformBound}
      \abs{\nabla u_\ell}^2 & =  \sum_{\alpha} \abs{2  \strain_{\alpha 3, \rho_\ell}(u_\ell)}^2 + \abs{ \rho_\ell  \strain_{3 3, \rho_\ell} (u_\ell) }^2 \\
      & \leq 4 \sum_{i,j} \abs{ \strain_{i j, \rho_\ell}(u_\ell)}^2
      \leq C \stiffness_{\rho_\ell}^{ijkl}  \strain_{ij, \rho_\ell}(u_\ell)  \strain_{kl, \rho_\ell}(u_\ell) \,, \nonumber
  \end{align}
  where the last inequality follows from \eqref{eq:positive-definite-stiffness}. Furthermore, from Proposition~\ref{scaledFuncProp} we infer that
  \begin{equation}
    \label{eq:uniform-lower-bound}
    C \leq [\nu_{u_\ell}]_\alpha g_{\rho_\ell}^{\alpha \beta} [\nu_{u_\ell}]_\beta +  \frac{1}{\rho_\ell^2} \bigl([\nu_{u_\ell}]_3\bigr)^2 \quad \text{and} \quad C \leq g_{\rho_\ell}.
  \end{equation}
  As a consequence, there holds
  \begin{equation*}
    C \biggl( \int_{\Om} \abs{\nabla u_\ell}^2 \,\dd x + \HH^2\bigl( S_{u_\ell} \bigr) \biggr) \leq \G_{\rho_\ell} (u_\ell) \leq \sup_{\ell}\,\G_{\rho_\ell}(u_{\ell})<+\infty.
  \end{equation*}
  Because of the $L^1$\nobreakdash{-}convergence of $u_\ell$, $\norm{u_\ell}_{L^1(\Om)}$ is uniformly bounded. Thus, by compactness properties of $\GSBV^2(\Om)$ (see,  e.g., Theorem~4.36 in \cite{AmbFusPal2000}), there holds $u \in \GSBV^2(\Om)$ and $\nabla u_\ell \wto \nabla u$ weakly in $L^2(\Om;\R^n)$.

  Applying Theorem~5.8 from \cite{AmbFusPal2000}, we obtain
  \begin{equation*}
    \int_{\Om} \abs{\partial_3 u}^2 \,\dd x \leq \liminf_{\ell \to \infty} \int_{\Om} \abs{\partial_3 u_\ell}^2 \,\dd x.
  \end{equation*}
  Hence, using \eqref{scaledStrain} and \eqref{uniformBound}, we have
  \begin{align*}
    \int_{\Om} \abs{\partial_3 u}^2 \,\dd x \leq
    \liminf_{\ell \to \infty} \rho_\ell^2 \int_{\Om} \bigabs{ \strain_{3 3, \rho_\ell} (u_\ell)}^2 \,\dd x
    \leq C \liminf_{\ell \to \infty} \rho_\ell^2 \G_{\rho_\ell} (u_\ell) = 0,
  \end{align*}
  which yields $\partial_3 u = 0$. Now, we show that $[\nu_u]_3= 0$.
  From Theorem~5.22 in \cite{AmbFusPal2000}, this lower semi-continuity property follows:
  For every $\tilde \rho >0$,
  \begin{align}\label{eq:lower-semi-conti-prop}
  \int_{S_u} & \sqrt{[\nu_{u}]_\alpha a^{\alpha \beta} [\nu_u]_\beta + \frac{1}{\tilde \rho^2} \bigabs{[\nu_u]_3}^2 } \sqrt{a} \,\dd \HH^2  \\
  & \leq \displaystyle \liminf_{\ell \to \infty} \int_{S_{u_\ell}} \sqrt{ [\nu_{u_\ell}]_\alpha a^{\alpha \beta} [\nu_{u_\ell}]_\beta + \frac{1}{\tilde \rho^2} \bigabs{[\nu_{u_\ell}]_3}^2 } \sqrt{a} \,\dd \HH^2 \,.  \nonumber
  \end{align}
  This yields that, for every $\tilde \rho >0$,
  \begin{equation}\label{eq:nu3-bound-1}
    \frac{1}{\tilde \rho}  \int_{S_{u}} \bigabs{[\nu_u]_3} \sqrt{a} \,\dd \HH^2
    \leq  \liminf_{\ell \to \infty} \int_{S_{u_\ell}} \sqrt{[\nu_{u_\ell}]_\alpha a^{\alpha \beta} [\nu_{u_\ell}]_\beta + \frac{1}{\tilde \rho^2} \bigabs{[\nu_{u_\ell}]_3}^2 } \sqrt{a} \,\dd \HH^2.
  \end{equation}
  From  Proposition~\ref{scaledFuncProp},~\eqref{eq:bound-G}, and \eqref{eq:uniform-lower-bound}, for sufficiently large $\ell$ we deduce
  \begin{align}
  \label{eq:fracture-estimate}
  \int_{S_{u_\ell}}&  \sqrt{[\nu_{u_\ell}]_\alpha a^{\alpha \beta} [\nu_{u_\ell}]_\beta + \frac{1}{\tilde \rho^2} \bigabs{[\nu_{u_\ell}]_3}^2 } \sqrt{a} \, \dd \HH^{2} \\
      & \leq \int_{S_{u_\ell}} \sqrt{[\nu_{u_\ell}]_\alpha g_{\rho_\ell}^{\alpha \beta} [\nu_{u_\ell}]_\beta + \frac{1}{\rho_\ell^2} \bigabs{[\nu_{u_\ell}]_3}^2 } \sqrt{g_{\rho_\ell}} \, \dd \HH^{2} + C \sqrt{\rho_\ell} + C \rho_\ell  \nonumber \\
      &\vphantom{\int} \leq \G_{\rho_\ell} (u_\ell) + C \sqrt{\rho_\ell}  + C \rho_\ell. \nonumber
  \end{align}
  By assumption~\eqref{eq:bound-G}, the right-hand side of~\eqref{eq:fracture-estimate} turns out to be uniformly bounded. Thus, combining~\eqref{eq:nu3-bound-1} and~\eqref{eq:fracture-estimate}, we derive
  \begin{equation*}
    \int_{S_{u}} \bigabs{[\nu_u]_3} \sqrt{a} \,\dd \HH^2 \leq C \sqrt{\tilde \rho} \quad \text{for every }
    \tilde \rho > 0.
  \end{equation*}
  The previous inequality implies that $[\nu_u]_3 = 0$ on~$S_u$, so that $u\in \U$.

  As in \eqref{uniformBound}, we obtain that $\norm{\partial_\alpha u_\ell}_{L^2(\Om)}$ and $\norm{ \strain_{\alpha \beta, \rho} (u_\ell)}_{L^2(\Om)}$ are uniformly bounded. Then, the weak convergence follows from
\eqref{scaledStrain} and \eqref{scaledChristoffel}.
\end{proof}

We now prove the $\liminf$\nobreakdash{-}inequality.
\begin{proposition}\label{liminfIneq}
  Under the same hypotheses as in Theorem~\ref{thm:dim-reduction}, there holds $\G_0 \leq \Gliminf_{\rho\to 0} \G_{\rho}$.
\end{proposition}
\begin{proof}
  Let $\lbrace\rho_\ell\rbrace$, with $\rho_\ell > 0$, be a null sequence, and let $u_\ell$ be a sequence converging in $L^1 (\Omega)$ to $u \in L^1 (\Omega)$. Without loss of generality, we can assume that $\liminf_{\ell \to \infty} \G_{\rho_\ell} (u_\ell) = \lim_{\ell \to \infty} \G_{\rho_\ell} (u_\ell) < \infty$.
From Lemma~\ref{Lemma1}, it follows that $u \in \U$.

After some algebraic manipulations, we have
  \begin{equation}\label{completedSquares}
    \G_{\rho_\ell} (u_\ell)
    = I^{(1)}_{\rho_\ell} (u_\ell)  + I^{(2)}_{\rho_\ell} (u_\ell) + I^{(3)}_{\rho_\ell} (u_\ell) + I^{(4)}_{\rho_\ell} (u_\ell)
  \end{equation}
  with
  \begin{align*}
 	 &\begin{aligned}
    I^{(1)}_{\rho_\ell} (u_\ell) \coloneq \half\int_{\Omega} \biggl(& \frac{2\lambda\mu}{\lambda + 2\mu} g_{\rho_\ell}^{\alpha \beta} g_{\rho_\ell}^{\sigma \tau} +\mu (g_{\rho_\ell}^{\alpha \sigma} g_{\rho_\ell}^{\beta \tau} + g_{\rho_\ell}^{\alpha \tau} g_{\rho_\ell}^{\beta \sigma}) \biggr)  \\
    & \vphantom{\int} \times \strain_{\alpha \beta, \rho_\ell} (u_\ell)\strain_{\sigma \tau, \rho_\ell} (u_\ell) \sqrt{g_{\rho_\ell}} \,\dd x
    \end{aligned}
    \\[1mm]
    &I^{(2)}_{\rho_\ell} (u_\ell) \coloneq \half\int_{\Omega} ( \lambda + 2\mu) \biggl( \frac{\lambda}{\lambda + 2 \mu} g^{\alpha \beta}_{\rho_\ell}  \strain_{\alpha \beta, \rho_\ell} (u_\ell) +  \strain_{3 3, \rho_\ell} (u_\ell) \biggr)^2 \sqrt{g_{\rho_\ell}} \,\dd x
    \\[1mm]
    &I^{(3)}_{\rho_\ell} (u_\ell)\coloneq 2 \mu \int_{\Omega}  g_{\rho_\ell}^{\alpha \beta}  \strain_{\alpha 3, \rho_\ell} (u_\ell)  \strain_{\beta 3, \rho_\ell} (u_\ell) \sqrt{g_{\rho_\ell}} \,\dd x
    \\[1mm]
    &I^{(4)}_{\rho_\ell} (u_\ell)\coloneq \kappa \int_{S_{u_\ell}}  \sqrt{[\nu_{u_\ell}]_\alpha g_{\rho_\ell}^{\alpha \beta} [\nu_{u_\ell}]_\beta +  \frac{1}{\rho_\ell^2} \bigabs{[\nu_{u_\ell}]_3}^2 } \sqrt{g_{\rho_\ell}} \,\dd \HH^2.
  \end{align*}
  We now prove the $\liminf$-inequality for $I^{(1)}_{\rho_\ell} $, $I^{(3)}_{\rho_\ell}$ and $I^{(4)}_{\rho_\ell}$, whereas the term~$I^{(2)}_{\rho_\ell}$ need not be estimated, being non-negative.

  From pointwise convergence (up to a subsequence) of $u_\ell$ almost everywhere and from Proposition~\ref{scaledFuncProp} we derive the pointwise convergence of the integrand of~$I^{(1)}_{\rho_\ell} (u_\ell)$. Hence, by Fatou lemma, we obtain
  \begin{equation}\label{eq:liminf-1}
    \half \int_{\Om} \stiff^{\alpha \beta \sigma \tau} b_{\alpha \beta} b_{\sigma \tau} \abs{u}^2 \sqrt{a} \,\dd x
    \leq \liminf_{\ell \to \infty} I^{(1)}_{\rho_\ell} (u_\ell).
  \end{equation}

  In view of~\eqref{eq:metric-norm-equivalence}, the map $v \mapsto \int_{\omega} a^{\alpha \beta} v_\alpha v_\beta \sqrt{a} \,\dd x$ is a norm in~$L^2 (\Om)$ and is therefore weakly lower semi-continuous in~$L^{2}(\Om)$. Hence, using the weak convergence $\strain_{\alpha 3, \rho_\ell} (u_\ell) \wto \half \partial_{\alpha} u$ in $L^2(\Om)$ proved in Lemma~\ref{Lemma1}, we obtain
  \begin{equation*}
    \frac{1}{4} \int_{\Om} a^{\alpha \beta} \partial_{\alpha} u  \partial_{\beta} u \sqrt{a} \,\dd x
    \leq \liminf_{\ell \to \infty} \int_{\Om} a^{\alpha \beta}  \strain_{\alpha 3, \rho_\ell} (u_\ell)  \strain_{\beta 3, \rho_\ell} (u_\ell)  \sqrt{a} \,\dd x.
  \end{equation*}
  From~\eqref{eq:metric-norm-equivalence} and from Proposition~\ref{scaledFuncProp}, for sufficiently large $\ell$ it holds
  \begin{align*}
    \int_{\Om} & a^{\alpha \beta}  \strain_{\alpha 3, \rho_\ell} (u_\ell)  \strain_{\beta 3, \rho_\ell} (u_\ell)  \sqrt{a} \,\dd x \\
   & \leq \displaystyle\int_{\Om} g_{\rho_\ell}^{\alpha \beta}  \strain_{\alpha 3, \rho_\ell} (u_\ell)  \strain_{\beta 3, \rho_\ell} (u_\ell)  \sqrt{g_{\rho_\ell}} \,\dd x + C \sqrt{\rho_{\ell}} \sum_{\alpha} \| \strain_{\alpha 3, \rho_\ell} (u_\ell)\|_{L^2(\Om)}^{2} \,,
  \end{align*}
  namely,
  \begin{equation}\label{eq:liminf-3}
    \frac{\mu}{2} \int_{\omega} a^{\alpha \beta} \partial_{\alpha} u \cdot \partial_{\beta} u \sqrt{a} \,\dd x \leq \liminf_{\ell \to \infty} I^{(3)}_{\rho_\ell} (u_\ell).
  \end{equation}
    Proceeding as in~\eqref{eq:lower-semi-conti-prop}--\eqref{eq:fracture-estimate}, for every $\tilde \rho >0$,
we have that
  \begin{align}
  \label{eq:liminf-4}
  \kappa\int_{S_u} & \sqrt{[\nu_u]_\alpha a^{\alpha \beta} [\nu_u]_\beta } \sqrt{a} \,\dd \HH^1 \\
      &\leq \liminf_{\ell \to \infty} \kappa \int_{S_{u_\ell}} \sqrt{[\nu_{u_\ell}]_\alpha a^{\alpha \beta} [ \nu_{u_\ell} ]_\beta  + \frac{1}{\tilde \rho^2} \bigabs{[ \nu_{u_\ell}]_3}^2 } \sqrt{a} \,\dd \HH^2 \nonumber \\
      &\leq \liminf_{\ell \to \infty} \kappa \int_{S_{u_\ell}} \sqrt{[\nu_{u_\ell}]_\alpha g_{\rho_\ell}^{\alpha \beta} [\nu_{u_\ell}]_\beta  + \frac{1}{\rho_\ell^2} \bigabs{[\nu_{u_\ell}]_3}^2 } \sqrt{g}\,\dd\HH^2 + C \sqrt{\rho_\ell} \nonumber \\
      &\vphantom{\int} = \liminf_{\ell \to \infty} I^{(4)}_{\rho_\ell} (u_\ell)\,. \nonumber
  \end{align}

  Summing up \eqref{eq:liminf-1}--\eqref{eq:liminf-4} and using that $I^{(2)}_{\rho_\ell}$ is non-negative, we deduce that
  \begin{align*}
    \G_0 (u) \leq{}& \liminf_{\ell \to \infty} \G_{\rho_\ell} (u_\ell),
  \end{align*}
  which concludes the proof.
\end{proof}

In the next proposition we prove the $\limsup$\nobreakdash{-}inequality.
\begin{proposition}\label{prop:dim-red-limsup}
  Under the same hypotheses as in Theorem~\ref{thm:dim-reduction}, there holds $\Glimsup_{\rho\to 0} \G_\rho \leq \G_0$.
\end{proposition}
\begin{proof}
  Let $\lbrace \rho_\ell \rbrace$, with $\rho_\ell > 0$, be a sequence such that $\rho_\ell \to 0$ as $\ell \to \infty$. We can assume that $\G_0 (u) < +\infty$ and thus $u\in \U$, otherwise, from Proposition~\ref{liminfIneq},
we have that $\liminf_{\ell \to \infty} \G_{\rho_\ell} (u) = {+} \infty$ and there is nothing to prove.
Moreover, setting $u^P\coloneq (- P) \vee u \wedge P$ for $P>0$, we clearly have that $u^P \to u$ in~$L^{1}(\Om)$ and $\G_0(u^P) \to \G_0(u)$ for $P \to +\infty$.
Therefore, we may just consider $u\in \SBV^{2}(\Om)\cap L^{\infty}(\Om)$.

  We pick the sequence $u_{\ell}$ in~$\SBV^2(\Om) \cap L^{\infty}(\Om)$ defined for all $\ell \in \N$ by
  \begin{equation*}
    u_{\ell} (x) = u(x_1, x_2) \exp\biggl( \frac{\lambda}{\lambda + 2 \mu} \, a^{\alpha \beta} b_{\alpha \beta} \,\rho_\ell \, x_3 \biggr)  \quad \text{for $x = (x_1, x_2, x_3) \in \Om$}.
  \end{equation*}
  It turns out that $u_\ell \to u$ in $L^1(\Om)$ as $\ell \to \infty$ and that $u_\ell$ is bounded in  $L^{\infty}(\Om)$. Starting from \eqref{completedSquares}, we show that each term~$I^{(k)}_{\rho_\ell} (u_\ell)$ (for $k=1,2,3,4$) converges as expected.

  Since all the functions involved in the exponential are uniformly bounded, it holds $\abs{u_\ell} \leq C \abs{u}$ for some constant $C>0$. Moreover, we deduce from~\eqref{scaledChristoffel} in Proposition~\ref{scaledFuncProp} that
  \begin{equation}\label{eq:limsup-strain-bound}
    \bigabs{ \strain_{\alpha \beta, \rho_\ell} (u_\ell)} =  \abs{\Lambda_{\alpha \beta ,\rho_\ell}^3 u_\ell} \leq C \abs{b_{\alpha \beta} u} + C \rho_\ell \abs{u}.
  \end{equation}
  Since $u, b_{\alpha \beta} \in L^\infty (\Om)$, the right-hand side of \eqref{eq:limsup-strain-bound} is bounded, and hence, in view of \eqref{scaledStrain} and Proposition~\ref{scaledFuncProp},~$\strain_{\alpha \beta, \rho_\ell} (u_\ell) \to -b_{\alpha\beta} u$ in $L^{2}(\Om)$.
  From~\eqref{eq:positive-definite-stiffness} (replacing~$\lambda$ with $\frac{2\lambda\mu}{\lambda + 2\mu}$),
we infer that there exists a constant $C > 0$ such that
  \begin{multline*}
    \biggl(\frac{2\lambda\mu}{\lambda + 2\mu} g_{\rho_\ell}^{\alpha \beta} g_{\rho_\ell}^{\sigma \tau} +\mu (g_{\rho_\ell}^{\alpha \sigma} g_{\rho_\ell}^{\beta \tau} + g_{\rho_\ell}^{\alpha \tau} g_{\rho_\ell}^{\beta \sigma}) \biggr)  \strain_{\alpha \beta, \rho_\ell} (u_\ell)  \strain_{\sigma \tau, \rho_\ell} (u_\ell) \leq C  \sum_{\alpha,\beta} \bigabs{  \strain_{\alpha \beta, \rho_\ell} (u_\ell)}^2.
  \end{multline*}
  Therefore, by the dominated convergence theorem, it follows that
  \begin{equation}\label{eq:limsup-1}
    \lim_{\ell \to \infty} I^{(1)}_{\rho_\ell} (u_\ell)= \half \int_\Om \stiff^{\alpha \beta \sigma \tau} b_{\alpha \beta} b_{\sigma \tau} \abs{u}^2 \sqrt{a} \, \dd x \,.
  \end{equation}

  Moving to the term $I^{(3)}_{\rho_\ell}$, we have that $\bigabs{ \strain_{\alpha 3}(u_\ell)}  \leq C \abs{\partial_\alpha u} \in L^2(\Om)$, so that, using~\eqref{eq:metric-tensor-bound}, we deduce that
  \begin{equation}\label{eq:limsup-3}
    \lim_{\ell \to \infty} I^{(3)}_{\rho_\ell}(u_\ell) = \frac{\mu}{2} \int_{\Om} a^{\alpha \beta} \partial_\alpha u \partial_\beta u \sqrt{a}\,\dd x.
  \end{equation}
Since it holds that
  \begin{equation*}
    S_{u_\ell} = S_u \,, \quad [\nu_{ u_\ell}]_3 = 0\,, \quad \text{and} \quad  [\nu_{u_\ell}]_\alpha =
    [\nu_u]_\alpha \,,
  \end{equation*}
  and thanks to Proposition~\ref{scaledFuncProp}, we obtain
  \begin{align}\label{eq:limsup-4}
        \lim_{\ell \to \infty}  I^{(4)}_{\rho_\ell}  &= \lim_{\ell\to\infty} \kappa \int_{S_u} \sqrt{[\nu_u]_\alpha g_{\rho_\ell} [\nu_u]_\beta } \sqrt{g_{\rho_\ell}} \,\dd \HH^2
        \\
        &
        = \kappa \int_{S_u}  \sqrt{[\nu_u]_\alpha a^{\alpha \beta} [\nu_u]_\beta } \sqrt{a} \,\dd \HH^2 . \nonumber
  \end{align}
Finally, we show that $I^{(2)}_{\rho_\ell} (u_\ell) \to 0$. With this aim, we note that
  \begin{equation*}
    \strain_{33, \rho_\ell} (u_\ell) = \frac{\lambda}{\lambda + 2 \mu} \, a^{\alpha \beta}  b_{\alpha \beta} \, u_\ell
  \end{equation*}
  and, therefore, by Proposition~\ref{scaledFuncProp} we have
  \begin{align*}
    \bigabs{g^{\alpha \beta}_{\rho_\ell}  \strain_{\alpha \beta, \rho_\ell} (u_\ell) + a^{\alpha \beta} b_{\alpha \beta} u_\ell } &\leq C\rho_\ell \abs{u_\ell}.
  \end{align*}
  Exploiting the fact that~$u_\ell \in L^{\infty}(\Om)$ and the uniformly bound
of~$\sqrt{g_{\rho_\ell}}$, we deduce
  \begin{equation*}
    \bigabs{I^{(2)}_{\rho_\ell} (u_\ell)} \leq C \rho_\ell^2 \norm{ u }_{L^\infty(\Om)}^{2},
  \end{equation*}
  which implies that $I^{(2)}_{\rho_\ell} (u_\ell) \to 0$. Eventually, this inequality, together with \eqref{eq:limsup-1}\nobreakdash--\eqref{eq:limsup-4}, implies that $\lim_{\ell \to \infty} \G_{\rho_\ell} (u_\ell) = \G_0 (u)$, which concludes the proof.
\end{proof}

We point out that the limit functional~$\G_0$ (or~$G_0$) is actually two dimensional.
Since the integrands do not depend on~$x_3$, as explained in Remark~\ref{r-nox3}, we can simply replace~$\Om$ with the two-dimensional domain~$\om$. Hence, for $u\in \GSBV^{2}(\om)$ we have
  \begin{align*}
    G_0(u) =\displaystyle \half & \int_{\om} \stiff^{\alpha \beta \sigma \tau} b_{\alpha \beta} b_{\sigma \tau} \abs{u}^2 \sqrt{a} \,\dd x + \frac{\mu}{2} \int_{\om} a^{\alpha \beta} \partial_{\alpha} u \partial_{\beta} u
  \sqrt{a} \,\dd x \\
    & + \displaystyle \kappa \int_{S_u} \sqrt{[\nu_u]_\alpha a^{\alpha \beta} [\nu_u]_\beta} \, \sqrt{a} \,\dd \HH^1\,.
  \end{align*}
Introducing the notation
\begin{equation*}
  b \coloneq \stiff^{\alpha \beta \sigma \tau} b_{\alpha \beta} b_{\sigma \tau} \sqrt{a} \quad \text{and} \quad A \coloneq (a^{\alpha \beta}) \sqrt{a},
\end{equation*}
we can rewrite $G_0$ as
\begin{equation*}
  G_0(u) = \half \int_{\om} b \abs{u}^2 \,\dd x + \frac{\mu}{2} \int_{\om} \nabla u^\top A \nabla u \,\dd x
  + \kappa \int_{S_u} \sqrt{\nu_u^\top A \nu_u \sqrt{a}} \,\dd \HH^1.
\end{equation*}
Notice that, due to \eqref{eq:metric-norm-equivalence}, the symmetric matrix $A(x) \in \R^{2\times 2}$ is positive definite, uniformly w.r.t.~$x\in \om$, i.e., there exist $0 < \alpha\leq \beta< +\infty$ such that
\begin{equation*}
 \alpha \abs{\zeta}^{2} \leq A(x) \zeta \cdot\zeta \leq \beta \abs{\zeta}^2 \qquad \text{for every } x\in\om \text{ and every } \zeta\in\R^2 \,.
\end{equation*}


%% file: numerics/numerics.tex
\subsection{The Regularized Reduced Model}\label{cha:numerics}
The numerical minimization of the functional $G_0$ can be tackled via phase-field models (see, e.g., \cite{ArtForMicPer2015,Bou2007a,BouFraMar2000,BurOrtSue2010}).
The seminal idea can be ascribed to \cite{AmbTor1990,AmbTor1992}, where the authors introduce an additional smooth variable, the phase field, which describes the fracture set.
The results of \cite{AmbTor1990,AmbTor1992} have been generalized in many ways,\cite{BelBre2019,Bra1998,DalIur2013,Foc2001,Iur2013} including the case of vector displacements.\cite{ChaCri2019} In our setting, we need a slightly more general result compared with \cite{Foc2001}, as we have to take into account the spatial dependence of~$A$ in the phase-field term. The $\Gamma$-convergence result is stated in Theorem~\ref{thm:h1-phase-field} below, whose proof is provided in the Appendix.

\begin{theorem}\label{thm:h1-phase-field}
  For $\eps > 0$, let $\eta_\eps > 0$ be such that $\eta_\eps/\eps \to 0$ as $\eps \to 0$.
Define the family of functionals $\lbrace\F_\eps\rbrace_{\eps > 0}$,
with $\F_\eps \colon L^1(\om) \times L^1 (\om) \to \R$ such that
  \begin{align}
  \label{eq:h1-approx-functional}
    \F_\eps (u,v) \coloneq\half &  \int_\om b \abs{u}^2 \,\dd x + \frac{\mu}{2} \int_\om (v^2 + \eta_\eps) \nabla u^\top A \nabla u \,\dd x \\
    & +\kappa \int_\om \bigg[\frac{1}{4\eps}  (1 - v)^2 \sqrt{a} +  \eps \nabla v^\top A \nabla v\bigg] \,\dd x \,,\nonumber
  \end{align}
  for all $u\in H^1(\om), v \in H^1 (\om; [0,1]) $ and $\F_\eps (u,v) \coloneq +\infty$ otherwise.
  Then $\G_0 = \Glim_{\eps \to 0} \F_\eps$ in the $L^{1}$-topology.
\end{theorem}

\begin{proof}
	See~\ref{appendix}.
\end{proof}
We remark that, loosely speaking, for small~$\eps$, the phase field minimizing~$\F_\eps$
is close to zero where the gradient of the displacement~$u$ is large, whereas
it approaches~1 elsewhere. This implies that the material is sound where~$v$ is close to~$1$, whereas
a fracture is detected where $v\ll 1$.
In particular, the third integral in~\eqref{eq:h1-approx-functional} converges to the length of the crack set.

With a view to the numerical approximation of the functional $\F_\eps$, for small~$\eps > 0$,
we restrict the function space to $H^1(\om) \times H^1(\om;[0,1])$, and omit the subscript~$\eps$,
as it will be fixed in the numerical test cases.
Moreover, for all $u \in H^1(\om)$, $v\in H^1(\om;[0,1])$, we introduce the stored elastic energy
\begin{equation}
\label{eq:elastic_energy}
  \E(u,v) \coloneq \half \int_\om b  \abs{u}^2 \,\dd x +  \frac{\mu}{2}\int_{\om} \bigl( v^{2} + \eta_{\eps} \bigr)\, \nabla u^\top A \nabla u \,\di x
\end{equation}
and the dissipation potential
\begin{equation}
\label{eq:dissipated_energy}
  \DD(v) \coloneq \kappa \int_\om \bigg[\frac{1}{4\eps}  (1 - v)^2 \sqrt{a} +  \eps \nabla v^\top A \nabla v\bigg]\,\dd x \,,
\end{equation}
so that
\begin{equation}
\label{eq:F}
  \F(u,v) = \E(u,v) + \DD(v)\,.
\end{equation}
Note that $\F (u,v)$ is Fréchet-differentiable in $H^1(\om) \times \big[H^1(\om) \cap L^\infty(\om)\big]$
(see, e.g., Proposition~1.1 in \cite{BurOrtSue2010}), with
\begin{align*}
  \partial_u \F (u,v)[\varphi] &= \int_\om b u \varphi \, \dd x + \mu \int_\om \bigl( v^2 + \eta_\eps \bigr)  \nabla u^\top A \nabla \varphi \,\dd x \,, \\
\partial_v \F (u,v)[\psi] &=\mu\int_\om v \psi \nabla u^\top A \nabla u \, \dd x  +  \kappa \int_{\om} \bigg[\frac{1}{2 \eps}  (v-1) \psi  \sqrt{a} + 2 \eps \, \nabla v^\top A \nabla \psi\bigg] \,\dd x \,,
\end{align*}
for all $u,\varphi \in H^1(\om)$, $v,\psi \in H^1(\om) \cap L^\infty(\om)$.

\input{numerics/discretization}
\input{numerics/residualestimate}
\input{numerics/meshconstruction}
\input{numerics/computation}


%% file: numerics/discretization.tex
\section{The Discrete Setting: a Finite Element Approximation}\label{sec:discrete-setting}
Let $\omega \subset \R^2$ be a polygonal domain, and let $\{ \T_h \}_{h>0}$ be a family of triangulations of~$\omega$. For every $h>0$, we denote by~$T$ a generic element of~$\T_h$ and we set $h_T:=\mathrm{diam} (T)$, where $h = \max_{T \in \T_h} h_T$.
Furthermore, we denote by~${\cal V}_h$ the set of all the vertices of~$\T_{h}$ and define
$N_{h}\coloneq \#{\cal V}_{h}$.

The discretization is cast in the space
\begin{displaymath}
\UU_h \coloneq \bigl\{ u \in H^1(\om) \colon u\rvert_T \in \mathbb P_1(T),\, \text{for every } T\in \T_h \}\,,
\end{displaymath}
of piecewise continuous linear finite elements, whose Lagrangian basis is denoted by~$\{\xi_{l}\}_{l=1}^{N_{h}}$.
We assume that this basis satisfies the non-positivity condition
\begin{equation}\label{stiffness}
  \int_{\om}\nabla \xi_{l}^\top A \nabla \xi_{m} \,\di x \leq 0 \quad \forall\, l,m\in\{1,\dotsc,N_{h}\},\ l\neq m.
\end{equation}
For the particular choice $A=I$, with $I$ the identity matrix, this condition is satisfied when $\T_h$ is an acute-angle mesh, and it ensures a discrete maximum principle in~$\UU_h$ (see \cite{CiaRav1973,StrFix2008}), i.e., that the phase field takes values in $[0,1]$ along the evolution
(cf. Proposition~6.14 in \cite{AlmBelNeg2019}).
In the present context, the matrix~$A$ corresponds to a metric tensor of a Riemannian manifold multiplied by a positive function. Thus, by coordinate transformation, condition \eqref{stiffness} is fulfilled if the triangulation is acute in the Riemannian space.
Indeed, according to the notation of Section~\ref{sec:two-dimensional-model}, the tangential gradient is
  \begin{equation*}
    \nabla_\tau \hat u \coloneq \bigl(\nabla \tilde u - \scprod{\nabla \tilde u}{g^3} g^3 \bigr) \bigr\rvert_{S} \quad \forall \hat u\in C^1(\phi(\om)),
  \end{equation*}
  where $\tilde u$ is an extension of $\hat u$ to $\Phi(\Om_\rho)$, which is characterized by a thickness~$\rho$. Then, by coordinate transformation, \eqref{stiffness} is equivalent to
  \begin{equation*}
    \int_{\phi(\omega)} \nabla_\tau (\xi_l\circ \phi^{-1})^\top \nabla_\tau (\xi_m \circ \phi^{-1}) \, \dd x \leq 0 \quad \forall l,m\in\{ 1,\dotsc, N_h \},\ l \neq m.
  \end{equation*}

In general, the space $\UU_h$ is endowed with the norm on $H^1(\om)$. However, we also adopt the norm
\begin{equation*}
  \norm{v}_{\UU_h} = \biggl( \int_\om \bigabs{\Pi_h(v^2)} \, \dd x \biggr)^\half \quad \tforall v \in \UU_h \,,
\end{equation*}
where $\Pi_h$ denotes the Lagrangian interpolant associated with the space $\UU_h$.

We introduce now the discrete counterpart of the elastic energy~\eqref{eq:elastic_energy} and of the dissipation potential~\eqref{eq:dissipated_energy}: for every $u, v \in \UU_{h}$, $0\leq v\leq 1$, let
\begin{align*}
  \E_{h}(u,v) &\coloneq \half \int_\om b  \abs{u}^2 \,\dd x +  \frac{\mu}{2}\int_{\om} \bigl( \p_{h}(v^{2}) + \eta_\eps \bigr)\, \nabla u^\top A \nabla u \,\di x \,, \\
  \DD_{h}(v) &\coloneq\kappa \int_\om \bigg[\frac{1}{4\eps}  \Pi_h\bigl( (1 - v)^2 \bigr) \sqrt{a} +  \eps \nabla v^\top A \nabla v \bigg] \,\dd x \,,
\end{align*}
which leads to the definition of the discrete phase field energy \eqref{eq:F} by
\begin{equation*}
  \F_{h} (u,v) \coloneq \E_h (u,v) + \DD_h (v) \quad \text{for } u, v\in \UU_h,\ 0\leq v\leq 1\,.
\end{equation*}
It holds that $\F_h (u,v)$ is Fr\'echet differentiable with
\begin{align*}
  \partial_u \F_h (u,v)[\varphi] ={}& \displaystyle \int_\om b u \varphi \, \dd x + \mu \int_\om \bigl( \Pi_h (v^2) + \eta_\eps \bigr)  \nabla u^\top A \nabla \varphi \,\dd x \,,\\
  \partial_v \F_h (u,v)[\psi] ={}& \displaystyle \mu\int_\om \Pi_h (v \psi) \nabla u^\top A \nabla u \, \dd x\\
  &+ \kappa \displaystyle\int_{\om} \bigg[\frac{1}{2 \eps} \Pi_h \bigl( (v-1) \psi \bigr) \sqrt{a} + 2 \eps \,
  \nabla v^\top A \nabla \psi \bigg] \,\dd x \,,
\end{align*}
for all $u, v, \varphi, \psi \in \UU_h$.

\begin{remark}\label{rmrk:use-of-interpolation}
  In general, the energy functional~$\F$  is discretized via restriction to the finite element space, i.e., by setting $\F_h \coloneq \F\rvert_{ \UU_h \times \UU_h}$. Here, following \cite{AlmBel2019,ArtForMicPer2015}, we define~$\F_h$ using the operator~$\Pi_h$. This ensures that also the discrete phase field takes
values in $[0,1]$ (see Proposition~6.14 in \cite{AlmBelNeg2019}).
\end{remark}

\subsection{An Alternating Minimization Scheme}
In order to approximate a quasi-static fracture evolution, we adopt here the scheme used in \cite{ArtForMicPer2015,Bou2007a,BouFraMar2000,BurOrtSue2010}, which is based on an alternating minimization procedure.
For a given time interval, $[0,T_f]$, with $T_f>0$, we consider the time step $\tau = \frac{T_f}{k}$, where $k\in \N$ is the number of time steps, and we denote the time levels by $t_i \coloneq i \tau$ for $i\in \{0,\dotsc, k\}$. Let $g$ be the time dependent Dirichlet boundary condition for the displacement field, assumed to be an absolutely continuous function in $AC ([0, T_f]; W^{1,p}(\om))$, with $p>2$.
The adopted alternating minimization scheme works as follows:
Let $u_0, v_0 \in \UU_h$ the assigned initial values. Then, for every $i\in\{1, \ldots, k\}$ and every~$j \in \N$, we inductively set $u_{i,0} \coloneq u_{i-1}$, $v_{i,0} \coloneq v_{i-1}$ and
\begin{align}\label{eq:sim-minu}
  u_{i,j} &\coloneq \argmin \, \bigl\{ \E_h (u, v_{i,j-1}):  u\in \UU_h,  u= g (t_{i}) \text{ on } \partial\om \bigr\}\,, \\
  \label{eq:sim-minv}
  v_{i,j} &\coloneq \argmin \, \biggl\{\F_h (u_{i,j} ,v) + \frac{\alpha}{2\tau} \norm{v-v_{i-1}}^2_{\UU_h} : v\in \UU_h, v \le v_{i-1} \biggr\} \,,
\end{align}
where $\alpha > 0$ is a tuning parameter.
As shown in Proposition~\ref{prop:critical-point}, there exists a subsequence~$j_m$ such that~$(u_{i,j_m}, v_{i,j_m})$ admits a limit in~$\UU_{h}\times\UU_h$ as $m\to \infty$. Thus, we set
\begin{equation*}
 u_i \coloneq \lim_{m\to \infty} u_{i,j_m} \quad \text{and} \quad v_i \coloneq \lim_{m\to \infty} v_{i,j_m} \,.
\end{equation*}

The inequality constraint in \eqref{eq:sim-minv} enforces the irreversibility of the fracture.
In this way, the phase field is constrained to decrease in time to avoid any crack healing.
Moreover, the constraint $v\geq 0$ is no longer required, since the adopted discretization automatically guarantees
$v_{i,j} \geq 0$ (see also Remark~\ref{rmrk:use-of-interpolation}).

Following Theorems~4.3,~5.13,~5.17 in \cite{AlmBelNeg2019}, we can show that, in the time continuous limit, the algorithm~\eqref{eq:sim-minu}--\eqref{eq:sim-minv} detects a unilateral $L^2$-gradient flow for the functional~$\F_h$.
Moreover, we obtain full consistency when $h \to 0$, namely, a sequence of $L^2$-gradient flows
of~$\F_h$ converge to an $L^2$-gradient flow of~$\F$.

As for the additional parameter $\alpha$, we assume that it is very small, so that a gradient flow of~$\F_h$ is expected to be close to a quasi-static evolution along critical points (see \cite{KneRosZan2013,Neg2019}). The choice $\alpha = 0$, made in \cite{AlmBel2019}, in order to directly obtain a quasi-static evolution, does not ensure an energy balance when $h\to 0$.

  Since $u \mapsto \F_h (u,v)$ is a convex map, the minimization~\eqref{eq:sim-minu} is equivalent to
  \begin{equation}\label{eq:boh1}
    \partial_u \E_h (u_{i,j},v_{i,j-1})[\varphi] = 0 \quad \text{for every } \varphi \in \UU_h, \mbox{ with } \varphi = 0  \mbox{ on } \partial\om.
  \end{equation}
  The minimization \eqref{eq:sim-minv}, instead, is equivalent to the variational inequality (cf. Chapter~3 of \cite{KikOde1988})
  \begin{equation}\label{eq:crit-point-v-2}
    \partial_v \F_h(u_{i,j},v_{i,j}) [v_{i,j} - \psi] + \frac{\alpha}{\tau} \int_\om \Pi_h \bigl( (v_{i,j} - v_{i-1}) (v_{i,j} - \psi) \bigr) \, \dd x \leq 0
  \end{equation}
  for all $ \psi \in \VV_h$, with $\psi \leq v_{i-1}$.

These remarks justify the following definition of a critical point of $\F_h$, subject to the inequality constraint in \eqref{eq:sim-minv}.
\begin{definition}\label{def:discrete-crit-point}
  Let $u, v, \tilde v \in \VV_h$ with $0\leq \tilde v \leq 1$. We define $(u,v)$ as a \emph{discrete critical point with bound~$\tilde v$} if the following two conditions hold
  \begin{align}\label{eq:disc-crit-point-u}
    0 &= \partial_u \E_h (u,v)[\varphi]  \,,\\
    \label{eq:disc-crit-point-v}
    0 &\geq \partial_v \F_h(u,v) [v - \psi] + \frac{\alpha}{\tau} \int_\om \Pi_h \bigl( (v- \tilde v) (v - \psi) \bigr) \, \dd x  \,,
  \end{align}
  for all $\varphi, \psi \in \UU_h$ with $\varphi = 0$ on $\partial \om$ and $\psi \leq \tilde v$.
\end{definition}
Notice that, relations \eqref{eq:disc-crit-point-u}--\eqref{eq:disc-crit-point-v} are equivalent to the single inequality
\begin{equation*}
    \partial_u \E_h (u,v)[\varphi] + \partial_v \F_h(u,v) [v - \psi] + \frac{\alpha}{\tau} \int_\om \Pi_h   \bigl( (v- \tilde v) (v - \psi) \bigr) \, \dd x \leq 0,
  \end{equation*}
  for all $\varphi \in \UU_h$ with $\varphi = 0$ on $\partial \om$ and for all $\psi \in \VV_h$ with $\psi \leq \tilde v$.

We will also employ the continuous counterpart of Definition~\ref{def:discrete-crit-point}:
\begin{definition}\label{def:cont-crit-point}
  Let $u \in H^1(\om)$ and $v, \tilde v \in H^1(\om;[0,1])$. We define $(u,v)$ as a \emph{critical point with bound~$\tilde v$} if the following two conditions hold
  \begin{align*}
    &\partial_u \E (u,v)[\varphi] = 0 \quad \forall \varphi \in H^1(\om), \mbox{ with }\varphi = 0 \text{ on } \partial \om, \\
    &\partial_v \F (u,v) [v - \psi] + \frac{\alpha}{\tau} \int_\om (v- \tilde v) (v - \psi) \, \dd x \leq 0 \quad \forall \psi \in H^1(\om;[0,1]), \mbox{ with }\psi \leq \tilde v.
  \end{align*}
\end{definition}
Following the idea of Proposition~2 in \cite{BurOrtSue2011}, we show the convergence of the minimization scheme~\eqref{eq:sim-minu}--\eqref{eq:sim-minv} to a discrete critical point.
The result can easily be extended to a space-continuous scheme where $\UU_h$ is replaced by $H^1(\om)$ in~\eqref{eq:sim-minu} and~\eqref{eq:sim-minv}.

\begin{proposition}\label{prop:critical-point}
 Let $i\in\{1,\ldots, k\}$ and $(u_{i,j}, v_{i,j})$ be defined as in \eqref{eq:sim-minu}--\eqref{eq:sim-minv}. Then, $(u_{i,j} , v_{i,j})$ converges, up to a subsequence, as $j \to \infty$ to a discrete critical point~$(u_i, v_i) \in \VV_h \times \VV_h$ with bound~$v_{i-1}$.
\end{proposition}
\begin{proof}
We have that, for all $j\in \N$
  \begin{align*}
    \F_h (u_{i,j}, v_{i,j}) + \frac{\alpha}{2\tau} \norm{v_{i,j} - v_{i-1}}_{\VV_h}^2
    &\leq \F_h (u_{i,j-1}, v_{i,j-1}) + \frac{\alpha}{2\tau} \norm{v_{i,j-1} - v_{i-1}}_{\VV_h}^2\\
    &\leq \F_h (u_{i,0} , v_{i,0}) + \frac{\alpha}{2\tau} \norm{v_{i,0} - v_{i-1}}_{\VV_h}^2 \,.
  \end{align*}

  Since~$A$ is uniformly positive definite, the sequence $(u_{i,j}, v_{i,j})$ is bounded in~$\UU_h \times \VV_h$. Hence, we can extract a subsequence $j_k$ such that, for some $u_i, v_i, w \in \VV_h$, we have
  \begin{equation}\label{eq:some_convergences}
    \nabla u_{i,j_k} \to \nabla u_i \,, \quad  v_{i,j_k} \to v_i\,, \quad v_{i,j_k -1} \to w \quad  \text{ as } k\to \infty.
  \end{equation}
  This also implies $u_{i,j_{k - 1}} \to u_i$ and $v_{i,j_{k-1}} \to v_i$ as $k \to \infty$.

  We now prove that $(u_i,v_i)$ is a discrete critical point. In view of \eqref{eq:boh1} and \eqref{eq:crit-point-v-2}, there holds for all $k \in \N$ and for all $\varphi, \psi \in \VV_h$ with $\varphi =0$ on $\partial \om$ and $\psi \leq v_{i-1}$
  \begin{equation*}
    \begin{split}
      0 &= \partial_u \E_h (u_{i,j_k},v_{i,j_k -1}) [\varphi] \,, \\
      0 &\leq \partial_v \F_h(u_{i,j_k},v_{i,j_k}) [\psi - v_{i,j_k}] + \frac{\alpha}{\tau} \int_\om \Pi_h \bigl( (v_{i,j_k} - v_{i-1}) (\psi - v_{i,j_k}) \bigr) \, \dd x \,.
    \end{split}
  \end{equation*}
  Passing to the limit for $k \to \infty$, it follows that
  \begin{equation}\label{eq:crit-point-proof}
   \begin{split}
    0 &= \partial_u \E_h (u_i, w) [\varphi]  \,, \\
    0 &\leq \partial_v \F_h(u_i, v_i) [\psi - v_i] + \frac{\alpha}{\tau} \int_\om \Pi_h \bigl( (v_i - v_{i-1}) (\psi - v_i) \bigr) \, \dd x \,.
    \end{split}
  \end{equation}
  We recall that the last inequality implies that~$v_{i}$ is a solution of~\eqref{eq:sim-minv} with displacement~$u_{i}$.

	It remains to show that $v_i = w$.
	By \eqref{eq:sim-minu}--\eqref{eq:sim-minv} and by the convergence result in~\eqref{eq:some_convergences}, we have that
  \begin{align*}
  \F_h (u_i, w) \! + \! \frac{\alpha}{2\tau} \norm{w - v_{i-1}}_{\UU_h}^2
    & = \lim_{k \to \infty} \F_h (u_{i,j_k}, v_{i,j_k-1}) \! + \! \frac{\alpha}{2\tau} \norm{v_{i,j_k-1} - v_{i-1}}_{\UU_h}^2 \\
    & \leq \lim_{k\to \infty} \F_h (u_{i,j_{k-1}}, v_{i,j_{k-1}}) + \frac{\alpha}{2\tau} \norm{v_{i,j_{k-1}} - v_{i-1}}_{\UU_h}^2 \\
    &=  \F_h (u_i, v_i) + \frac{\alpha}{2\tau} \norm{v_i - v_{i-1}}_{\UU_h}^2 \,.
  \end{align*}
  By strict convexity,~\eqref{eq:sim-minv} has a unique solution. Hence, $v_i = w$. Inequalities~\eqref{eq:crit-point-proof} imply that~$(u_i,v_i)$ is a discrete critical point with bound~$v_{i-1}$.
\end{proof}


%% file: numerics/residualestimate.tex
\subsection{An Anisotropic a Posteriori Error Analysis}\label{sec:residual-estimate}
Goal of this section is to quantify the error associated with a computed discrete critical point through the minimization \eqref{eq:sim-minu}--\eqref{eq:sim-minv}. In particular, we exploit the benefits led by the
employment of an anistropically adapted mesh.
We adopt the setting in \cite{ForPer2001} to recover the anisotropic information,
and we consider a reference triangle $\hat T$, so that, for $T \in \T_h$, there exists an affine map
$\RR_T \colon \hat T \to T$,
with $\RR_T(\hat x) = M_T \hat x + \theta_T$ for all $\hat x \in \hat T$, where $M_T\in \R^{2\times 2}$ is invertible and $\theta_T \in \R^2$ is the shift vector.
We choose $\hat T$ as the equilateral triangle inscribed in the unit circle with one vertex at $(0,1)$. Hence, if $T \in \T_h$ has vertices $(x_1, y_1), (x_2, y_2), (x_3, y_3)$, we have
\begin{equation*}
  M_T = \frac{1}{3}
  \begin{pmatrix}
    \sqrt{3} (x_2 - x_1) & \quad 2 x_3 - x_1 - x_2 \\
    \sqrt{3} (y_2 - y_1) & \quad 2 y_3 - y_1 - y_2
  \end{pmatrix}
  \quad \text{and} \quad
   \theta_T = \frac13
  \begin{pmatrix}
    x_1 + x_2 + x_3\\
    y_1 + y_2 + y_3
  \end{pmatrix}.
\end{equation*}

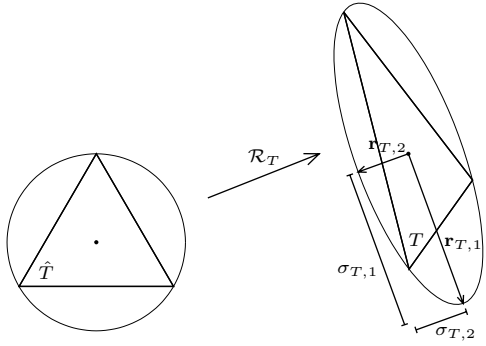
\begin{figure}
  \centering
  \input{numerics/figures/triangle}
  \caption{Geometric sketch of the affine map $\RR_T$, together with the main anisotropic quantities.}
  \label{fig:triangle-mapping}
\end{figure}

We consider the singular value decomposition $M_T = U_T \Sigma_T V_T^\top$, of the matrix $M_T$,
with $U_T = [\rr_{T,1}, \rr_{T,2}]$, $V_T \in \R^{2\times 2}$ orthogonal and $\Sigma_T \in \R^{2\times 2}$ diagonal with entries $\sigma_{T,1} \geq \sigma_{T,2} > 0$.
Hence, for every vector $z \in \R^n$ the following inequality holds
\begin{equation}
\label{svd-inequality}
\sigma_{T,2} \abs{z} \leq \abs{ M_T  z}\leq \sigma_{T,1} \abs{z}\,.
\end{equation}
Geometrically, the left singular vectors $\rr_{T,i}$ identify the directions of the semiaxes of the ellipse circumscribed to~$T$, while the singular values~$\sigma_{T,i}$ measure the corresponding lengths, with $i=1,2$.
The deformation of $T$ is quantified by the aspect ratio $s_T \coloneq \sigma_{T,1}/\sigma_{T,2} \geq 1$, where $s_T = 1$ for equilateral triangles. The matrices $U_T$ and $V_T$ apply rotations, whereas the matrix~$\Sigma_T$ deforms the element (see Figure~\ref{fig:triangle-mapping}).

We denote by $\hat u\big|_{\hat T} \coloneq u \circ \RR_T$ the pull-back on the reference triangle of a generic function $u\colon T \to \R$, and we set $\hat e \coloneq \RR^{-1}_T (e)$ for all $e \in E_h \cap T$, where
$E_h$ represents the skeleton of $\T_h$.
We recall here the anisotropic interpolation error estimates derived in \cite{ForPer2001,ForPer2003} for
the quasi-interpolant operator $Q_h$ as defined in \cite{Cle1975,ScoZha1990,Ver1999}.

\begin{lemma}\label{anisotropicInterpolation}
  Assume that $\#(\Delta_T) \leq {\mathcal N}$ and ${\rm diam}(\RR^{-1}_T(\Delta_T)) \leq C_\Delta$ for every $T \in \T_h$, with $\#(\cdot)$ and ${\rm diam}(\cdot)$
  the cardinality and the diameter of a given set, and $\Delta_T = \{\cup_{K \in \T_h} K \, : \, K \cap T \neq \emptyset\}$ the patch
  of elements associated with $T$. Then, for every $T\in \T_h$, every $e\in E_h$ with $e \in \partial T$, and every $u\in H^1(\Delta_T)$,
  there hold
  \begin{align*}
    |u - Q_h u|_{H^s(T)} &\leq C_s \frac{1}{\sigma_{2,T}^s}\, \norm{M_T^\top \nabla u}_{L^2(\Delta_T)}, \quad s = 0,1\\
    \norm{u - Q_h u}_{L^2(e)} &\leq C_2 \biggl(\frac{h_e}{\sigma_{T,1} \, \sigma_{T,2}}\biggr)^{1/2}\norm{ M_T^\top \nabla u}_{L^2(\Delta_T)},
  \end{align*}
  where $C_i = C_i({\hat T}, {\mathcal N}, C_\Delta)$ for $i=0,1,2$.
\end{lemma}
We also provide the anisotropic error estimate associated with the Lagrangian interpolant $\Pi_h$ (for the
proof, see Proposition 3.3 in \cite{ArtForMicPer2015}), together with the equivalence result between the
standard $H^1(\Delta_T)$-seminorm and the corresponding anisotropic counterpart:
\begin{lemma}\label{anisotropicNodalInterpolation}
  Let $v,\psi \in \VV_h$ and $T\in \T_h$. Then, we have
  \begin{equation*}
    \bignorm{v \psi - \Pi_h (v \psi)}_{L^2 (T)} \leq \hat C h_{T}^2
    \abs{v}_{W^{1,\infty} (T)} \bignorm{\nabla \psi}_{L^2 (T)}\,,
  \end{equation*}
where $\hat C=\hat C(\hat T)$.
\end{lemma}

From \eqref{svd-inequality} we directly infer the following Lemma:
\begin{lemma}\label{equivalence}
Let $z \in H^1(\om)$ and $T \in \T_h$. Then, we have
\begin{equation*}
  \sigma_{T,2} \leq \frac{\norm{ M_T^\top \nabla z}_{L^2(\Delta_T)}}{\norm{\nabla z}_{L^2(\Delta_T)}} \leq \sigma_{T,1} \,.
\end{equation*}
\end{lemma}

Finally, we introduce the notation for the jump of the conormal derivative of a function $w \in \VV_h$:
\begin{equation*}
  \llbracket A \nabla w \rrbracket \coloneq \left\{
    \begin{aligned}
      &\bigabs{ \bigl( \nabla w \vert_{T} - \nabla w \vert_{T'} \bigr)^\top  A \nu_T \, } && \text{on } e\in E_{h} \text{ if }
       \exists T, T' \in \T_h \colon T \cap T' = e \\
      &2 \Bigabs{ \nabla w \vert_{T}^\top\,  A \nu_T  \, } && \text{on } e \in E_h \text{ if }
         \exists T \in \T_h
       \colon  e \subset \partial\om \cap \partial T \,,
    \end{aligned} \right.
\end{equation*}
with $\nu_T$ the unit outward normal vector to $T$. Moreover, we define the edge length function
$h_{\partial T} \colon \partial T \to \R$ by $h_{\partial T} = h_e$ for $e\in E_h \cap \partial T$.
\begin{theorem}\label{thm:residual-estimate}
  Let $(u_h,v_h) \in \UU_h \times \VV_h$ be a  discrete  critical point with bound $\tilde v_h \in \VV_h$.
  For every $T\in\T_h$, we define the quantities
  \begin{align*}
    \gamma_T(u_h,v_h)
    &\coloneq{} \displaystyle \norm{p(u_h,v_h) }_{L^2(T)} + \frac{\mu}{\sigma_{T,2}}
      \bignorm{ \bigl( v_h^2 - \Pi_h(v_h^2) \bigr)  A \nabla u_h}_{L^2(T)} \\
    & \qquad + \frac{\mu}{2 \sqrt{\sigma_{T,1} \sigma_{T,2}}} \bignorm{\sqrt{h_{\partial T}} \bigl( v_h^2 + \eta_\eps \bigr)
      \llbracket A \nabla u_h \rrbracket }_{L^2(\partial T)}\,,\\
    \vphantom{\frac12} p(u_h,v_h)
    &\coloneq b u_h - 2 \mu v_h \nabla u_h^\top A \nabla v_h
      - \mu (v_h^2 + \eta_\eps ) \nabla u_h \cdot \diver(A) \,,\\
    \rho_T(u_h, v_h)
    & \coloneq{} \norm{q(u_h,v_h)}_{L^2(T)} + \frac{\kappa \eps}{\sqrt{\sigma_{T,1} \, \sigma_{T,2}}}  \bignorm{\sqrt{h_{\partial T}}  \llbracket A \nabla v_h \rrbracket }_{L^2(\partial T)} \\
    & \qquad + \frac{h_T^2}{\sigma_{T,2}} \biggnorm{\mu \nabla u_h^\top A \nabla u_h + \frac{\kappa}{2\eps}  \sqrt{a}}_{L^2 (T)} \abs{v_h}_{W^{1,\infty} (T)} \\
    & \qquad + \frac{\alpha \, h_T^2}{\tau \sigma_{T,2}} \norm{\nabla (v_h - \tilde v_h)}_{L^2 (T)}\,,\\
    \vphantom{\frac12} q(u_h,v_h)
    & \coloneq \mu v_h  \nabla u_h^\top A \nabla u_h + \frac{\kappa}{2 \eps} (v_h-1) \sqrt{a}
      - 2 \kappa \eps \nabla v_h \cdot \diver(A) + \frac{\alpha}{\tau} (v_h - \tilde v_h)  \,.
  \end{align*}
  Then, we have
  \begin{equation}
    \begin{aligned}
      \label{classicalResidualEstimateU}
      \bigabs{\partial_u \E (u_h,v_h)[\varphi]}
      \leq{}&  C \sum_{T \in \T_h} \gamma_T(u_h,v_h)  \norm{M_T^\top \nabla \varphi}_{L^2(\Delta_T)}\quad \forall \varphi \in H^1_0 (\om)\,,
    \end{aligned}
  \end{equation}
  and
  \begin{align}
  \label{classicalResidualEstimateV}
    \partial_v \F (u_h,v_h) & [v_h - \psi] + \frac{\alpha}{\tau} \int_\om  (v_h - \tilde v_h ) (v_h - \psi)
     \,\dd x \\
   &  \leq  C \sum_{T\in \T_h} \rho_T(u_h,v_h) \bignorm{M_T^\top \nabla (\psi - v_h)}_{L^2 (\Delta_T)} \nonumber
  \end{align}
  for all $\psi \in H^1 (\Om)$ with $\psi \leq \tilde v_h $.
 \end{theorem}

\begin{proof}
  The linearity of $\varphi \mapsto \partial_u \E (u_h,v_h)[\varphi]$ yields
  \begin{equation} \label{classicalPartialEstimateU}
    \bigabs{\partial_u \E(u_h,v_h)[\varphi]} \leq \bigabs{\partial_u \E(u_h,v_h)[\varphi - \varphi_h]} +
    \bigabs{\partial_u \E(u_h,v_h)[\varphi_h]} \,.
  \end{equation}
  We consider the first term on the right-hand side. Using the divergence theorem and the fact that every second derivative of $u_h|_T$ is zero, we have
  \begin{multline*}
    \partial_u \E(u_h,v_h)[\varphi - \varphi_h]\\
    \begin{aligned}
    ={}& \sum_{T \in \T_h} \biggl\{ \int_{T} b u_h (\varphi-\varphi_h) \,\dd x \\
    &- \mu \int_{T} \Bigl( 2 v_h \nabla u_h^\top A \nabla v_h + \bigl(v_h^2 + \eta_\eps \bigr) \nabla u_h^\top \cdot \diver(A) \Bigr)
    (\varphi - \varphi_h) \,\dd x \\
    &+ \mu \int_{\partial T} \bigl(v_h^2 + \eta_\eps \bigr) \nabla u_h^\top A \, \nu_T \, (\varphi- \varphi_h)  \,\dd x \biggr\}\\
   ={}& \sum_{T \in \T_h} \biggl\{ \int_{T} p(u_h,v_h) (\varphi - \varphi_h) \,\dd x + \frac{\mu}{2}  \int_{\partial T} \bigl(v_h^2 + \eta_\eps \bigr)
    \llbracket A \nabla u \rrbracket \, (\varphi- \varphi_h) \,\dd x \biggr\} \,.
    \end{aligned}
  \end{multline*}
  Hence, by the Cauchy-Schwarz inequality
  \begin{align*}
    \bigabs{\partial_u \E(u_h,v_h)[\varphi - \varphi_h]} \leq & \sum_{T \in \T_h} \Bigl\{ \bignorm{p(u_h,v_h)}_{L^2(T)}
    \bignorm{\varphi - \varphi_h}_{L^2(T)} \\
    & + \frac{\mu}{2} \bignorm{\bigl(v_h^2 + \eta_\eps\bigr)\llbracket A \nabla u_h \rrbracket }_{L^2(\partial T)}
    \norm{\varphi- \varphi_h}_{L^2(\partial T)} \Bigr\}\,.
  \end{align*}
  We now select $\varphi_h \coloneq Q_h \varphi$. By Lemma~\ref{anisotropicInterpolation}, we can estimate
  \begin{align}\label{term11}
    \bigabs{\partial_u \E_h & (u_h,v_h)[\varphi - \varphi_h]}
    \leq C_3 \sum_{T \in \T_h} \biggl( \bignorm{p(u_h,v_h) }_{L^2(T)} \\
    & + \frac{\mu}{2 \sqrt{\sigma_{T,1} \sigma_{T,2}}} \Bignorm{\sqrt{h_{\partial T}} \bigl( v_h^2 + \eta_\eps \bigr)
    \llbracket A \nabla u_h \rrbracket }_{L^2(\partial T)}\biggr) \,  \bignorm{M_T^\top \nabla \varphi}_{L^2(\Delta_T)} \,, \nonumber
  \end{align}
  where $C_3 \coloneq \max \{C_0,C_2\}$.

We now deal with the second contribution on the right-hand side of \eqref{classicalPartialEstimateU}. Using~\eqref{eq:disc-crit-point-u}, Lemmas~\ref{anisotropicInterpolation} and Lemma~\ref{equivalence}, and the fact that~$Q_h$ preserves the boundary values, we obtain
  \begin{align*}
    & \vphantom{\sum_{T \in \T} } \bigabs{\partial_u \E (u_h,v_h)[\varphi_h]}  = \bigabs{\partial_u \E (u_h,v_h)[\varphi_h] - \partial_u \E_h(u_h,v_h)[\varphi_h]} \\
    & \qquad \leq \sum_{T\in \T_h} \Bignorm{ \mu \bigl( v_h^2 - \Pi_h(v_h^2) \bigr) A \nabla u_h}_{L^2(T)}\, \Bigl( \norm{\nabla \varphi -
    \nabla \varphi_h}_{L^2(T)} + \norm{\nabla \varphi}_{L^2(T)}\Bigr) \\
    & \qquad \leq C_4 \sum_{T\in \T_h} \frac{\mu}{\sigma_{T,2}} \, \Bignorm{ \bigl( v_h^2 - \Pi_h(v_h^2) \bigr)  A \nabla u_h}_{L^2(T)}
    \, \norm{M_T^\top \nabla \varphi}_{L^2(\Delta_T)} \,,
  \end{align*}
  with $C_4 \coloneq 1 + C_1$.
  This last estimate, combined with~\eqref{term11}, provides estimate~\eqref{classicalResidualEstimateU}.

  Let us now deal with \eqref{classicalResidualEstimateV}. By~\eqref{eq:disc-crit-point-v},
  for every $\psi \in H^1(\om)$ and every $\psi_h \in \VV_h$ with $0\leq \psi, \psi_h \leq \tilde v_h$ we have
  \begin{align}\label{eq:estimate-v}
    \partial_v \F & (u_h,v_h)[v_h - \psi] + \frac{\alpha}{\tau} \int_\om  (v_h - \tilde v_h ) (v_h - \psi)  \,\dd x \\
    & \leq \partial_v \F (u_h,v_h)[v_h - \psi] + \frac{\alpha}{\tau} \int_\om  (v_h - \tilde v_h ) (v_h - \psi) \,\dd x \nonumber\\
    &\qquad - \partial_v \F_h(u_h,v_h)[v_h - \psi_h] - \frac{\alpha}{\tau} \int_\om  \Pi_h \bigl( (v_h - \tilde v_h )
    (v_h - \psi_h)  \bigr) \,\dd x \nonumber \\
    & \leq \underbrace{\partial_v \F(u_h,v_h)[\psi_h - \psi] + \frac{\alpha}{\tau} \int_\om  (v_h - \tilde v_h )
    (\psi_h - \psi) \,\dd x}_{{\text{(I)}}} \nonumber \\
    &\qquad + \underbrace{\partial_v \F(u_h,v_h)[v_h - \psi_h] - \partial_v \F_h(u_h,v_h)[v_h - \psi_h]}_{{\text{(II)}}} \nonumber \\
    &\qquad + \underbrace{\frac{\alpha}{\tau} \int_\om  (v_h - \tilde v_h ) (v_h - \psi_h) \,\dd x  - \frac{\alpha}{\tau} \int_\om  \Pi_h
    \bigl( (v_h - \tilde v_h ) (v_h - \psi_h)  \bigr) \,\dd x}_{{\text{(III)}}} \,,\nonumber
  \end{align}
  where, in the second inequality, we have added and subtracted the terms
$\partial_v \F(u_h,v_h)[\psi_h]$ and
  $\frac{\alpha}{\tau} \int_\om (v_h -\tilde v_h) \psi_h \,\dd x$.

  We consider the term (I). After integrating by parts on each element $T\in\T_h$, we obtain
  \begin{multline*}
    \partial_v \F (u_h,v_h) [\psi_h - \psi] + \frac{\alpha}{\tau} \int_\om  (v_h - \tilde v_h) (\psi_h - \psi)  \,\dd x\\
    \begin{aligned}
    = \sum_{T\in \T_h} \biggl\{&\mu \int_{T} v_h (\psi_h - \psi) \nabla u_h^\top A \nabla u_h \,\dd x + \frac{\kappa}{2\eps}
    \int_{T} (v_h-1) (\psi_h - \psi) \sqrt{a} \, \dd x  \\
    &- 2 \kappa \eps \int_{T} \nabla v_h \cdot \diver(A)(\psi_h - \psi) \,\dd x  + \kappa \eps \int_{\partial T}
    \llbracket A \nabla v_h \rrbracket (\psi_h - \psi) \,\dd x \\
    &+ \frac{\alpha}{\tau} \int_T  (v_h - \tilde v_h) (\psi_h - \psi)  \,\dd x \biggr\} \,,
    \end{aligned}
  \end{multline*}
  which can be bounded by the Cauchy-Schwarz inequality as
  \begin{align}
  \label{eq:estimate-2}
  &  \partial_v \F (u_h,v_h)  [\psi_h - \psi]  + \frac{\alpha}{\tau} \int_\om  (v_h - \tilde v_h)
    (\psi_h - \psi)  \,\dd x \\
   & \leq \sum_{T\in \T_h} \norm{q(u_h, v_h)}_{L^2(T)}\norm{\psi_h - \psi}_{L^2(T)}
    + \kappa \eps \sum_{T\in \T_h} \bignorm{ \llbracket A \nabla v_h \rrbracket}_{L^2(\partial T)}
   \bignorm{\psi_h - \psi}_{L^2(\partial T)} \,.\nonumber
  \end{align}
  We then choose $\psi_h = Q_h\psi$ and notice that $Q_h(\psi -v_h) = \psi_h - v_h$ and
  $\psi - \psi_h = \psi - v_h - Q_h (\psi - v_h)$.
  This choice, together with Lemma~\ref{anisotropicInterpolation}, allows us to rewrite~\eqref{eq:estimate-2} as
  \begin{align}\label{eq:estimate-v-1}
    \partial_v \F_h & (u_h,v_h)[\psi_h - \psi]  + \frac{\alpha}{\tau} \int_\om  (v_h - \tilde v_h) (\psi_h - \psi)  \,\dd x \\
    &  \leq C_3 \sum_{T\in \T_h}\biggl( \norm{q(u_h, v_h)}_{L^2(T)}  + \frac{\kappa \eps}{\sqrt{\sigma_{T,1} \sigma_{T,2}}}
      \bignorm{\sqrt{h_{\partial T}} \llbracket A \nabla v_h \rrbracket}_{L^2(\partial T)}  \biggr) \nonumber \\
 & \qquad \times \bignorm{M_T^\top \nabla (\psi- v_h)}_{L^2(\Delta_T)} . \nonumber
  \end{align}
  Next, we estimate term (II). The equality
  \begin{equation*}
    (v_h - 1)(v_h-\psi_h) - \Pi_h ((v_h-1) (v_h - \psi_h)) = v_h (v_h-\psi_h) - \Pi_h(v_h (v_h - \psi_h))
  \end{equation*}
  yields
  \begin{align*}
    \vphantom{\int} \partial_v \F & (u_h,v_h) [v_h - \psi_h] - \partial_v \F_h(u_h,v_h)[v_h - \psi_h]\\
    & = \sum_{T\in \T_h} \bigg\{ \mu \int_T \Bigl(v_h (v_h - \psi_h) - \Pi_h \bigl( v_h (v_h-\psi_h) \bigr)  \Bigr) \nabla u_h^\top A \nabla u_h \,\dd x \\
     & \qquad + \frac{\kappa}{2 \eps} \int_T \Bigl(  v_h (v_h - \psi_h) - \Pi_h \bigl( v_h (v_h - \psi_h) \bigr)
      \Bigr)
    \sqrt{a} \,\dd x \bigg\} \,.
  \end{align*}
  Thus, thanks to the Cauchy-Schwarz inequality, to Lemma~\ref{anisotropicNodalInterpolation}, and to the choice of~$\psi_h$, we obtain
  \begin{align*}
  \partial_v \F & (u_h,v_h) [v_h - \psi_h] - \partial_v \F_h(u_h,v_h)[v_h - \psi_h] \\
      & \leq  \sum_{T\in \T_h} \biggnorm{\mu \nabla u_h^\top A \nabla u_h + \frac{\kappa}{2\eps}  \sqrt{a}}_{L^2 (T)} \Bignorm{v_h (v_h - \psi_h) - \Pi_h \bigl( v_h (v_h - \psi_h) \bigr) }_{L^2(T)} \\
      & \leq \hat C \sum_{T\in \T_h} h_T^2 \biggnorm{\mu \nabla u_h^\top A \nabla u_h + \frac{\kappa}{2\eps}  \sqrt{a}}_{L^2 (T)} \, \abs{v_h}_{W^{1,\infty}(T)} \bignorm{\nabla (\psi_h - v_h) }_{L^2 (T)} \,.
  \end{align*}
Now, since $\psi_{h} - v_{h} = \big[ Q_{h}(\psi - v_{h}) - (\psi - v_{h}) \big] + \psi - v_{h}$,
by exploiting Lemma~\ref{anisotropicInterpolation} for $s=1$ and Lemma~\ref{equivalence}, we conclude that
  \begin{align}
  \label{eq:estimate-v-2bis}
      & \partial_v \F (u_h,v_h) [v_h - \psi_h] - \partial_v \F_h(u_h,v_h)[v_h - \psi_h] \\
      & \leq \hat C C_4 \sum_{T\in \T_h} h_T^2\biggnorm{\mu \nabla u_h^\top A \nabla u_h + \frac{\kappa}{2\eps}  \sqrt{a}}_{L^2 (T)}  \abs{v_h}_{W^{1,\infty}(T)} \nonumber \\
      &  \qquad \Big( \frac{1}{\sigma_{2, T}}\bignorm{M_{T}^{\top}\nabla (\psi - v_h)}_{L^2 (\Delta_T)} + \bignorm{\nabla (\psi - v_h) }_{L^2 (\Delta_T)}  \Big) \nonumber \\
      & \leq C_5 \sum_{T\in \T_h} \frac{h_T^2}{\sigma_{T,2}} \biggnorm{\mu \nabla u_h^\top A \nabla u_h +
        \frac{\kappa}{2\eps}  \sqrt{a}}_{L^2 (T)}  \abs{v_h}_{W^{1,\infty}(T)} \bignorm{M_{T}^{\top}\nabla
       (\psi - v_h) }_{L^2 (\Delta_T)} \,, \nonumber
  \end{align}
with $C_5 = 2 \hat C C_4$.
We proceed in a similar way on term (III) in \eqref{eq:estimate-v}, so that we obtain
  \begin{align}\label{eq:estimate-v-3}
      \frac{\alpha}{\tau} & \int_\om    (v_h - \tilde v_h ) (v_h - \psi_h) \,\dd x - \frac{\alpha}{\tau}
      \int_\om  \Pi_h \bigl( (v_h - \tilde v_h ) (v_h - \psi_h) \bigr) \,\dd x  \\
      & \leq \sum_{T\in \T_h} \frac{\alpha}{\tau} \abs{T}^\half \Bignorm{(v_h- \tilde v_h) ( v_h - \psi_h)
      - \Pi_h \bigl( (v_h - \tilde v_h) ( v_h - \psi_h) \bigr)}_{L^2(T)} \nonumber  \\
      & \leq \hat C \sum_{T\in \T_h} \frac{ \alpha \, h_T^2}{\tau} \abs{T}^\half
        \abs{v_h - \tilde v_h}_{W^{1,\infty} (T)} \norm{ \nabla (\psi_h - v_h)}_{L^2(T)} \nonumber \\
      & \leq C_5 \sum_{T\in \T_h} \frac{ \alpha \, h_T^2}{\tau \sigma_{T,2}}
        \norm{\nabla (v_h - \tilde v_h)}_{L^2 (T)} \norm{ M_T^\top  \nabla (\psi - v_h)}_{L^2(T)} \,, \nonumber
  \end{align}
where, in the last inequality, we have also exploited the property that $v_h - \tilde v_h$ is piecewise affine. Combining estimates \eqref{eq:estimate-v-1}--\eqref{eq:estimate-v-3}, we deduce result \eqref{classicalResidualEstimateV}.
\end{proof}

With a view to the mesh adaptation procedure, we combine \eqref{classicalResidualEstimateU} and \eqref{classicalResidualEstimateV} in a single estimate, i.e.,
  \begin{align}\label{eq:residual-estimate}
      \partial_u \E &  (u_h,v_h)[\varphi] + \partial_v \F(u_h,v_h) [v_h - \psi] + \frac{\alpha}{\tau}
      \int_\om  (v_h - \tilde v_h) (v_h - \psi)  \, \dd x \\
    &  \leq C \sum_{T \in \T_h} \Big[\gamma_T(u_h,v_h)  \norm{M_T^\top \nabla \varphi}_{L^2(\Delta_T)} +
      \rho_T  (u_h,v_h) \bignorm{M_T^\top \nabla (\psi - v_h)}_{L^2(\Delta_T)}\Big] \,, \nonumber
  \end{align}
  for all $\varphi \in H^1(\om)$ with $\varphi = 0$ on $\partial \om$, and for all $\psi \in H^1(\om)$ with $\psi \leq \tilde v_h$.

It is evident that result \eqref{eq:residual-estimate} is not yet useful in practice since it depends on the generic functions~$\varphi$ and~$\psi$. As detailed in the next section, to make computable the right-hand side of~\eqref{eq:residual-estimate}, we follow the approach in \cite{ArtForMicPer2015}, first picking $\varphi = u - u_h$ and $\psi = v$, i.e., setting
\begin{align}
 & \Xi (u_h, v_h) \coloneq \sum_{T\in \T_h} \Xi_T (u_h , v_h) \,,\\[1mm]
 & \Xi_T (u_h ,v_h) \coloneq \gamma_T  (u_h, v_h) \norm{M_T^\top \nabla (u-u_h)}_{L^2(\Delta_T)} \label{eq:error-estimate} \\
  &\qquad \qquad\qquad\qquad+ \rho_T (u_h, v_h)  \bignorm{M_T^\top \nabla (v-v_h)}_{L^2(\Delta_T)} \,, \nonumber
\end{align}
for any $T \in \T_h$, and then resorting to a gradient recovery procedure to replace the derivatives
of $u$ and $v$.


%% file: numerics/figures/triangle.tex
\begingroup
\scriptsize
  \makeatletter
  \providecommand\color[2][]{%
    \GenericError{(gnuplot) \space\space\space\@spaces}{%
      Package color not loaded in conjunction with
      terminal option `colourtext'%
    }{See the gnuplot documentation for explanation.%
    }{Either use 'blacktext' in gnuplot or load the package
      color.sty in LaTeX.}%
    \renewcommand\color[2][]{}%
  }%
  \providecommand\includegraphics[2][]{%
    \GenericError{(gnuplot) \space\space\space\@spaces}{%
      Package graphicx or graphics not loaded%
    }{See the gnuplot documentation for explanation.%
    }{The gnuplot epslatex terminal needs graphicx.sty or graphics.sty.}%
    \renewcommand\includegraphics[2][]{}%
  }%
  \providecommand\rotatebox[2]{#2}%
  \@ifundefined{ifGPcolor}{%
    \newif\ifGPcolor
    \GPcolortrue
  }{}%
  \@ifundefined{ifGPblacktext}{%
    \newif\ifGPblacktext
    \GPblacktexttrue
  }{}%
  \let\gplgaddtomacro\g@addto@macro
  \gdef\gplbacktext{}%
  \gdef\gplfronttext{}%
  \makeatother
  \ifGPblacktext
    \def\colorrgb#1{}%
    \def\colorgray#1{}%
  \else
    \ifGPcolor
      \def\colorrgb#1{\color[rgb]{#1}}%
      \def\colorgray#1{\color[gray]{#1}}%
      \expandafter\def\csname LTw\endcsname{\color{white}}%
      \expandafter\def\csname LTb\endcsname{\color{black}}%
      \expandafter\def\csname LTa\endcsname{\color{black}}%
      \expandafter\def\csname LT0\endcsname{\color[rgb]{1,0,0}}%
      \expandafter\def\csname LT1\endcsname{\color[rgb]{0,1,0}}%
      \expandafter\def\csname LT2\endcsname{\color[rgb]{0,0,1}}%
      \expandafter\def\csname LT3\endcsname{\color[rgb]{1,0,1}}%
      \expandafter\def\csname LT4\endcsname{\color[rgb]{0,1,1}}%
      \expandafter\def\csname LT5\endcsname{\color[rgb]{1,1,0}}%
      \expandafter\def\csname LT6\endcsname{\color[rgb]{0,0,0}}%
      \expandafter\def\csname LT7\endcsname{\color[rgb]{1,0.3,0}}%
      \expandafter\def\csname LT8\endcsname{\color[rgb]{0.5,0.5,0.5}}%
    \else
      \def\colorrgb#1{\color{black}}%
      \def\colorgray#1{\color[gray]{#1}}%
      \expandafter\def\csname LTw\endcsname{\color{white}}%
      \expandafter\def\csname LTb\endcsname{\color{black}}%
      \expandafter\def\csname LTa\endcsname{\color{black}}%
      \expandafter\def\csname LT0\endcsname{\color{black}}%
      \expandafter\def\csname LT1\endcsname{\color{black}}%
      \expandafter\def\csname LT2\endcsname{\color{black}}%
      \expandafter\def\csname LT3\endcsname{\color{black}}%
      \expandafter\def\csname LT4\endcsname{\color{black}}%
      \expandafter\def\csname LT5\endcsname{\color{black}}%
      \expandafter\def\csname LT6\endcsname{\color{black}}%
      \expandafter\def\csname LT7\endcsname{\color{black}}%
      \expandafter\def\csname LT8\endcsname{\color{black}}%
    \fi
  \fi
    \setlength{\unitlength}{0.0500bp}%
    \ifx\gptboxheight\undefined%
      \newlength{\gptboxheight}%
      \newlength{\gptboxwidth}%
      \newsavebox{\gptboxtext}%
    \fi%
    \setlength{\fboxrule}{0.5pt}%
    \setlength{\fboxsep}{1pt}%
\begin{picture}(4320.00,2880.00)%
    \gplgaddtomacro\gplbacktext{%
      \csname LTb\endcsname
      \put(2152,1522){\makebox(0,0)[l]{\strut{}$\RR_T$}}%
    }%
    \gplgaddtomacro\gplfronttext{%
      \csname LTb\endcsname
      \put(589,658){\makebox(0,0)[l]{\strut{}$\hat T$}}%
      \put(3349,890){\makebox(0,0)[l]{\strut{}$T$}}%
      \put(2817,691){\makebox(0,0)[l]{\strut{}$\sigma_{T,1}$}}%
      \put(3549,225){\makebox(0,0)[l]{\strut{}$\sigma_{T,2}$}}%
      \put(3615,890){\makebox(0,0)[l]{\strut{}$\rr_{T,1}$}}%
      \put(3050,1622){\makebox(0,0)[l]{\strut{}$\rr_{T,2}$}}%
    }%
    \gplbacktext
    \put(0,0){\includegraphics{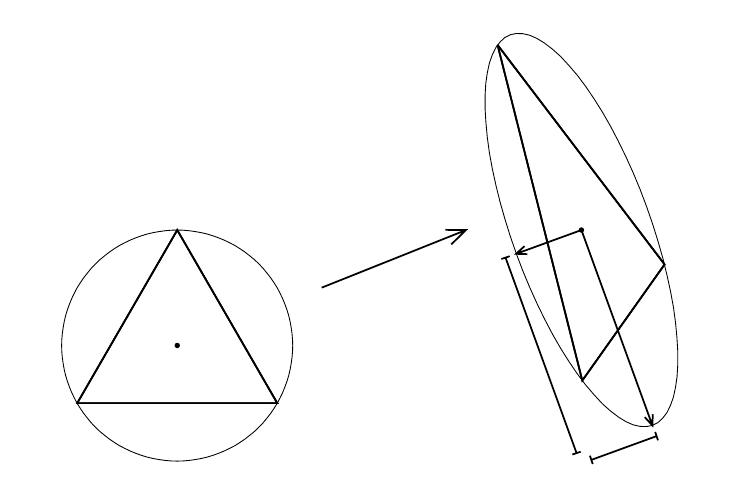}}%
    \gplfronttext
  \end{picture}%
\endgroup

%% file: numerics/meshconstruction.tex
\section{From the Estimator to the Mesh}\label{sec:mesh-construction}
To commute $\Xi(u_h,v_h)$ into an actual a posteriori error estimator able to drive a mesh adaptation procedure,
we follow the metric-based approach in \cite{ArtForMicPer2015,MicPer2011,ForPer2003,WCCM}.
This consists of an iterative procedure, so that, at each iteration~$j$, with $j\geq 0$,
\emph{(i)} we compute the error estimator in the current mesh, $\T_h^{(j)}$; \emph{(ii)} we derive the metric tensor field,
${\cal M}^{(j+1)}$; \emph{(iii)}
we build the new adapted mesh, $\T_h^{(j+1)}$.
We now detail these three steps.

\vspace*{3mm}
\noindent
\emph{(i)} For every $T\in \T^{(j)}_h$ and every $w \in H^1(\om)$, using the singular value decomposition,
$M_T = U_T \Sigma_T V_T^\top$, we can rewrite the norm $\norm{M_T^\top \nabla w}^2_{L^2(\Delta_T)}$
in $\Xi_T(u_h,v_h)$ as
\begin{equation*}
  \begin{split}
    \bignorm{M_T^\top \nabla w}^2_{L^2(\Delta_T)}
    &= \bignorm{\Sigma_T U_T^\top \nabla w }^2_{L^2(\Delta_T)} =  \sum_{i=1}^{2} \int_{\Delta_T}  \sigma_{T,i}^2 \bigabs{ \rr_{T,i} \cdot \nabla w}^2  \,\dd x \\
    &= \sum_{i=1}^{2} \sigma_{T,i}^2 \, \rr_{T,i}^\top \, {\cal G}_T(w) \, \rr_{T,i} \,,
  \end{split}
\end{equation*}
where ${\mathcal G}_T \colon H^1(\Delta_T) \to L^2(\Delta_T;\R^{2\times2})$ is the symmetric semipositive definite matrix
\begin{equation*}\everymath{\displaystyle}
  {\cal G}_T(w) \coloneq
  \begin{pmatrix}
    \int_{\Delta_T} \abs{\partial_1 w}^2 \, \dd x & \int_{\Delta_T} \partial_1 w \, \partial_2 w \, \dd x \\[12pt]
    \int_{\Delta_T} \partial_1 w \, \partial_2 w \, \dd x & \int_{\Delta_T} \abs{\partial_2 w}^2 \, \dd x
  \end{pmatrix}.
\end{equation*}
From \eqref{eq:error-estimate} we obtain
\begin{align*}
  \Xi_T (u_h ,v_h) = & \ \gamma_T (u_h, v_h) \biggl( \sum_{i=1}^{2} \sigma_{T,i}^2 \, \rr_{T,i}^\top \,
  {\cal G}_T(u-u_h) \, \rr_{T,i}\biggr)^\half \\
  & + \rho_T (u_h, v_h) \biggl( \sum_{i=1}^{2} \sigma_{T,i}^2 \, \rr_{T,i}^\top \, {\cal G}_T(v-v_h) \,
  \rr_{T,i} \biggr)^\half.
\end{align*}
Now, the first-order partial derivatives of $u$ and $v$ in ${\cal G}_T$ are replaced via the well-known
Zienkiewicz-Zhu recovery procedure (see \cite{ZZ87,ZZ92}), so that we obtain the local a posteriori error
estimator,
\begin{align}\label{eq:mesh-error-estimate}
  \Xi_T^R (u_h ,v_h) = & \ \gamma_T (u_h, v_h) \biggl( \sum_{i=1}^{2} \sigma_{T,i}^2 \, \rr_{T,i}^\top \,
  {\cal G}_T^R(u_h) \, \rr_{T,i}\biggr)^\half \\
  & + \rho_T (u_h, v_h) \biggl( \sum_{i=1}^{2} \sigma_{T,i}^2 \, \rr_{T,i}^\top \, {\cal G}_T^R(v_h) \,
  \rr_{T,i} \biggr)^\half, \nonumber
\end{align}
where $[{\cal G}_T^R(w_h)]_{ij} = \int_{\Delta_T} \big(\partial_i w_h - R^i(w_h)\big)
\big(\partial_j w_h - R^j(w_h)\big)\, \dd x$, with $i$, $j=1,2$, $w_h \in \UU_h$ and where $[R^1(w_h),R^2(w_h)]^\top$
denotes the recovered gradient of $w_h$.

\vspace*{3mm}
\noindent
\emph{(ii)} Two criteria drive the derivation of the metric, i.e., the minimization of the number of the mesh elements for a given accuracy~$\TOL$ on the global error estimator,
\[
  \Xi^R(u_h,v_h) := \sum_{T \in \T_h}\Xi_T^R(u_h,v_h)\,,
\]
and the error equidistribution,
\[
  \Xi_T^R(u_h ,v_h) \leq \frac{\TOL}{\# \T^{(j)}_h}\,.
\]
For this purpose, we first scale \eqref{eq:mesh-error-estimate} with respect to the area $|T|=\abs{\hat T} \sigma_{T,1} \sigma_{T,2}$ of the element $T\in \T^{(j)}_h$,
such that
\begin{equation*}
  \Xi_T^R (u_h ,v_h) = \alpha_T \Upsilon_T (s_T , \rr_{T,1}) \,,
\end{equation*}
where
\begin{align*}
  \alpha_T &\coloneq \abs{\hat T} (\sigma_{T,1} \, \sigma_{T,2})^\frac{3}{2}, \\
  \Upsilon_T (s_T , \rr_{T,1}) &\coloneq \Bigl( s_T \, \rr_{T,1}^\top \,  \Gamma_T (u_h,v_h)\, \rr_{T,1} + \frac{1}{s_T}\, \rr_{T,2}^\top \, \Gamma_T (u_h,v_h) \, \rr_{T,2} \Bigr)^\half, \\
  \Gamma_T (u_h, v_h) &\coloneq \overline \gamma_T^2 (u_h,v_h) \, \overline {\cal G}_T^R(u_h) +  \overline\rho_T^2 (u_h,v_h) \, \overline {\cal G}_T^R(v_h) \,, \\
  \overline \gamma_T (u_h, v_h) &\coloneq \frac{ \gamma_T (u_h, v_h)}{(\abs{\hat T} \sigma_{T,1} \sigma_{T,2})^\half} \quad \text{and}\quad  \overline \rho_T (u_h, v_h) \coloneq \frac{ \rho_T (u_h, v_h)}{(\abs{\hat T} \sigma_{T,1} \sigma_{T,2})^\half} \,, \\
  \overline {\cal G}_T(w_h) &\coloneq \frac{{\cal G}_T(w_h)}{\abs{\hat T} \sigma_{T,1} \sigma_{T,2}} \quad
  \mbox{ with } w_h = u_h, v_h\, .
\end{align*}
Notice that the quantity $\Upsilon_T (s_T , \rr_{T,1})$ implicitly depends also on~$\rr_{T,2}$ via
the orthonormality condition $\rr_{T,1}^\top\rr_{T,2} = 0$.

Thus, to minimize the cardinality of the mesh (or, likewise, to maximize the triangle area)
while enforcing the local accuracy ${\TOL}/{\# \T^{(j)}_h}$,
we are led to solve the local constrained minimization problem
$$
  \min_{\substack{s_T \geq 1, \, \rr_{T,1} \in \bS^1}} \Upsilon_T (s_T, \rr_{T,1}) \,,
$$
$\bS^1$ being the unit sphere.
Following \cite{ForMicPer2004}, we can analytically compute the unique solution to this problem, given
by
$$
  s_T^* = \sqrt{ \frac{\vartheta_{T,1}}{\vartheta_{T,2}} },
  \quad \rr^*_{T,1} = \vv_{T,2} \,,
$$
with $\{\mathbf{v}_{T,i}, \vartheta_{T,i}\}$ the eigenpair of $\Gamma_T(u_h,v_h)$ for $i = 1,2$,
with $\vartheta_{T,1} > \vartheta_{T,2}$ and $\mathbf{v}_{T,i}^{\top} \mathbf{v}_{T,j}=\delta_{ij}$.
Finally, the equidistribution criterion allows us to compute the optimal lengths
\begin{equation}\label{eq:new-metric}
   \sigma^*_{T,1} = \Biggl( \frac{\TOL}{\sqrt{2} \, \abs{\hat T}\, \# \T_h^{(j)}} \, \sqrt{\frac{\vartheta_{T,1}}{\vartheta^2_{T,2}}} \Biggr)^\frac{1}{3} \quad \text{and} \quad \sigma^*_{T,2} = \Biggl( \frac{\TOL}{\sqrt{2} \, \abs{\hat T}\, \# \T_h^{(j)}} \, \sqrt{\frac{\vartheta_{T,2}}{\vartheta^2_{T,1}}} \Biggr)^\frac{1}{3}.
\end{equation}
The metric field $\M^{(j+1)}$ is approximated by a piecewise tensor, provided by
\begin{equation}\label{metrica}
  \M^{(j+1)}\big|_T = \frac{1}{({\sigma^*_{T,1}})^2}\rr^*_{T,1}{\rr^{*,\top}_{T,1}} +
                      \frac{1}{({\sigma^*_{T,2}})^2}\rr^*_{T,2}{\rr^{*,\top}_{T,2}} \,,
\end{equation}
for each $T\in\T_h^{(j)}$ (see \cite{GeoBor1998}).
We remark that the mismatch between the index $(j)$ for the mesh and $(j+1)$ for the metric is due to
the predictive feature of the adaptive algorithm, which exploits the information on the current mesh
to extrapolate the ``optimal'' mesh for the next iteration.

\vspace*{3mm}
\noindent
\emph{(iii)} This step is committed to a metric-based mesh generator. In particular, we choose the FreeFEM environment~\cite{Hec2012}. The metric $\M^{(j+1)}$ becomes the input to the built-in
function ${\tt adaptmesh}$, which provides the anisotropic adapted mesh $\T_h^{(j+1)}$.


%% file: numerics/computation.tex
\section{Numerical Examples}\label{sec:num-examples}
Next step is to properly combine the minimization in \eqref{eq:sim-minu}--\eqref{eq:sim-minv} together
with the adaptation procedure detailed in the previous section.
With this aim, we resort to an approach that is a variant to Algorithms 2 and 3 in \cite{ArtForMicPer2015},
itemized in Algorithm~\ref{alg:mesh-adapt-alg} below.

The procedure consists of three main loops: the outermost cycle steps over the quasi-static time advancing,
the intermediate one manages the update of the mesh, while the innermost loop controls the optimization
of the physical variables~$u$ and~$v$.
This last phase is supervised by a maximum number $\mathtt{MaxIt}$ of iterations, together
with a control on the increment of the phase field, to within the tolerance $\TOL_v$.
In order to recover the possible lack of accuracy on $v$, the same check on the increment is also
required in the intermediate loop, in combination with a stagnation of the mesh cardinality, up to
a tolerance $\TOL_m$.

The minimization performed in lines 9 and 16 are carried out by an interior point method using the package \texttt{IPOPT} (see \cite{WaeBie2006}), included in \texttt{FreeFEM} (see \cite{Hec2012}).
\texttt{IPOPT} is a common large-scale nonlinear optimization tool based on the interior point algorithm. Both equality and inequality constraints can be tackled via suitable slack variables. The bound on the phase field can be directly enforced as a box contraint.
Among the input parameters of \texttt{IPOPT}, the user has to also provide the gradient of the functional and of the constraint with respect to the phase field.

The metric computation in line 11 is driven by $u_h = u_{i,j}$ and $v_h = v_{i,j}$.
The operator $\Pi_h^{(m+1)}$ is the Lagrangian interpolant associated with the mesh~$\T_h^{m}$
evaluated at the vertices of the mesh $\T_h^{(m+1)}$, which is employed to project the phase field
on the newly adapted mesh before the next iteration.

Table~\ref{tab:num-parameter} gathers the values adopted in the numerical assessment for both the
input parameters to Algorithm~\ref{alg:mesh-adapt-alg} and for the physical quantities involved in functional~\eqref{eq:h1-approx-functional}. For a sensitivity analysis with respect to some of these parameters in the anti-plane case we refer to \cite{ArtForMicPer2015b}. In the tests below, following \cite{BouFraMar2000,BurOrtSue2010,ArtForMicPer2015},
we consider notched specimens characterized by a thin slit to model the initial crack.
The time dependent boundary condition in \eqref{eq:sim-minu} is assumed to be linear.
For technical reasons related to the definition of space $\GSBV(\om)$, we extend the physical domain beyond the Dirichlet boundary. Such an extension turns out to be advisable also for the phase field in order to
avoid an underestimate of the fracture energy when the damage approaches the Dirichlet boundary.
\begin{algorithm}[htb]
  \caption{Alternating Minimization + Anisotropic Mesh Adaptation for Shells}\label{alg:mesh-adapt-alg}
  \begin{algorithmic}[1]
    \everymath{\displaystyle}
    \setlength{\parskip}{3pt}
    \State Input: $\TOL$, $\TOL_m$, $\TOL_v$, $\mathtt{MaxIt}$, $\alpha$, $\tau$,
                  $u_0, v_0$, $\T_h^{(0)}$
    \For{$i=0$ to $k$}
    \State $j\leftarrow 0$; $u_{i,0} \leftarrow u_0$; $v_{i,0} \leftarrow v_0$
    \Repeat
    \State $m \leftarrow 0$
    \Repeat
    \State $j \leftarrow j+1$; $m \leftarrow m+1$
    \State $u_{i,j} \leftarrow \argmin \bigl\{\E_h (u, v_{i,j-1}) : u \in \UU_h^{(m)} , u = g(t_i) \text{ on } \partial\omega \bigr\}$
    \State $v_{i,j} \leftarrow \argmin \biggl\{\F_h ( u_{i,j} , v) + \frac{\alpha}{2\tau} \norm{v-v_{i-1}}^2_{\VV_h} : v \in \VV_{h}^{(m)}, v \leq v_{i-1} \biggr\}$
    \Until $m = \mathtt{MaxIt}$ \textbf{or} $\norm{ v_{i,j} -   v_{i,j-1}}_\infty < \TOL_v$
    \State compute $\M^{(m+1)}$ based on \eqref{metrica}
    \State generate $\T_h^{(m+1)}$ associated with $\M^{(m+1)}$
    \State $v_{i,j} \leftarrow \Pi_h^{(m+1)} (v_{i,j})$; $v_{i-1} \leftarrow \Pi_h^{(m+1)} (v_{i-1})$
    \Until{$\frac{\abs{\# \T_h^{(m+1)} - \# \T_h^{(m)}}}{\# \T_h^{(m)}} < \TOL_m$ \textbf{and} $\norm{ v_{i,j} -   v_{i,j-1}}_\infty < \TOL_v$}
\State $u_i \leftarrow \argmin \bigl\{\E_h (u, v_{i,j}) : u \in \UU_h^{(m+1)} , u = g(t_i) \text{ on } \partial\omega \bigr\}$
\State $v_i \leftarrow \argmin \biggl\{\F_h ( u_{i} , v) + \frac{\alpha}{2\tau} \norm{v-v_{i-1}}^2_{\VV_h} : v \in \VV_{h}^{(m+1)}, v \leq v_{i-1} \biggr\}$
    \State $\T_h^{(0)} = \T_h^{(m+1)}$
    \EndFor
  \end{algorithmic}
\end{algorithm}

\begin{table}
  \caption{Input parameters to Algorithm \ref{alg:mesh-adapt-alg} and physical quantities
           for functional~\eqref{eq:h1-approx-functional}.}
  \centering
  \begin{tabular}{cccccccccc}
    \hline
    $\TOL$ & $\TOL_m$ & $\TOL_v$ & $\mathtt{MaxIt}$ & $\tau$ & $\eps$ & $\eta$ & $\kappa$ & $\lambda$ & $\mu$ \\
    \hline
    $ 10^{-3} $ & $ 10^{-2}  $ & $ 2\cdot 10^{-3} $ & 8 & $10^{-2}$ &  $ 5 \cdot 10^{-3} $ & $ 10^{-5} $ & $ 1 $ & $0$ & $1$ \\
    \hline
  \end{tabular}
  \label{tab:num-parameter}
\end{table}

\subsection{A Piece of a Cylinder}
We consider a piece of cylindrical surface with radius $R=1$ and length $L$.
As the map $\phi$, we choose cylindrical coordinates
\begin{equation}\label{eq:param-cylinder}
  (x,y) \mapsto
  \begin{pmatrix}
    R \cos x \\
    R \sin x \\
    y
  \end{pmatrix}
  \quad \tforall (x,y) \in \om = \biggl(-\frac{\pi}{2} , \frac{\pi}{2} \biggr) \times (0,L).
\end{equation}
With this at hand, we have
\begin{equation*}
  (a^{\alpha \beta}) =
  \begin{pmatrix}
    \displaystyle\frac{1}{R^2 } & 0 \\
    0 & 1 \\
  \end{pmatrix},
  \quad
  (b_{\alpha\beta}) =
  \begin{pmatrix}
    -R & 0 \\
    0 & 0 \\
  \end{pmatrix}
  \quad \text{and} \quad
  \sqrt{a} = R.
\end{equation*}
For the crack initialization, we define the notch $\Gamma \coloneq [-10^{-3} , 10^{-3}] \times [0,0.3]$,
so that the computation takes place in $\om \setminus \Gamma$. We also set
\begin{equation}\label{eq:bc}
  g(t)  \coloneq
  \left\{
    \begin{aligned}
      &t &&\text{on }  [10^{-3}, \pi/2 ] \times \{0\}, \\
      -&t &&\text{on } [-\pi/2 , -10^{-3}] \times \{0\}, \\
      &0&&\text{elsewhere.}
    \end{aligned}
  \right.
\end{equation}
The extended domain adopted in such a case is $\om \cup (-\frac{\pi}{2} , \frac{\pi}{2} ) \times (-0.1,0]$.
\begin{figure}
  \centering
  \begin{subfigure}{0.3\textwidth}
    \includegraphics[width=\textwidth]{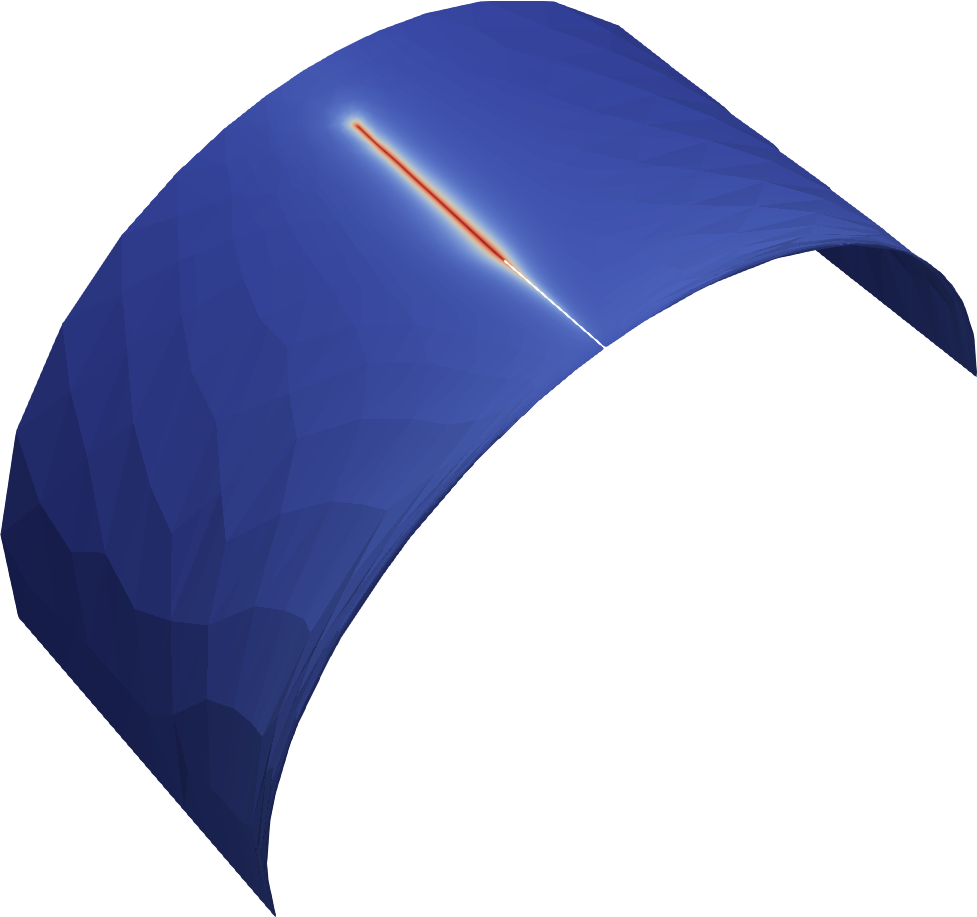}%
    \subcaption*{$t=1.9$}%
  \end{subfigure}%
  \hspace{1em}%
  \begin{subfigure}{0.3\textwidth}
    \includegraphics[width=\textwidth]{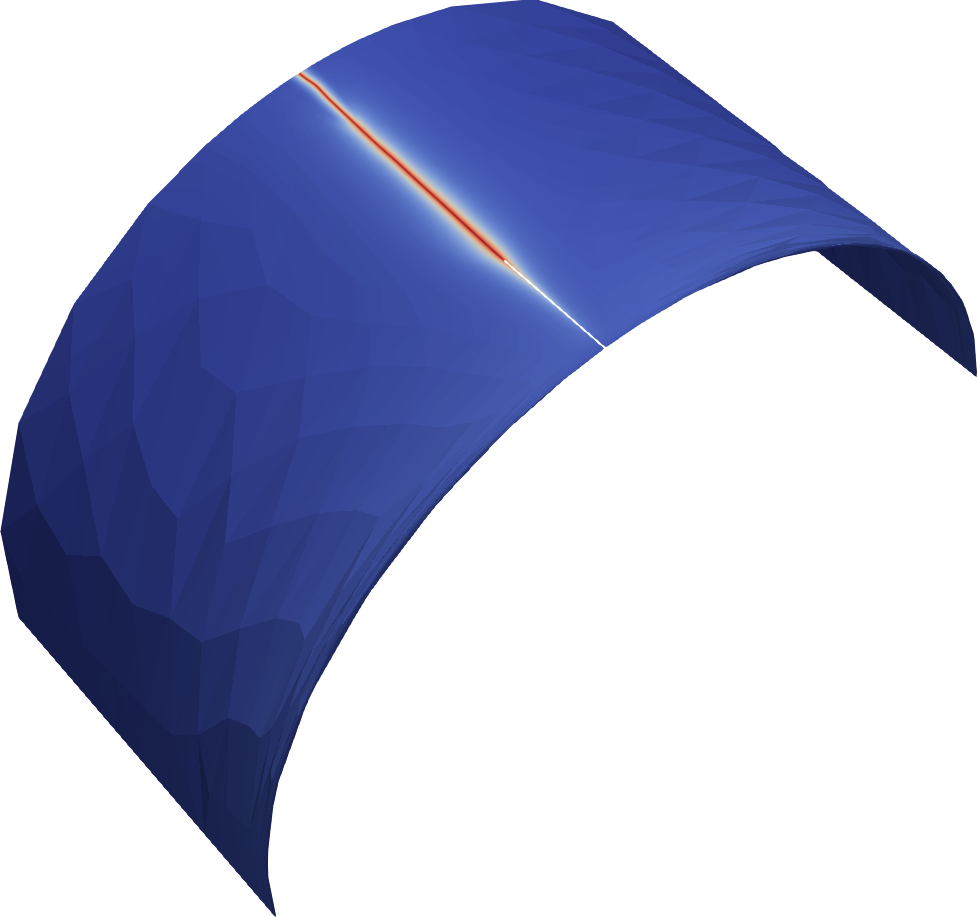}%
    \subcaption*{$t=1.91$}%
  \end{subfigure}%
  \hspace{1em}%
  \begin{subfigure}{0.3\textwidth}
    \includegraphics[width=\textwidth]{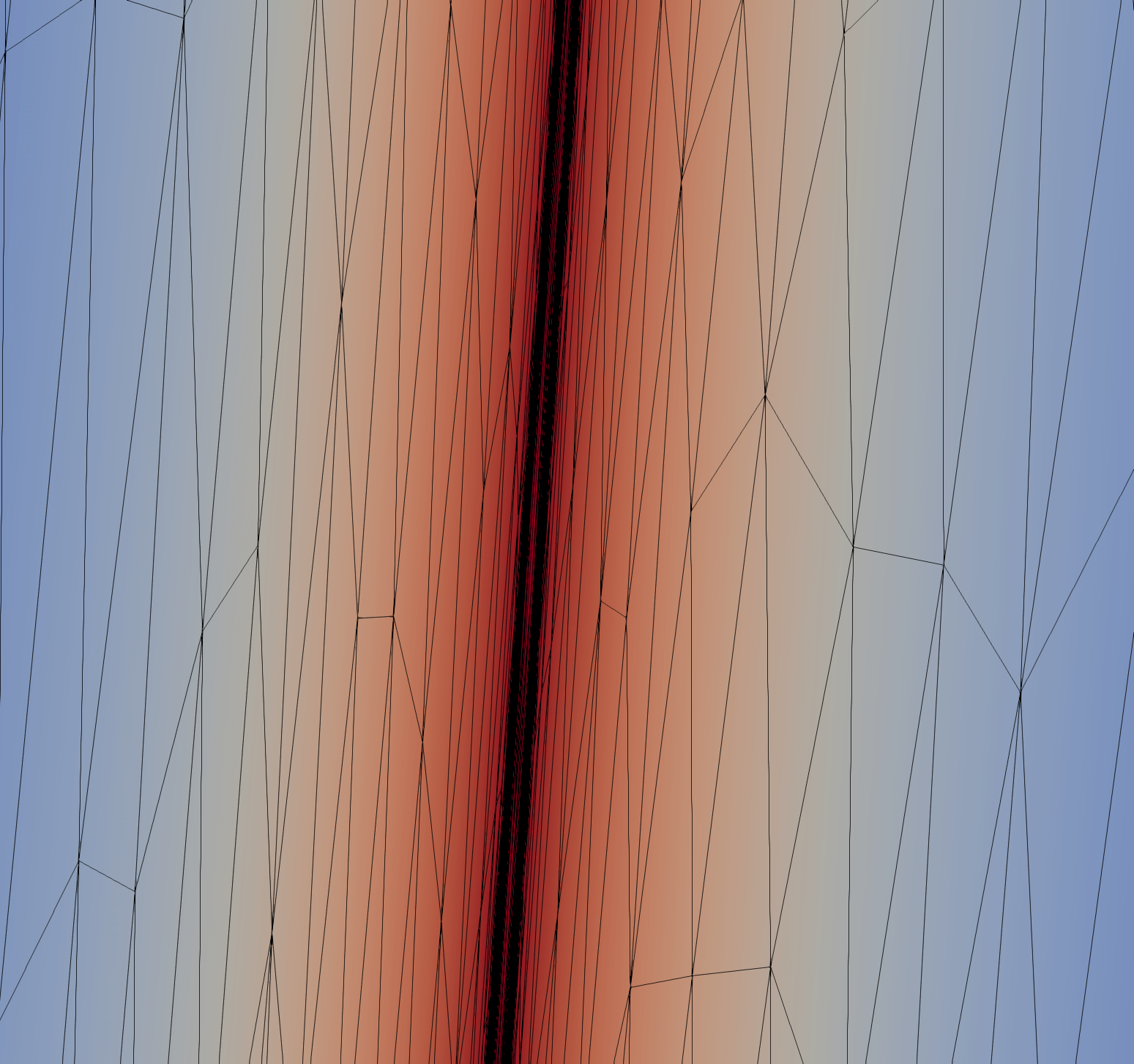}%
    \subcaption*{Enlarged mesh along the crack}%
  \end{subfigure}%
  \caption{Piece of a cylinder: phase field at two consecutive times and detail of the mesh around the crack
          for $L=1$.}%
  \label{fig:straight-cylinder}%
\end{figure}

In Figure~\ref{fig:straight-cylinder} we show the phase field computed for $L=1$ as well as a zoom in on
the mesh close to the crack, where it exhibits a strong directional behavior.
\begin{figure}\centering
  \begin{subfigure}{0.3\textwidth}
    \includegraphics[width=\textwidth]{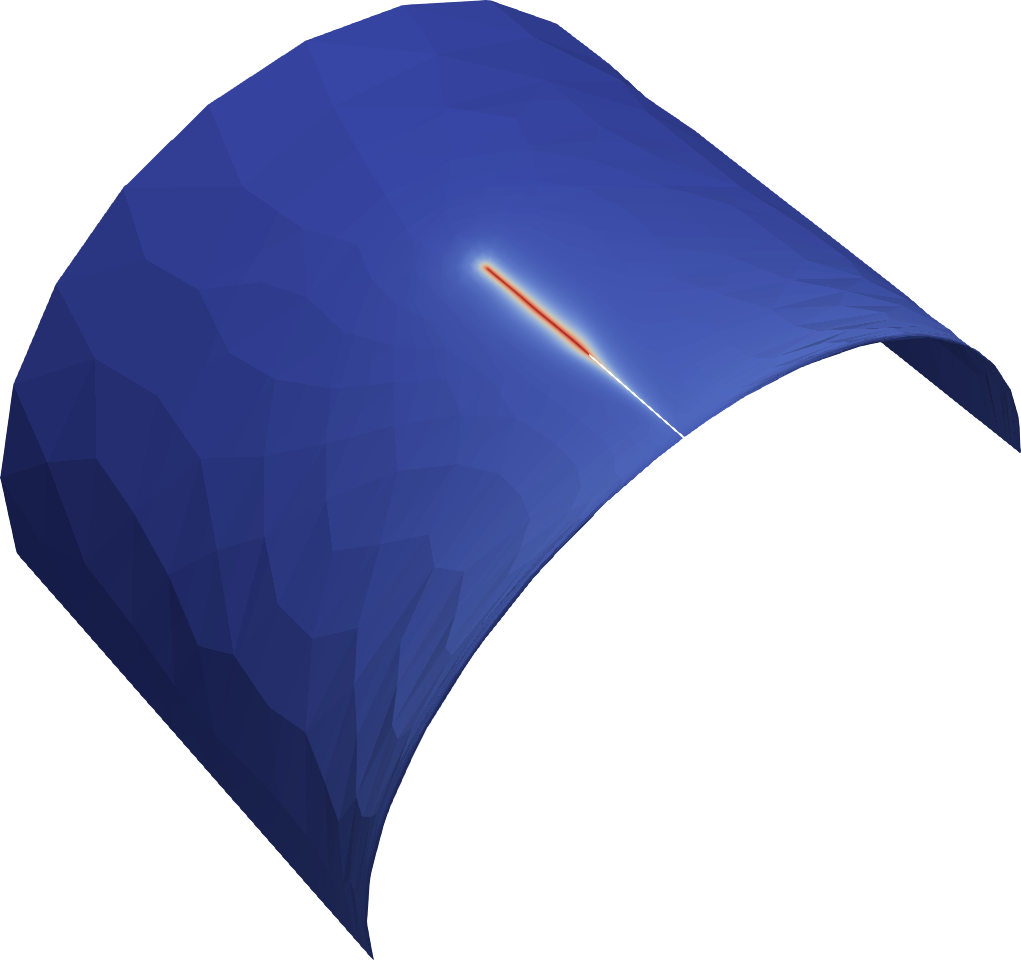}%
    \subcaption*{$t=1.91$}%
  \end{subfigure}%
  \hspace{1em}%
  \begin{subfigure}{0.3\textwidth}
    \includegraphics[width=\textwidth]{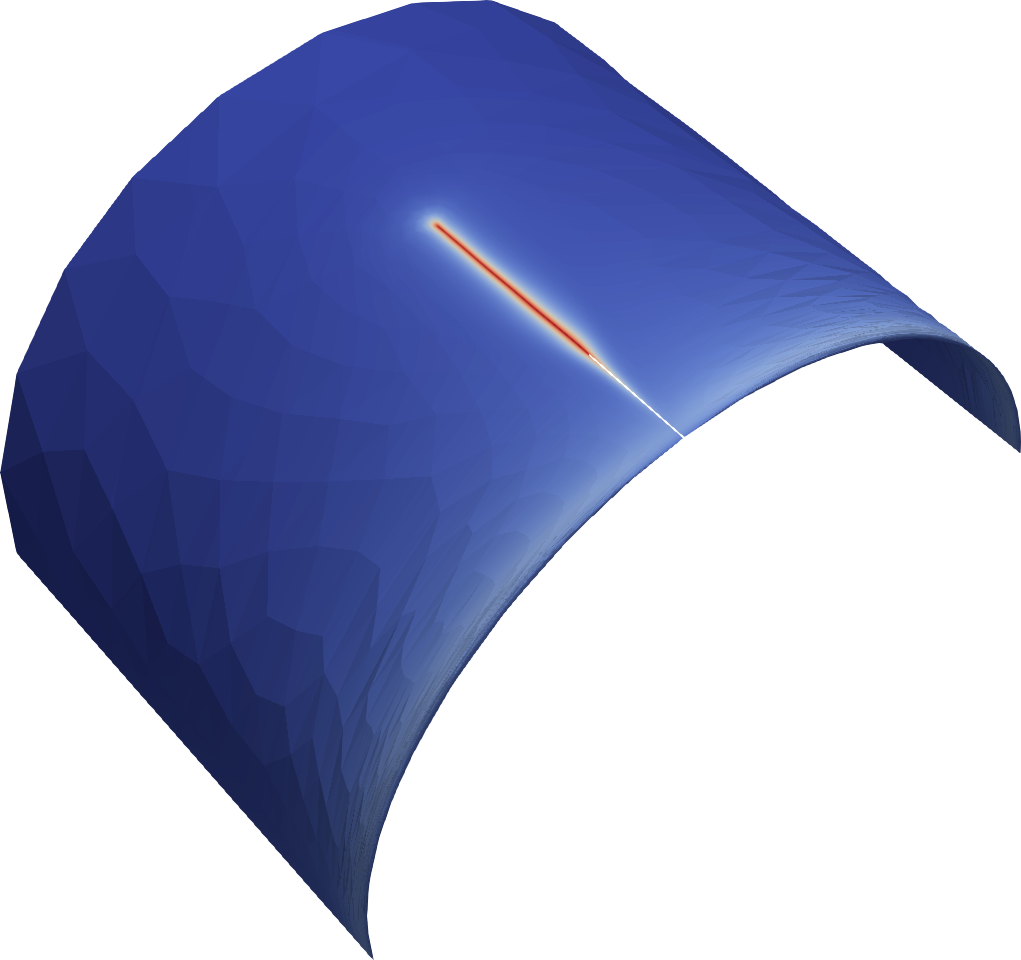}%
    \subcaption*{$t=2.83$}%
  \end{subfigure}%
  \hspace{1em}%
  \begin{subfigure}{0.3\textwidth}
    \includegraphics[width=\textwidth]{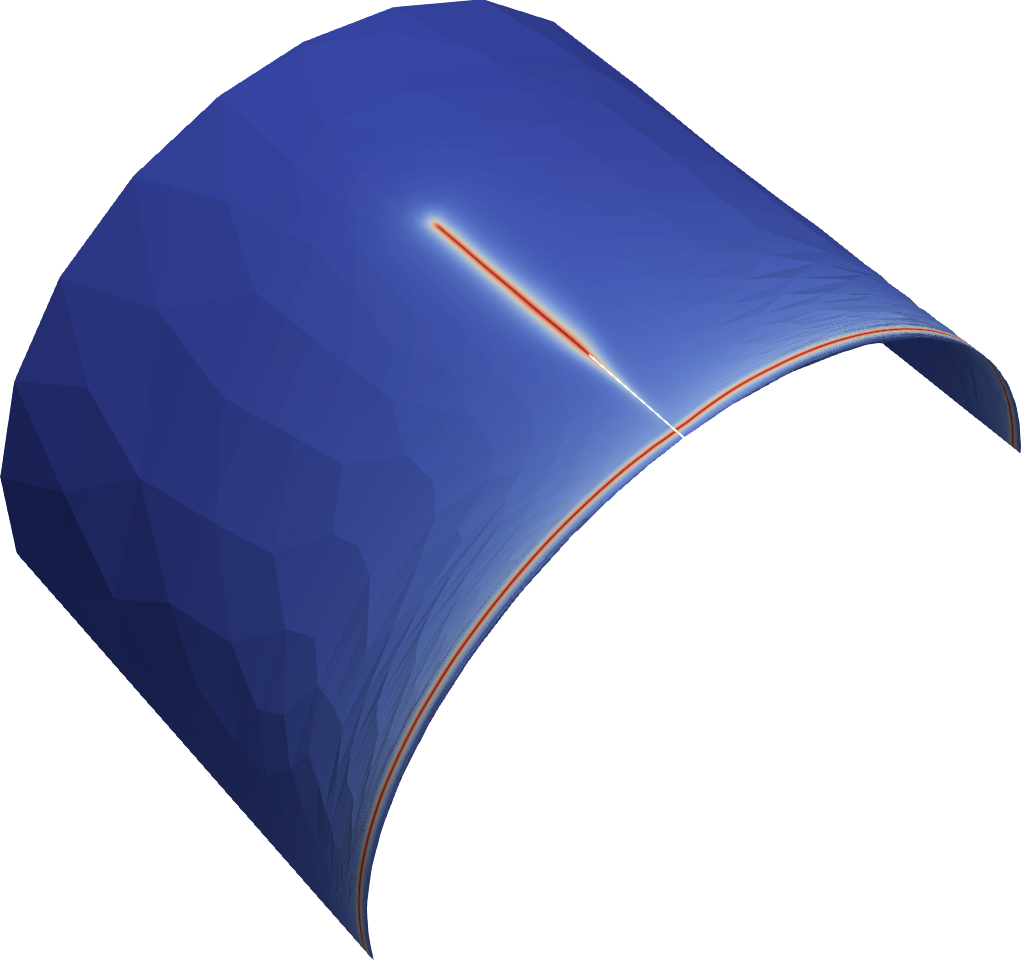}%
    \subcaption*{$t=2.84$}%
  \end{subfigure}%
  \caption{Piece of a cylinder: phase field at three times for $L=2$.}%
  \label{fig:long-cylinder}%
\end{figure}

Note that the term $\int_\om \stiff^{\alpha \beta \sigma \tau} b_{\alpha \beta} b_{\sigma \tau} \abs{u}^2 \sqrt{a} \,\dd x $ in the functional \eqref{eq:h1-approx-functional} adds some energy even though the displacement is constant, due to a curvature effect. Furthermore, the boundary condition creates some tension along the boundary
itself. Thus, if the length $L$ is sufficiently large, a crack is generated along the boundary before
the original crack fully develops. This phenomenon is confirmed in Figure~\ref{fig:long-cylinder}, where
we set $L=2$. The initial crack propagates until $t=2.83$. Then, at $t=2.84$ the surface suddenly breaks along the Dirichlet boundary. To contain this effect, we pick the Lam\'e coefficient $\lambda$ equal to zero
in Table~\ref{tab:num-parameter}.

\begin{figure}
  \centering
    \includegraphics[width=0.45\textwidth]{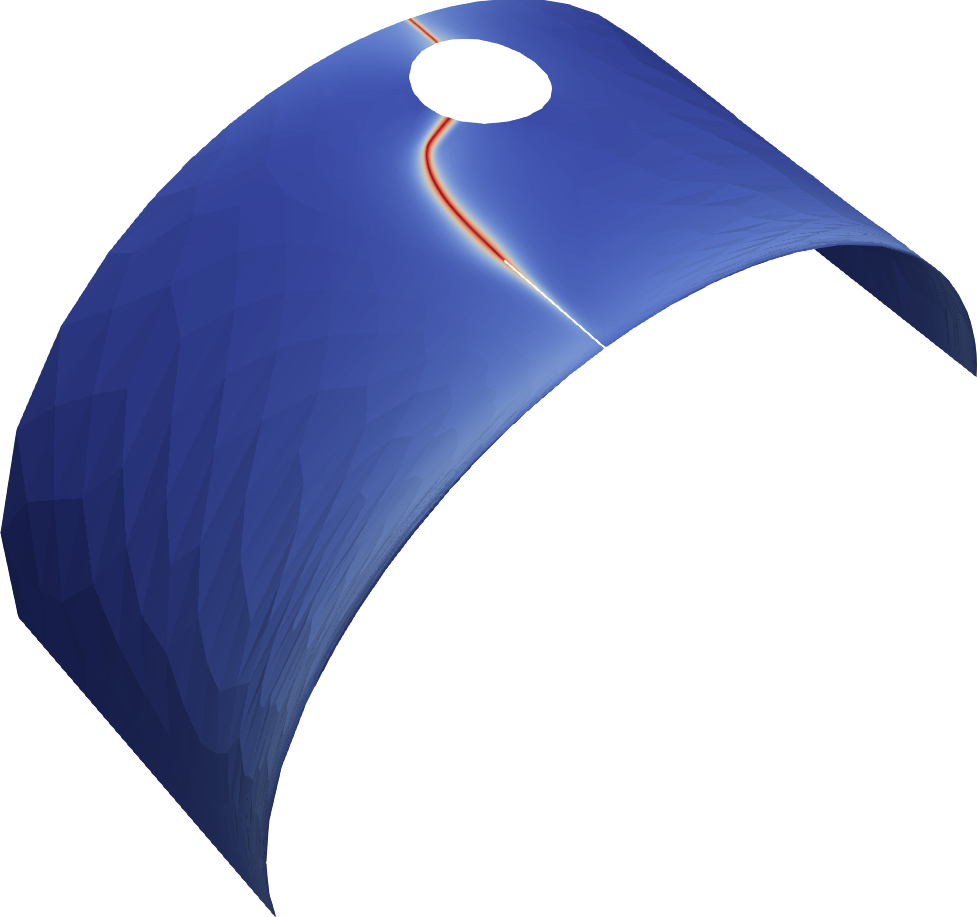}%
    \qquad
    \includegraphics[width=0.45\textwidth]{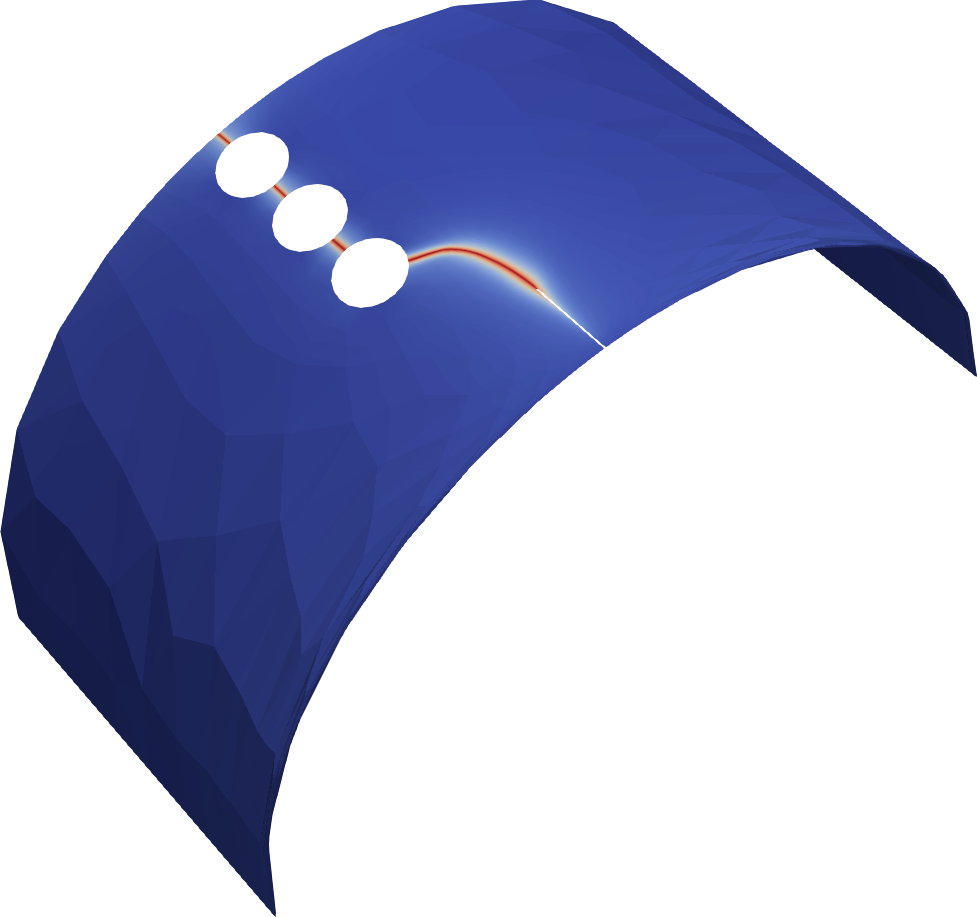}%
  \caption{Piece of a cylinder: phase field at time $t = 2.97$ and $t=1.46$ for the single-hole
           (left) and three-hole (right) configuration for $L=1$.}
  \label{fig:hole-cylinder}%
\end{figure}
\begin{figure}
    \centering
    \scalebox{.8}{\input{numerics/figures/CrackLength.tex}}
    \scalebox{.8}{\input{numerics/figures/NumberOfElements.tex}}
  \caption{Piece of a cylinder: crack length (left) and number of triangles (right) as functions of time for
  the configurations in Figure~\ref{fig:straight-cylinder}, (a), in Figure~\ref{fig:hole-cylinder}, left, (b),
  and in Figure~\ref{fig:hole-cylinder}, right, (c).}%
  \label{fig:plots}%
\end{figure}
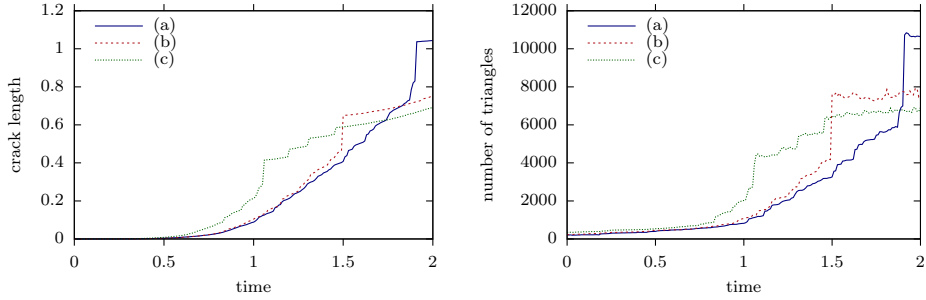
We now weaken the surface by introducing holes. In particular, we consider the two configurations
in Figure~\ref{fig:hole-cylinder} characterized by a single hole with radius $0.15$ centered at $(0.3,0.75)$
and by three holes with radius $0.08$ and centers $(-0.2,0.88)$, $(-0.2,0.68)$, $(-0.2,0.48)$.
In both cases, the crack bends entering the holes. This confirms that the crack path is not biased
by the anisotropic mesh adaptatation, consistently with what observed in \cite{ArtForMicPer2015}.

In Figure~\ref{fig:plots} we provide more quantitative information about the physics of the problem and the mesh adaptation procedure for all the considered configurations.
In particular, in the left panel, we plot the time evolution of the quantity $\kappa^{-1} \DD_h(v_h)$, which $\Gamma$\nobreakdash-converges to the length of the crack
(see~Section~\ref{cha:numerics} and~\ref{appendix}), while, in the right panel, we show the trend of the cardinality of the mesh.
Both the crack length and the number of triangles exhibit a similar trend since the most relevant phenomena occur around the crack path.
\begin{figure}
  \centering
  \begin{subfigure}{0.28\textwidth}
    \includegraphics[width=\textwidth]{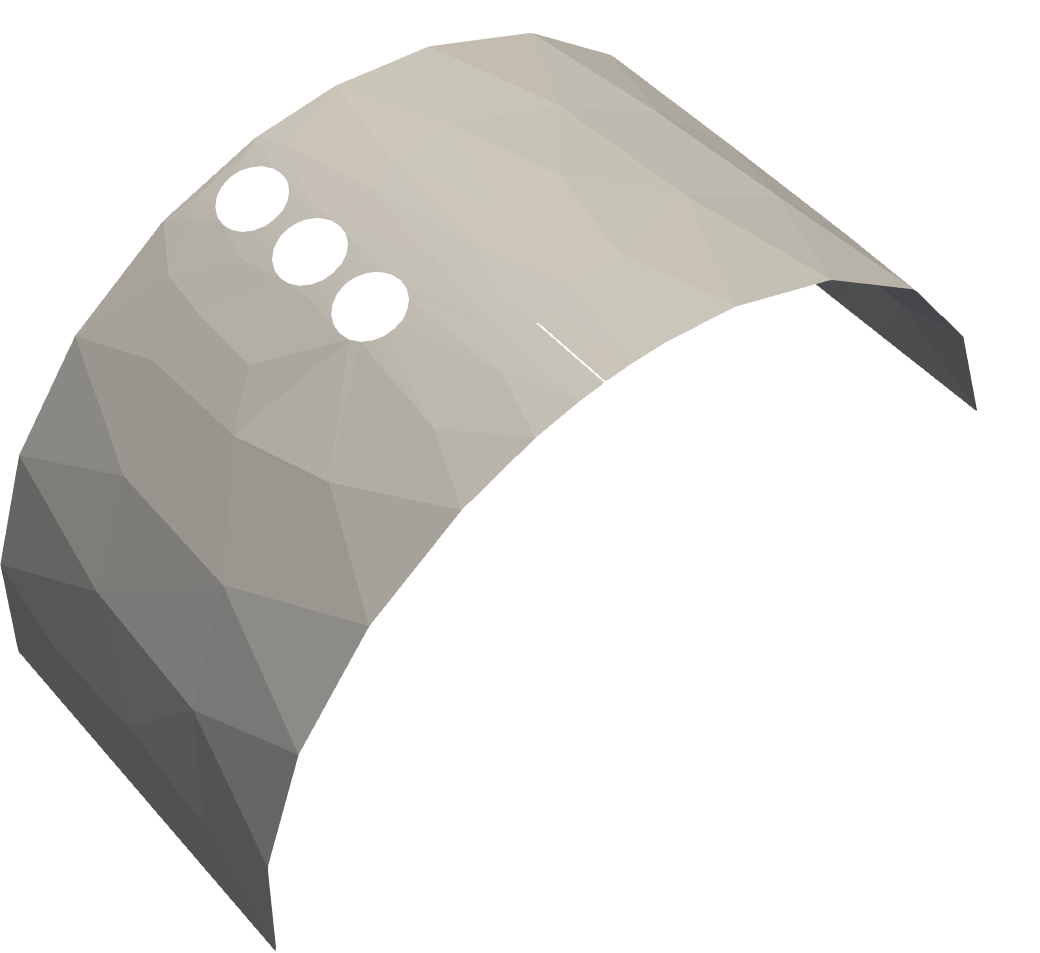}%
    \subcaption*{$t=0$}%
  \end{subfigure}%
  \hspace{2em}%
  \begin{subfigure}{0.28\textwidth}
    \includegraphics[width=\textwidth]{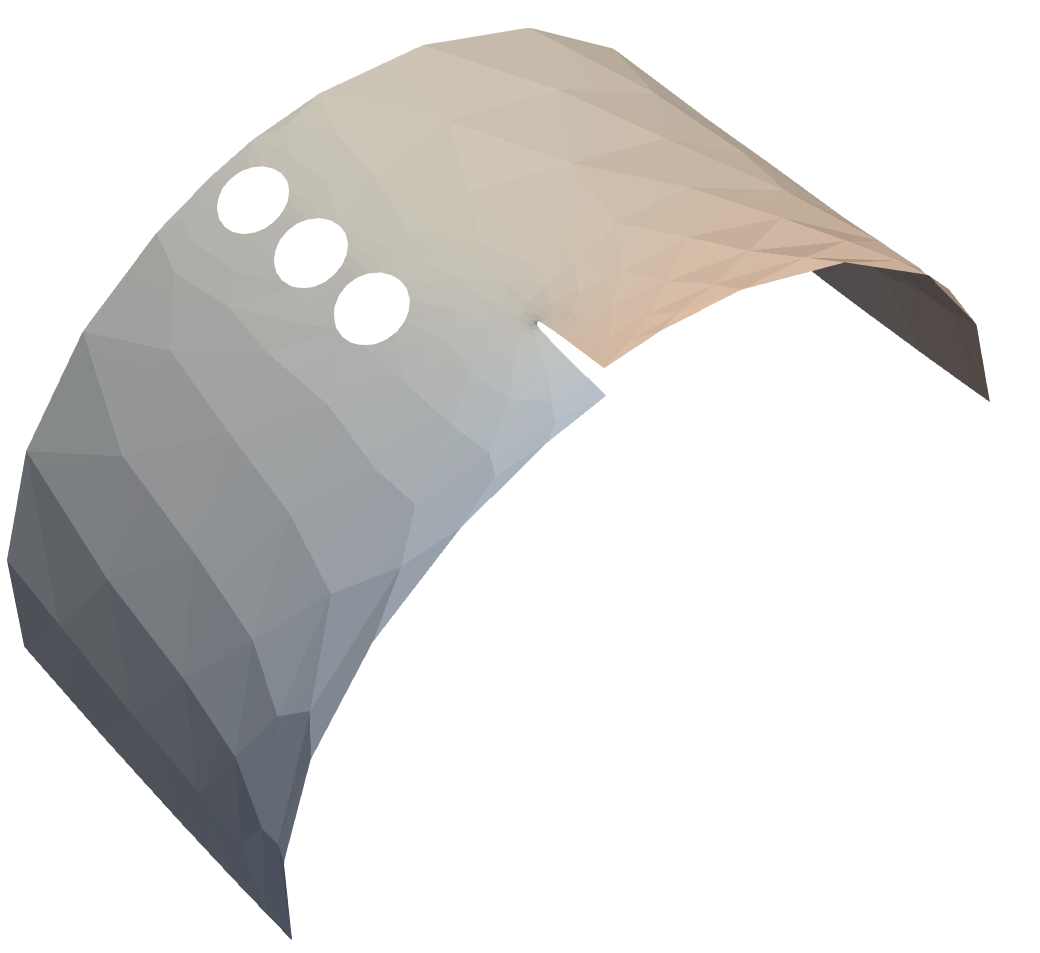}%
    \subcaption*{$t=0.4$}%
  \end{subfigure}%
  \hspace{2em}%
  \begin{subfigure}{0.28\textwidth}
    \includegraphics[width=\textwidth]{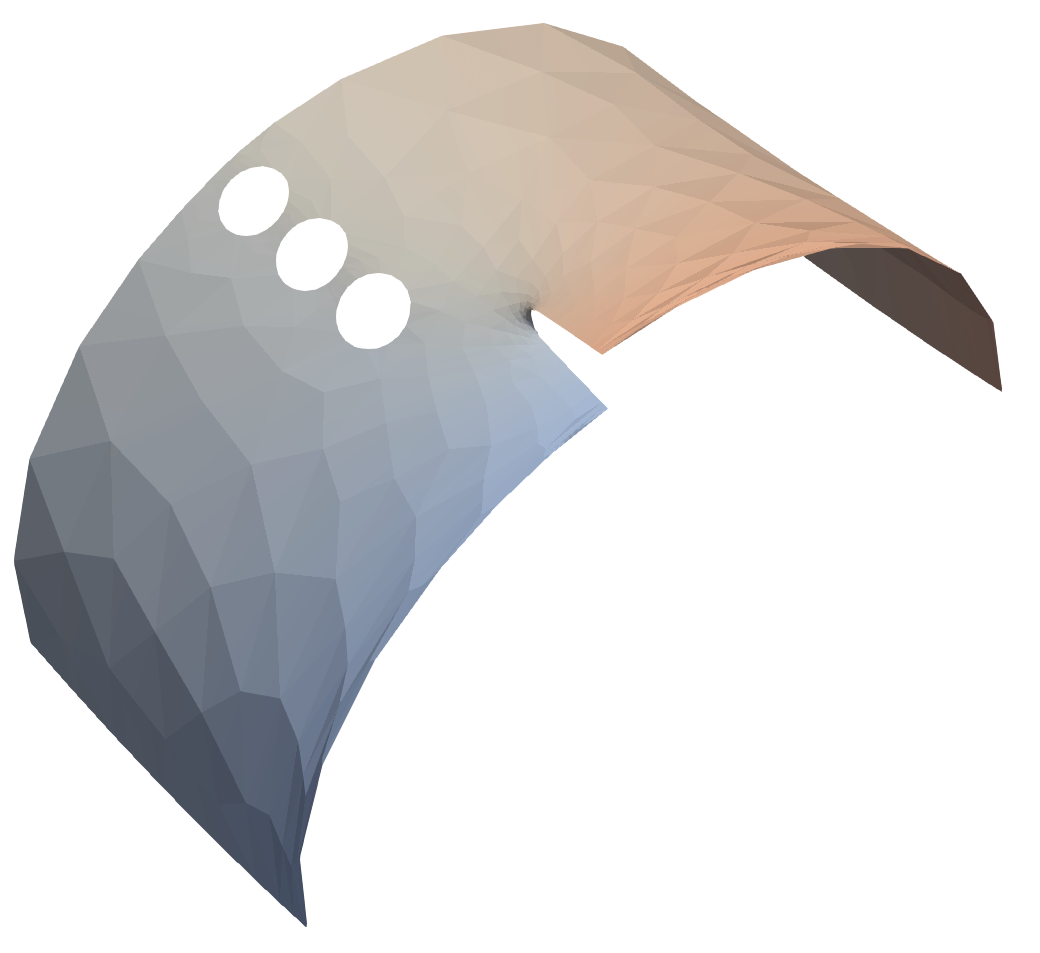}%
    \subcaption*{$t=0.8$}%
  \end{subfigure}\\[1em]%
  \begin{subfigure}{0.28\textwidth}
    \includegraphics[width=\textwidth]{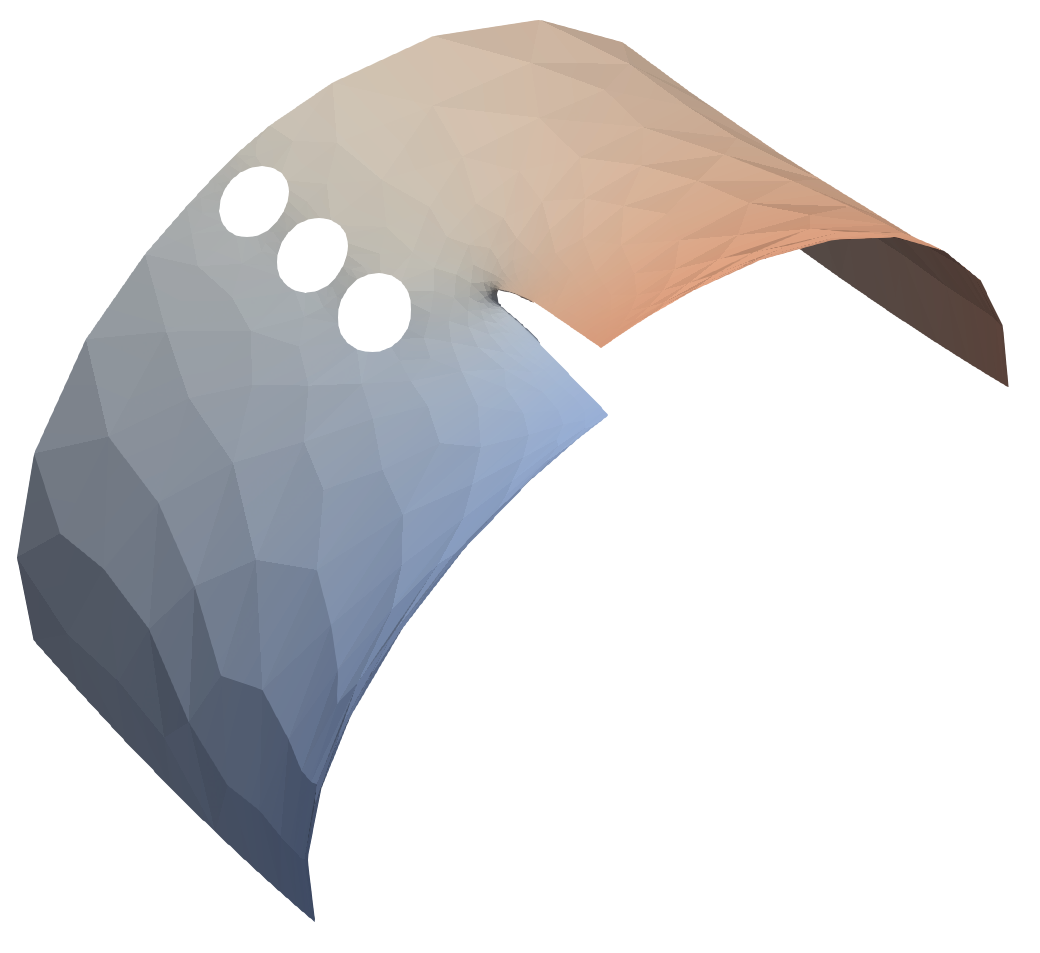}%
    \subcaption*{$t=1$}%
  \end{subfigure}%
  \hspace{2em}%
  \begin{subfigure}{0.28\textwidth}
    \includegraphics[width=\textwidth]{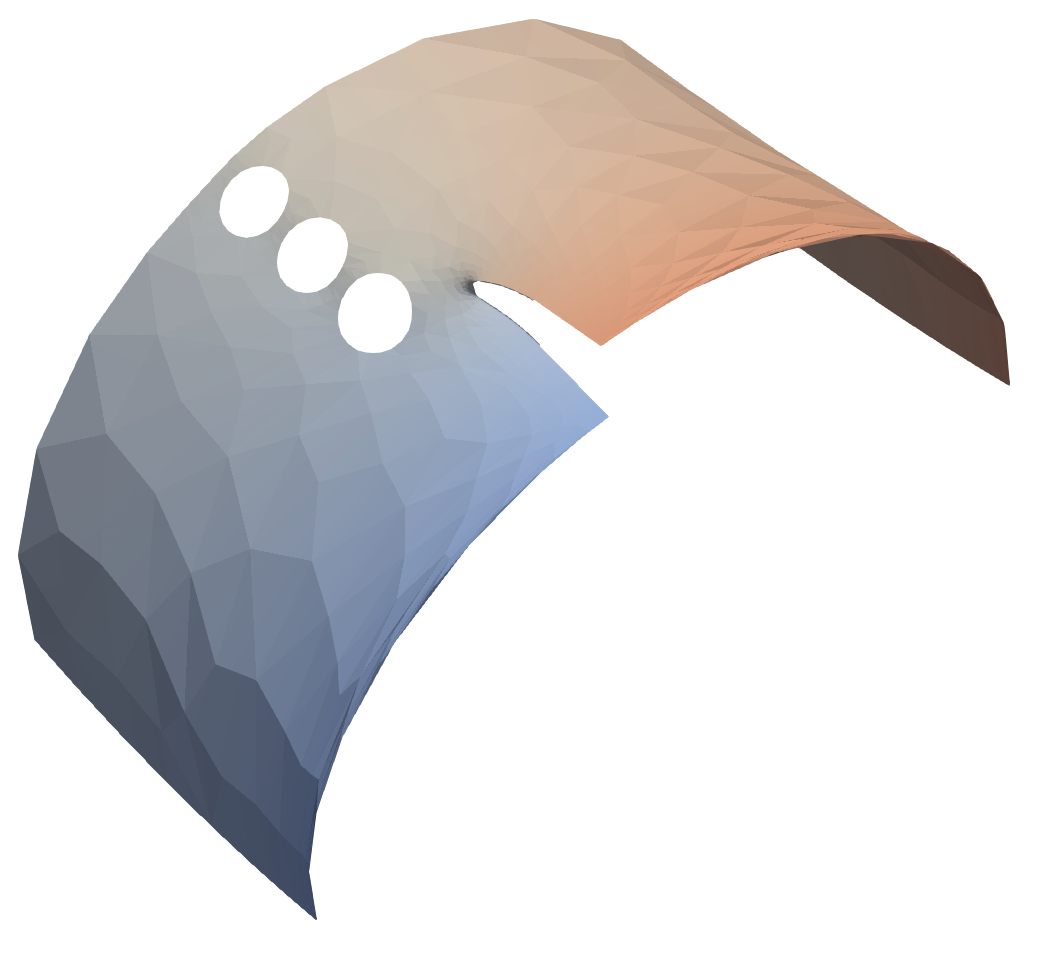}%
    \subcaption*{$t=1.05$}%
  \end{subfigure}%
  \hspace{2em}%
  \begin{subfigure}{0.28\textwidth}
    \includegraphics[width=\textwidth]{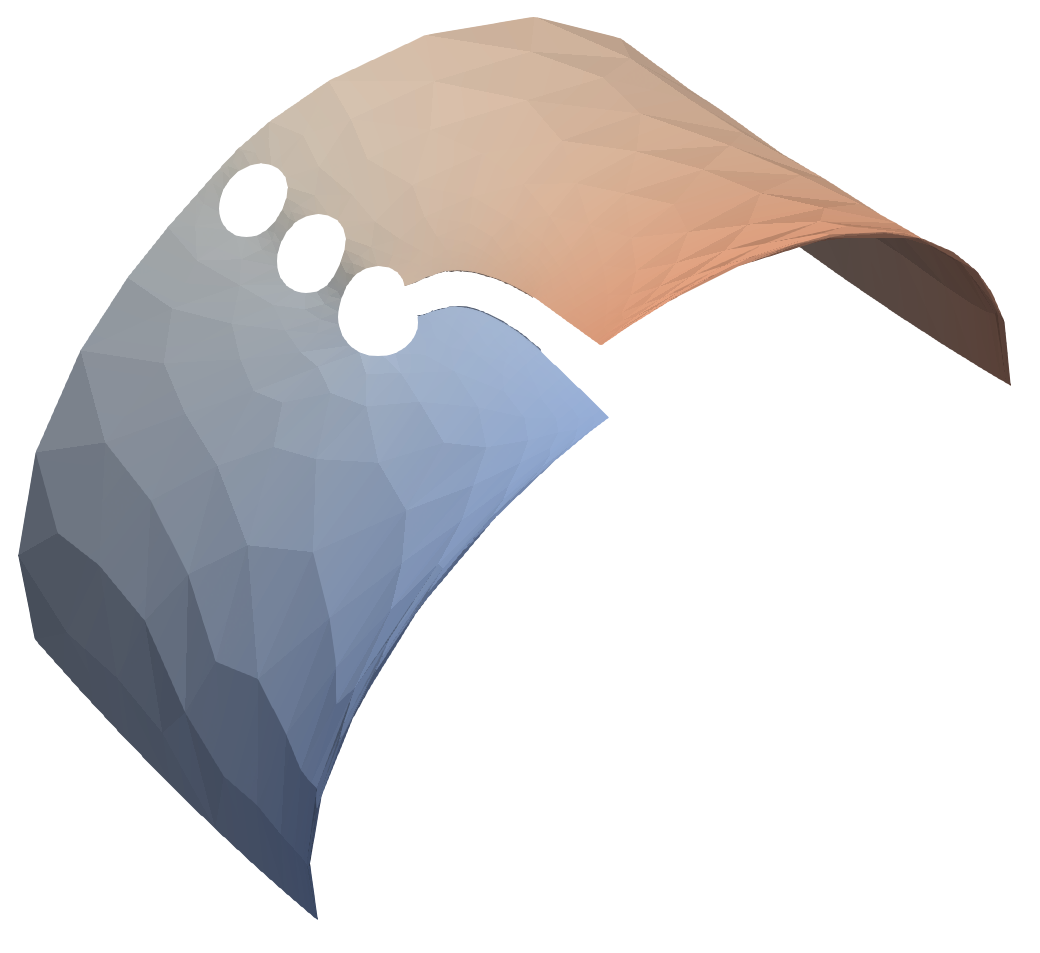}%
    \subcaption*{$t=1.07$}%
  \end{subfigure}\\[1em]%
  \begin{subfigure}{0.28\textwidth}
    \includegraphics[width=\textwidth]{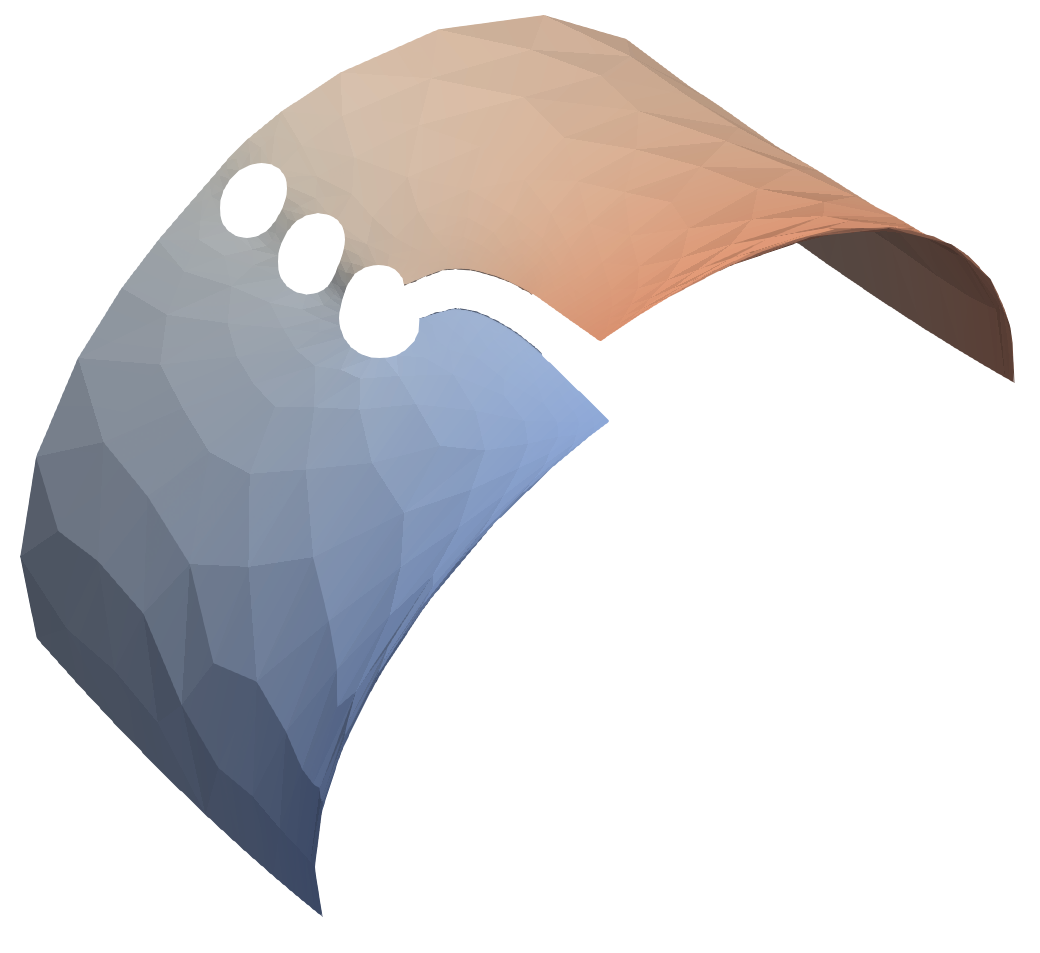}%
    \subcaption*{$t=1.19$}%
  \end{subfigure}%
  \hspace{2em}%
  \begin{subfigure}{0.28\textwidth}
    \includegraphics[width=\textwidth]{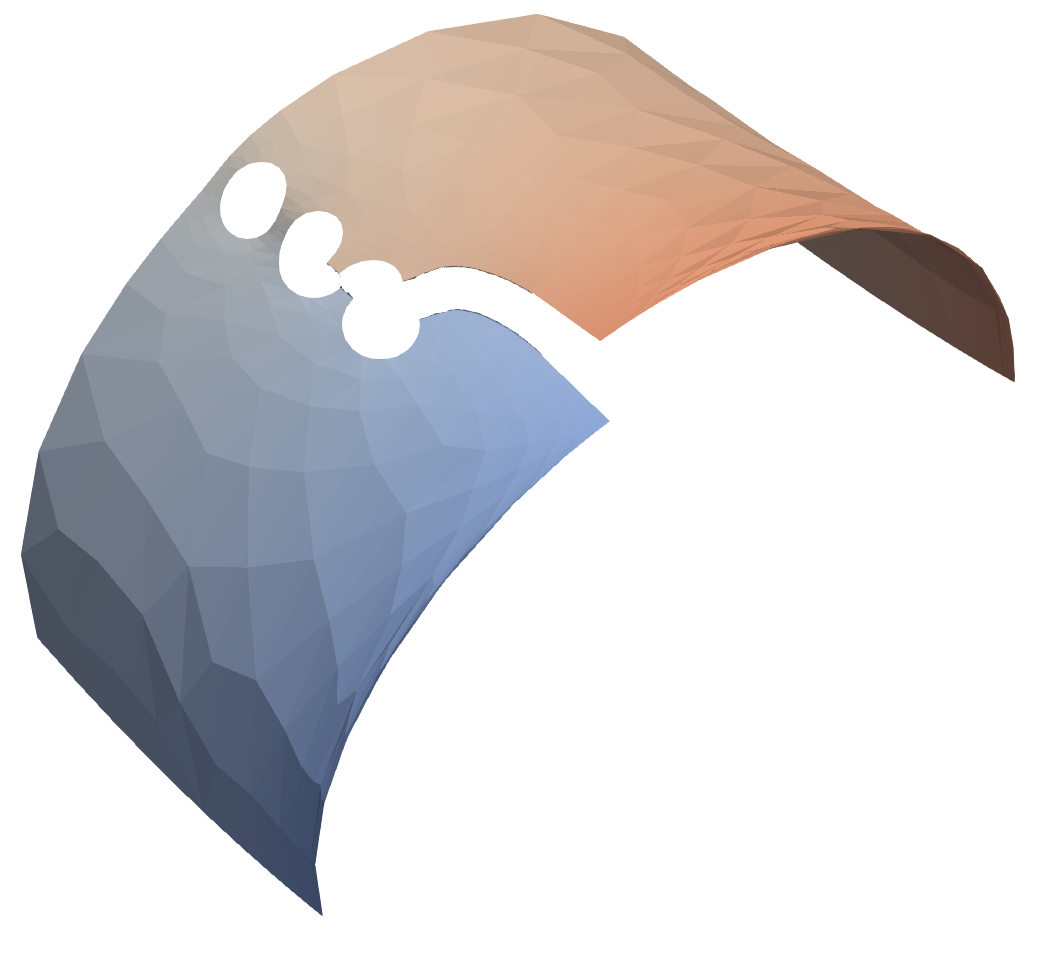}%
    \subcaption*{$t=1.20$}%
  \end{subfigure}%
  \hspace{2em}%
  \begin{subfigure}{0.28\textwidth}
    \includegraphics[width=\textwidth]{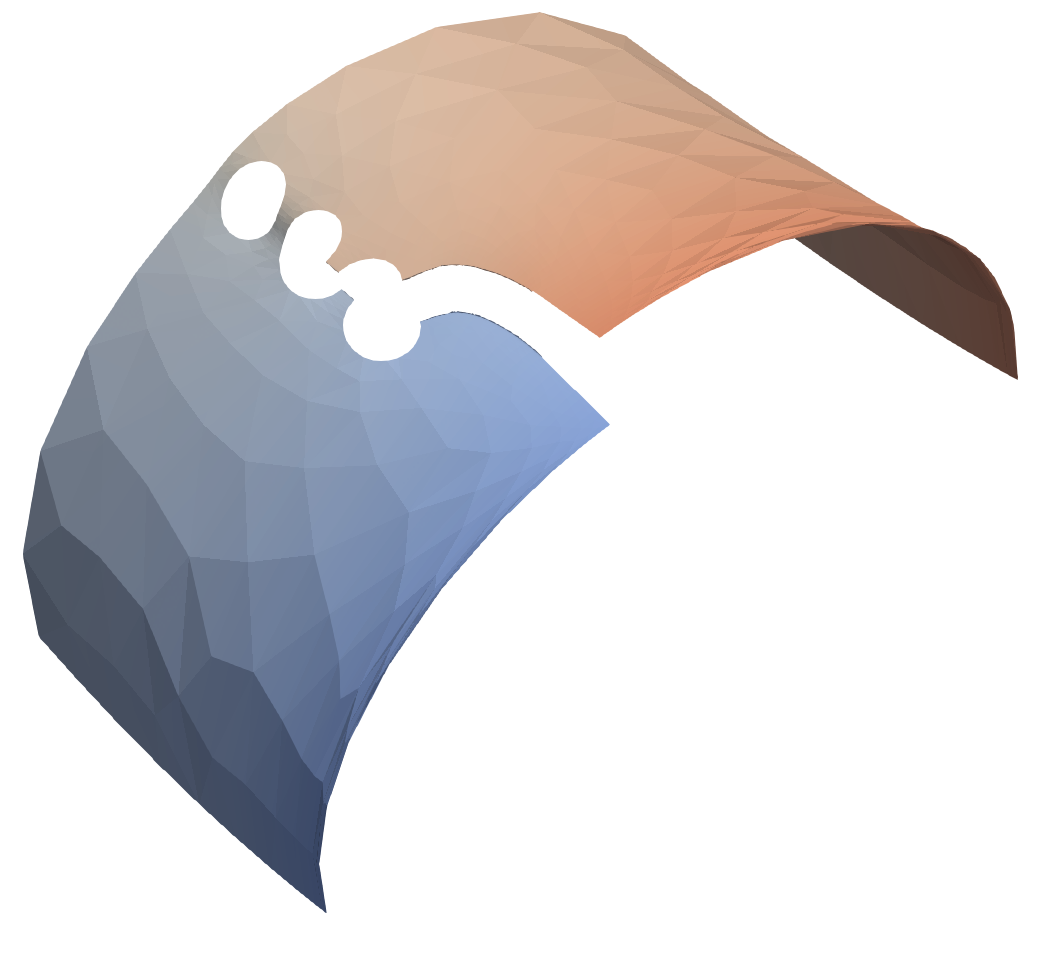}%
    \subcaption*{$t=1.30$}%
  \end{subfigure}\\[1em]%
  \begin{subfigure}{0.28\textwidth}
    \includegraphics[width=\textwidth]{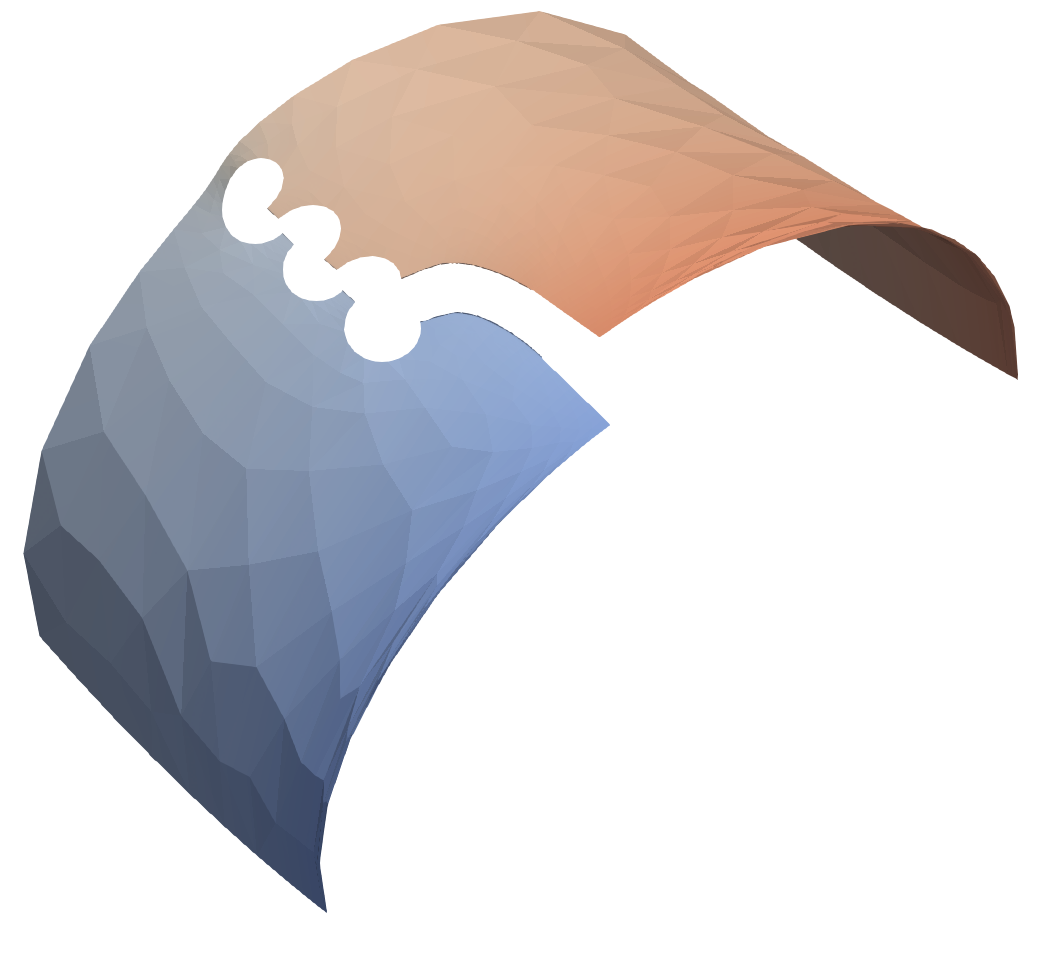}%
    \subcaption*{$t=1.31$}%
  \end{subfigure}%
  \hspace{2em}%
  \begin{subfigure}{0.28\textwidth}
    \includegraphics[width=\textwidth]{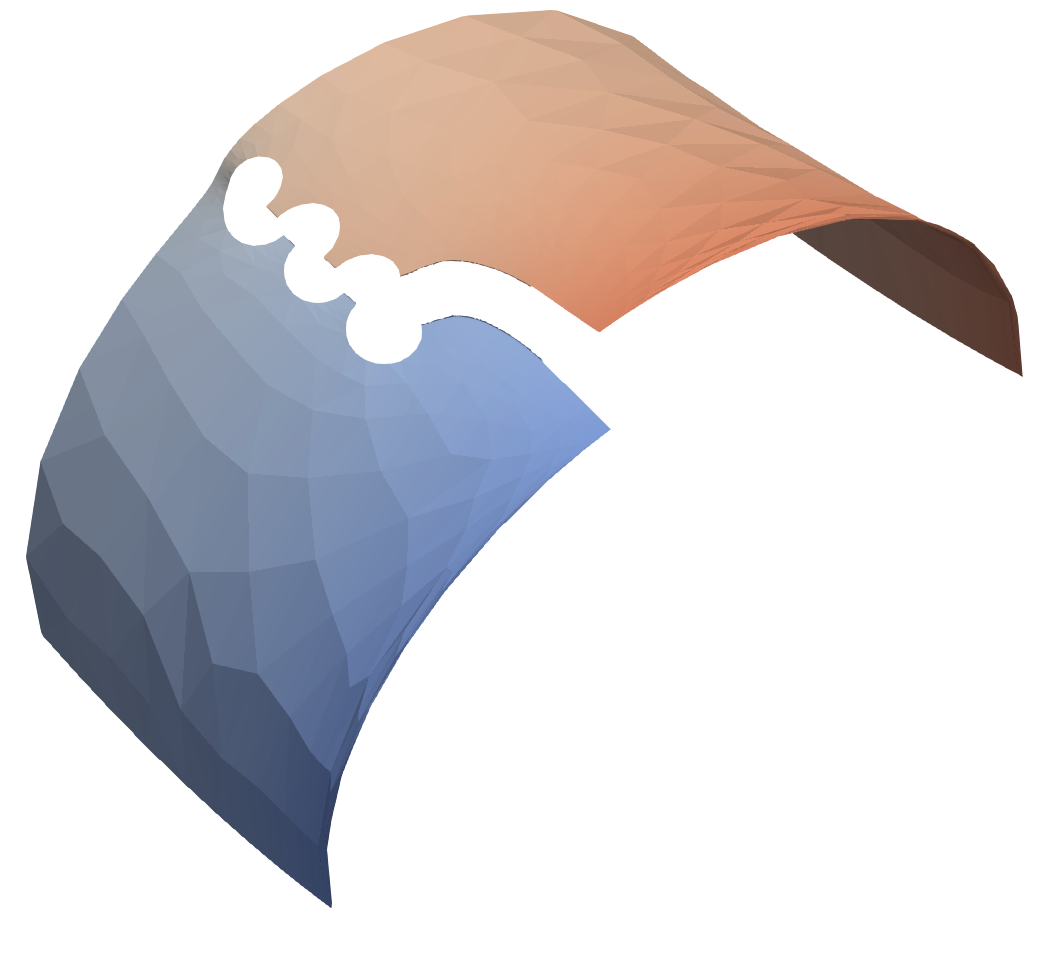}%
    \subcaption*{$t=1.45$}%
  \end{subfigure}%
  \hspace{2em}%
  \begin{subfigure}{0.28\textwidth}
    \includegraphics[width=\textwidth]{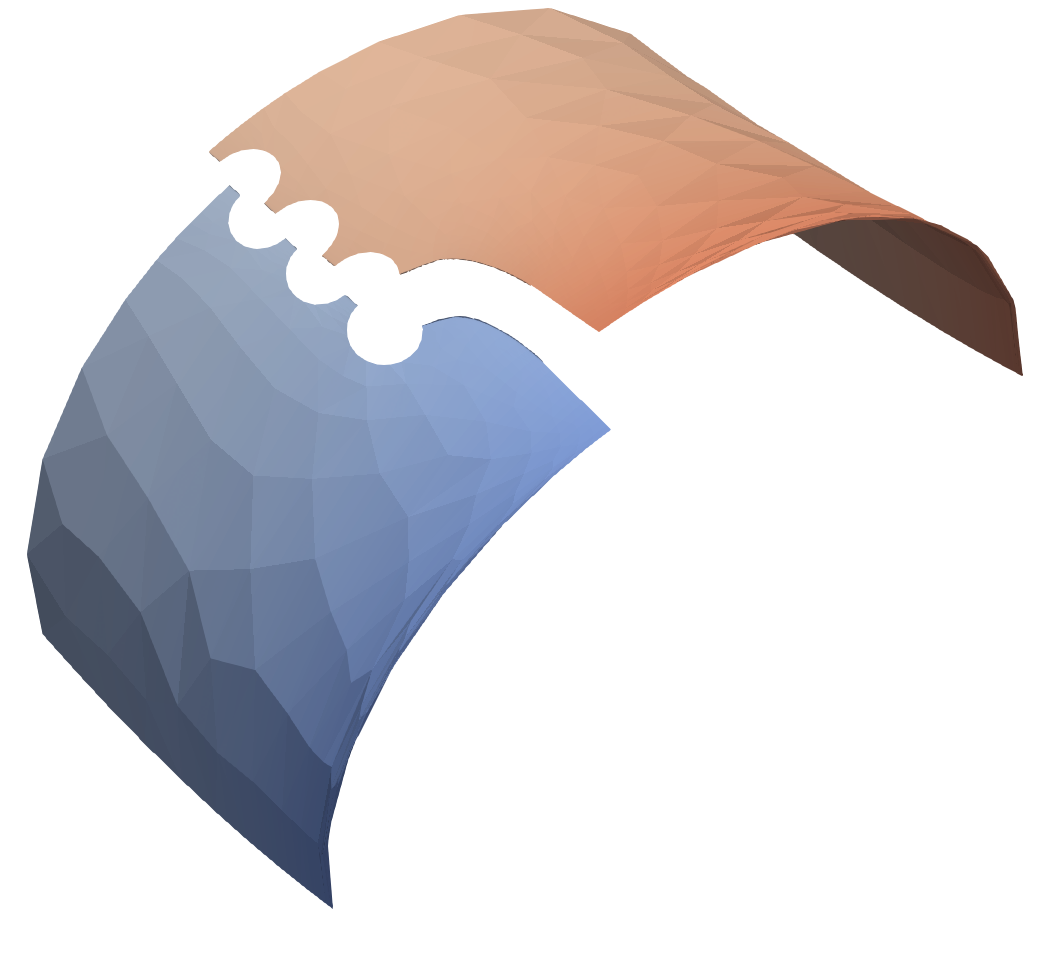}%
    \subcaption*{$t=1.46$}%
  \end{subfigure}%
  \caption{Piece of a cylinder: specimen deformation for the three-hole configuration at different times.}%
  \label{fig:smart-deformation}%
\end{figure}

Finally, we consider the effect, i.e., the deformation, induced by the crack propagation on the specimen for the three-hole configuration.
With this aim, we apply to the undeformed surface the computed displacement~$u_h$ along the normal direction~$a^3$.
However, for visualization purposes, we remove the points of the surface where the phase field is below a certain threshold, here set to  $10^{-2}$,
to model the physical crack.
Figure~\ref{fig:smart-deformation} gathers twelve snapshots tracking the whole evolution of the crack, from the undamaged initial configuration to the
complete breaking of the specimen.

\subsection{A Piece of a Sphere}
As a second test case, we consider a portion of a sphere with radius $R=1$. We adopt the parametrization
\begin{equation*}
  (x , y)  \mapsto R
  \begin{pmatrix}
    \cos x \cos y \\
    \sin x \cos y \\
    \sin y
  \end{pmatrix}
  \quad \text{for } (x, y) \in \omega\coloneq (- \bar x , \bar x) \times (-\bar y, \bar y)\,,
\end{equation*}
for some $0 < \bar x < \pi, 0 < \bar y < \frac{\pi}{2}$.
With this setting, we have
\begin{equation*}
  (a^{\alpha \beta}) = \displaystyle \frac{1}{R^2}
  \begin{pmatrix}
    \displaystyle \frac{1}{\cos^2 y} & 0 \\
    0 & 1 \\
  \end{pmatrix}, \quad
  (b_{\alpha \beta}) = -R
  \begin{pmatrix}
    \cos^2 y & 0 \\
    0 & 1  \\
  \end{pmatrix}
  \quad \text{and} \quad \sqrt{a} = R^2 \cos y \,.
\end{equation*}
\begin{figure}
  \centering
    \includegraphics[width=0.45\textwidth]{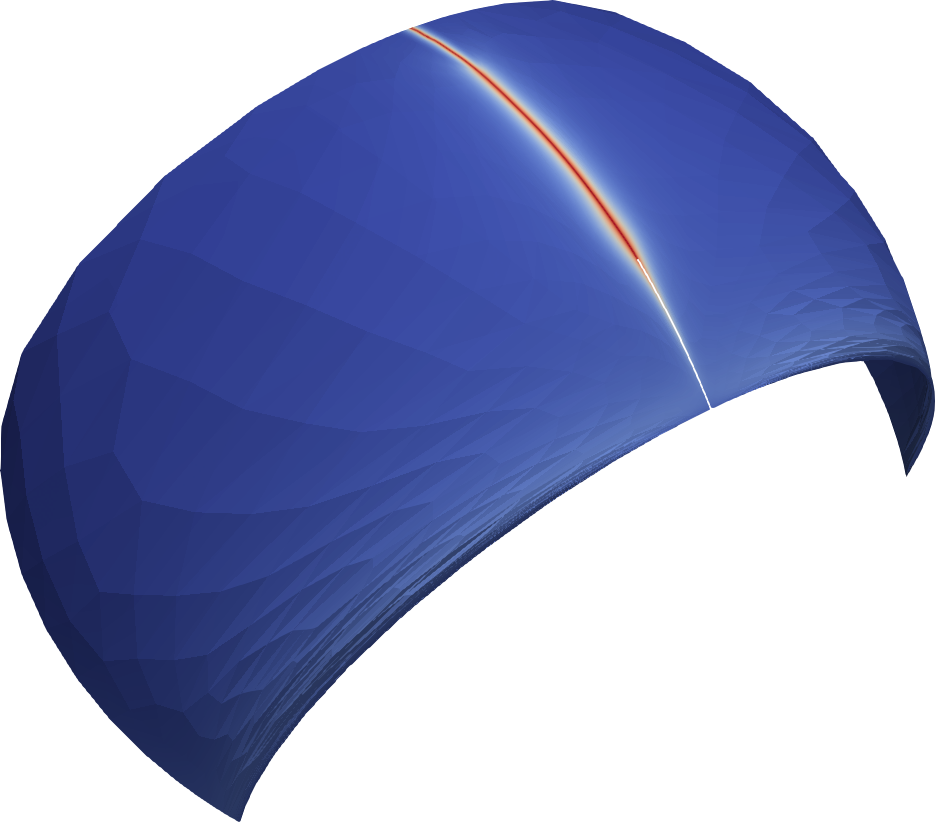}%
    \qquad
    \includegraphics[width=0.45\textwidth]{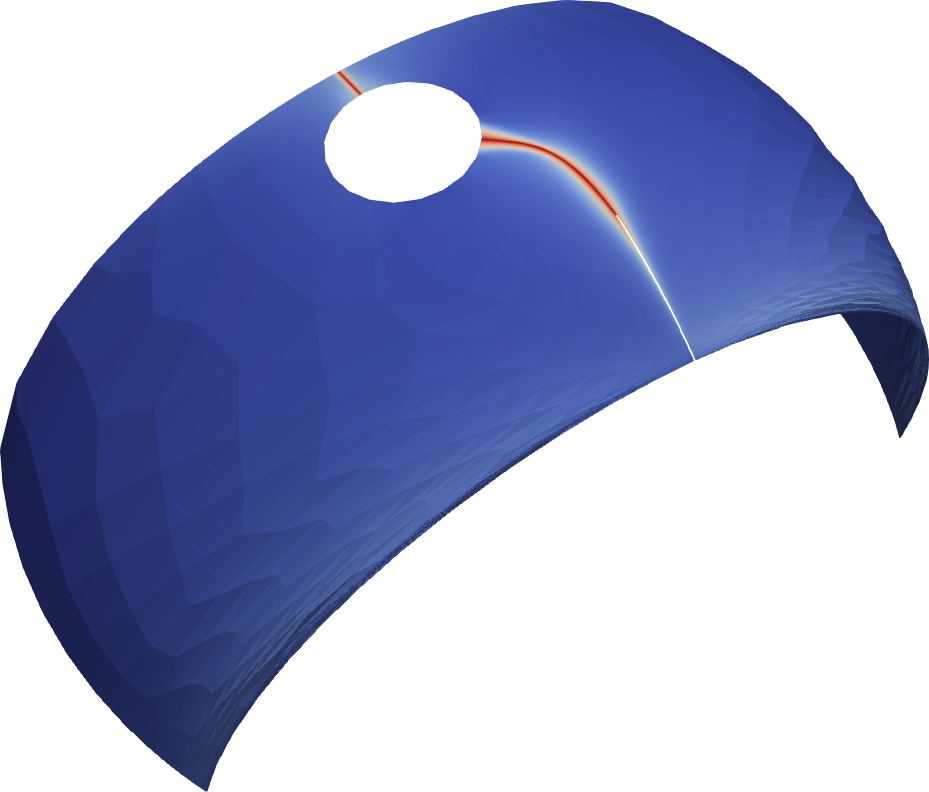}%
  \caption{Piece of a sphere: phase field for the plain (left) and for the single-hole (right) configuration.}%
  \label{fig:sphere}%
\end{figure}

We set $\bar x = \frac{\pi}{2}$ and we make two different choices for $\bar y$. Concerning the initial notch, we choose
$\Gamma \coloneq [-10^{-3} , 10^{-3}] \times [-\bar y, 0.3-\bar y]$ and we select $g$ as in \eqref{eq:bc}
for the Dirichlet boundary condition.

Figure~\ref{fig:sphere} shows on the left the final phase field at $t=2.38$ and for $\bar y = \frac{\pi}{6}$. Analogously as in the previous section,
we modify the plain configuration by digging a hole with center at $(-0.25,0.5)$ and radius $0.15$. The associated function $v_h$, for
$\bar y = \frac{\pi}{7}$, is displayed on the right of Figure~\ref{fig:sphere} for $t=2.64$. The choice for $\bar y$ avoids the generation of a secondary crack
along the Dirichlet boundary, consistently with what remarked for the piece of cylinder test case.


%% file: numerics/figures/CrackLength.tex
\begingroup
\footnotesize
  \makeatletter
  \providecommand\color[2][]{%
    \GenericError{(gnuplot) \space\space\space\@spaces}{%
      Package color not loaded in conjunction with
      terminal option `colourtext'%
    }{See the gnuplot documentation for explanation.%
    }{Either use 'blacktext' in gnuplot or load the package
      color.sty in LaTeX.}%
    \renewcommand\color[2][]{}%
  }%
  \providecommand\includegraphics[2][]{%
    \GenericError{(gnuplot) \space\space\space\@spaces}{%
      Package graphicx or graphics not loaded%
    }{See the gnuplot documentation for explanation.%
    }{The gnuplot epslatex terminal needs graphicx.sty or graphics.sty.}%
    \renewcommand\includegraphics[2][]{}%
  }%
  \providecommand\rotatebox[2]{#2}%
  \@ifundefined{ifGPcolor}{%
    \newif\ifGPcolor
    \GPcolortrue
  }{}%
  \@ifundefined{ifGPblacktext}{%
    \newif\ifGPblacktext
    \GPblacktexttrue
  }{}%
  \let\gplgaddtomacro\g@addto@macro
  \gdef\gplbacktext{}%
  \gdef\gplfronttext{}%
  \makeatother
  \ifGPblacktext
    \def\colorrgb#1{}%
    \def\colorgray#1{}%
  \else
    \ifGPcolor
      \def\colorrgb#1{\color[rgb]{#1}}%
      \def\colorgray#1{\color[gray]{#1}}%
      \expandafter\def\csname LTw\endcsname{\color{white}}%
      \expandafter\def\csname LTb\endcsname{\color{black}}%
      \expandafter\def\csname LTa\endcsname{\color{black}}%
      \expandafter\def\csname LT0\endcsname{\color[rgb]{1,0,0}}%
      \expandafter\def\csname LT1\endcsname{\color[rgb]{0,1,0}}%
      \expandafter\def\csname LT2\endcsname{\color[rgb]{0,0,1}}%
      \expandafter\def\csname LT3\endcsname{\color[rgb]{1,0,1}}%
      \expandafter\def\csname LT4\endcsname{\color[rgb]{0,1,1}}%
      \expandafter\def\csname LT5\endcsname{\color[rgb]{1,1,0}}%
      \expandafter\def\csname LT6\endcsname{\color[rgb]{0,0,0}}%
      \expandafter\def\csname LT7\endcsname{\color[rgb]{1,0.3,0}}%
      \expandafter\def\csname LT8\endcsname{\color[rgb]{0.5,0.5,0.5}}%
    \else
      \def\colorrgb#1{\color{black}}%
      \def\colorgray#1{\color[gray]{#1}}%
      \expandafter\def\csname LTw\endcsname{\color{white}}%
      \expandafter\def\csname LTb\endcsname{\color{black}}%
      \expandafter\def\csname LTa\endcsname{\color{black}}%
      \expandafter\def\csname LT0\endcsname{\color{black}}%
      \expandafter\def\csname LT1\endcsname{\color{black}}%
      \expandafter\def\csname LT2\endcsname{\color{black}}%
      \expandafter\def\csname LT3\endcsname{\color{black}}%
      \expandafter\def\csname LT4\endcsname{\color{black}}%
      \expandafter\def\csname LT5\endcsname{\color{black}}%
      \expandafter\def\csname LT6\endcsname{\color{black}}%
      \expandafter\def\csname LT7\endcsname{\color{black}}%
      \expandafter\def\csname LT8\endcsname{\color{black}}%
    \fi
  \fi
    \setlength{\unitlength}{0.0500bp}%
    \ifx\gptboxheight\undefined%
      \newlength{\gptboxheight}%
      \newlength{\gptboxwidth}%
      \newsavebox{\gptboxtext}%
    \fi%
    \setlength{\fboxrule}{0.5pt}%
    \setlength{\fboxsep}{1pt}%
\begin{picture}(4320.00,2880.00)%
    \gplgaddtomacro\gplbacktext{%
      \csname LTb\endcsname
      \put(592,575){\makebox(0,0)[r]{\strut{}$0$}}%
      \put(592,932){\makebox(0,0)[r]{\strut{}$0.2$}}%
      \put(592,1290){\makebox(0,0)[r]{\strut{}$0.4$}}%
      \put(592,1647){\makebox(0,0)[r]{\strut{}$0.6$}}%
      \put(592,2004){\makebox(0,0)[r]{\strut{}$0.8$}}%
      \put(592,2362){\makebox(0,0)[r]{\strut{}$1$}}%
      \put(592,2719){\makebox(0,0)[r]{\strut{}$1.2$}}%
      \put(688,352){\makebox(0,0){\strut{}$0$}}%
      \put(1524,352){\makebox(0,0){\strut{}$0.5$}}%
      \put(2360,352){\makebox(0,0){\strut{}$1$}}%
      \put(3195,352){\makebox(0,0){\strut{}$1.5$}}%
      \put(4031,352){\makebox(0,0){\strut{}$2$}}%
    }%
    \gplgaddtomacro\gplfronttext{%
      \csname LTb\endcsname
      \put(144,1647){\rotatebox{-270}{\makebox(0,0){\strut{}crack length}}}%
      \put(2359,112){\makebox(0,0){\strut{}time}}%
      \csname LTb\endcsname
      \put(1423,2576){\makebox(0,0)[l]{\strut{}(a)}}%
      \csname LTb\endcsname
      \put(1423,2416){\makebox(0,0)[l]{\strut{}(b)}}%
      \csname LTb\endcsname
      \put(1423,2256){\makebox(0,0)[l]{\strut{}(c)}}%
    }%
    \gplbacktext
    \put(0,0){\includegraphics{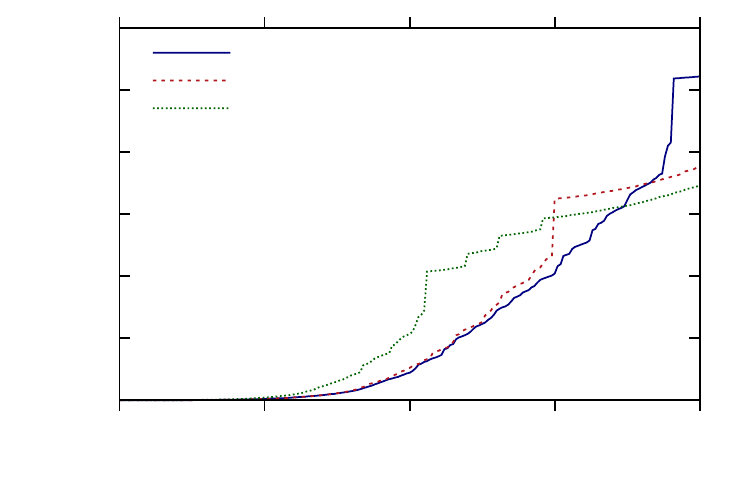}}%
    \gplfronttext
  \end{picture}%
\endgroup

%% file: numerics/figures/NumberOfElements.tex
\begingroup
\footnotesize
  \makeatletter
  \providecommand\color[2][]{%
    \GenericError{(gnuplot) \space\space\space\@spaces}{%
      Package color not loaded in conjunction with
      terminal option `colourtext'%
    }{See the gnuplot documentation for explanation.%
    }{Either use 'blacktext' in gnuplot or load the package
      color.sty in LaTeX.}%
    \renewcommand\color[2][]{}%
  }%
  \providecommand\includegraphics[2][]{%
    \GenericError{(gnuplot) \space\space\space\@spaces}{%
      Package graphicx or graphics not loaded%
    }{See the gnuplot documentation for explanation.%
    }{The gnuplot epslatex terminal needs graphicx.sty or graphics.sty.}%
    \renewcommand\includegraphics[2][]{}%
  }%
  \providecommand\rotatebox[2]{#2}%
  \@ifundefined{ifGPcolor}{%
    \newif\ifGPcolor
    \GPcolortrue
  }{}%
  \@ifundefined{ifGPblacktext}{%
    \newif\ifGPblacktext
    \GPblacktexttrue
  }{}%
  \let\gplgaddtomacro\g@addto@macro
  \gdef\gplbacktext{}%
  \gdef\gplfronttext{}%
  \makeatother
  \ifGPblacktext
    \def\colorrgb#1{}%
    \def\colorgray#1{}%
  \else
    \ifGPcolor
      \def\colorrgb#1{\color[rgb]{#1}}%
      \def\colorgray#1{\color[gray]{#1}}%
      \expandafter\def\csname LTw\endcsname{\color{white}}%
      \expandafter\def\csname LTb\endcsname{\color{black}}%
      \expandafter\def\csname LTa\endcsname{\color{black}}%
      \expandafter\def\csname LT0\endcsname{\color[rgb]{1,0,0}}%
      \expandafter\def\csname LT1\endcsname{\color[rgb]{0,1,0}}%
      \expandafter\def\csname LT2\endcsname{\color[rgb]{0,0,1}}%
      \expandafter\def\csname LT3\endcsname{\color[rgb]{1,0,1}}%
      \expandafter\def\csname LT4\endcsname{\color[rgb]{0,1,1}}%
      \expandafter\def\csname LT5\endcsname{\color[rgb]{1,1,0}}%
      \expandafter\def\csname LT6\endcsname{\color[rgb]{0,0,0}}%
      \expandafter\def\csname LT7\endcsname{\color[rgb]{1,0.3,0}}%
      \expandafter\def\csname LT8\endcsname{\color[rgb]{0.5,0.5,0.5}}%
    \else
      \def\colorrgb#1{\color{black}}%
      \def\colorgray#1{\color[gray]{#1}}%
      \expandafter\def\csname LTw\endcsname{\color{white}}%
      \expandafter\def\csname LTb\endcsname{\color{black}}%
      \expandafter\def\csname LTa\endcsname{\color{black}}%
      \expandafter\def\csname LT0\endcsname{\color{black}}%
      \expandafter\def\csname LT1\endcsname{\color{black}}%
      \expandafter\def\csname LT2\endcsname{\color{black}}%
      \expandafter\def\csname LT3\endcsname{\color{black}}%
      \expandafter\def\csname LT4\endcsname{\color{black}}%
      \expandafter\def\csname LT5\endcsname{\color{black}}%
      \expandafter\def\csname LT6\endcsname{\color{black}}%
      \expandafter\def\csname LT7\endcsname{\color{black}}%
      \expandafter\def\csname LT8\endcsname{\color{black}}%
    \fi
  \fi
    \setlength{\unitlength}{0.0500bp}%
    \ifx\gptboxheight\undefined%
      \newlength{\gptboxheight}%
      \newlength{\gptboxwidth}%
      \newsavebox{\gptboxtext}%
    \fi%
    \setlength{\fboxrule}{0.5pt}%
    \setlength{\fboxsep}{1pt}%
\begin{picture}(4464.00,2880.00)%
    \gplgaddtomacro\gplbacktext{%
      \csname LTb\endcsname
      \put(784,575){\makebox(0,0)[r]{\strut{}$0$}}%
      \put(784,932){\makebox(0,0)[r]{\strut{}$2000$}}%
      \put(784,1290){\makebox(0,0)[r]{\strut{}$4000$}}%
      \put(784,1647){\makebox(0,0)[r]{\strut{}$6000$}}%
      \put(784,2004){\makebox(0,0)[r]{\strut{}$8000$}}%
      \put(784,2362){\makebox(0,0)[r]{\strut{}$10000$}}%
      \put(784,2719){\makebox(0,0)[r]{\strut{}$12000$}}%
      \put(880,352){\makebox(0,0){\strut{}$0$}}%
      \put(1704,352){\makebox(0,0){\strut{}$0.5$}}%
      \put(2528,352){\makebox(0,0){\strut{}$1$}}%
      \put(3351,352){\makebox(0,0){\strut{}$1.5$}}%
      \put(4175,352){\makebox(0,0){\strut{}$2$}}%
    }%
    \gplgaddtomacro\gplfronttext{%
      \csname LTb\endcsname
      \put(144,1647){\rotatebox{-270}{\makebox(0,0){\strut{}number of triangles}}}%
      \put(2527,112){\makebox(0,0){\strut{}time}}%
      \csname LTb\endcsname
      \put(1615,2576){\makebox(0,0)[l]{\strut{}(a)}}%
      \csname LTb\endcsname
      \put(1615,2416){\makebox(0,0)[l]{\strut{}(b)}}%
      \csname LTb\endcsname
      \put(1615,2256){\makebox(0,0)[l]{\strut{}(c)}}%
    }%
    \gplbacktext
    \put(0,0){\includegraphics{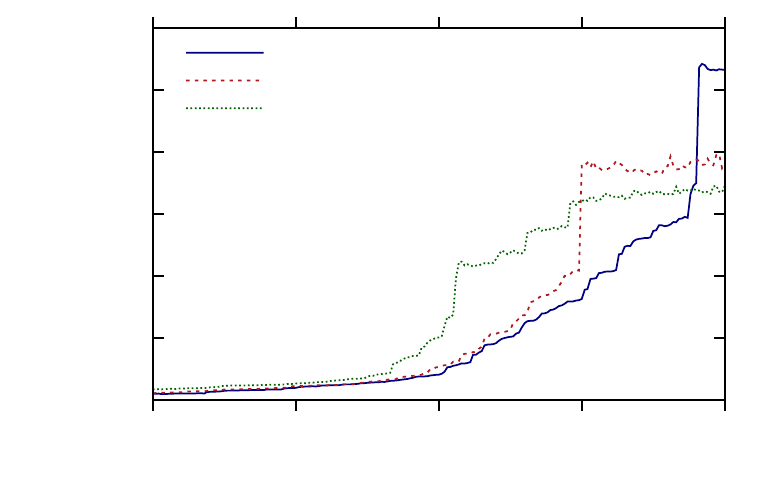}}%
    \gplfronttext
  \end{picture}%
\endgroup

%% file: appendix.tex
\section{Proof of Theorem~\ref{thm:h1-phase-field}}\label{appendix}
\renewcommand{\thelemma}{\Alph{section}.\arabic{lemma}}

In order to prove Theorem~\ref{thm:h1-phase-field} we need the following two lemmas.

\begin{lemma}
	\label{lemma:liminf-10}
	Let $I \subset \R$, $f,g\in C^1(\bar I)$ with $f,g>0$ in $\bar I$. Assume that $(u_\eps, v_\eps) \to (u,v)$ in $L^1(I)$ as $\eps \to 0$ and that
	\begin{equation}
	\label{eq:liminf-10}
	\liminf_{\eps \to 0} \, \frac{\mu}{2} \int_I v_\eps^2 \abs{u_\eps'}^2 \,\dd x \\
	+ \kappa \int_I \bigg[ \frac{1}{4\eps}  (1 - v_\eps)^2 f +  \eps \abs{v_\eps'}^{2} g \bigg]\,\dd x < +\infty \,.
	\end{equation}
	Then, there holds $v=1$ a.e.~and
	\begin{equation}
	\label{eq:liminf-00}
	\int_{S_u} \sqrt{fg} \,\dd \HH^0 \leq \liminf_{\eps \to 0} \int_I \bigg[ \frac{1}{4 \eps} (1 - v_\eps)^2 f + \eps \abs{v_\eps'}^2 g \bigg] \,\dd x\,.
	\end{equation}
\end{lemma}

\begin{proof}
	Up to a subsequence, we assume that the $\liminf$ in \eqref{eq:liminf-10} is actually a limit. All the involved limits in the proof are considered as $\eps \to 0$.

	We have $v=1$ a.e.~in~$I$, since otherwise $\frac{1}{4\eps} \int_I (1-v_\eps)^2 f \,\dd x \to + \infty$.
	In order to prove~\eqref{eq:liminf-00}, we fix~$y_0 \in S_u$ and $\delta > 0$ such that $B_\delta (y_0) \subset I $. Arguing as in \cite{Foc2001,Bra1998}, we find a sequence $(y_\eps)_{\eps>0}$ in $B_{\frac{\delta}{2}}(y_0)$ such that $v_\eps (y_\eps) \to 0$. Since $v_\eps \to 1$ a.e.~in~$I$, there exist $y^+, y^- \in B_\delta (y_0)$ such that $y^- < y_0 < y^+$ and $v_\eps (y^\pm) \to 1$.

	It is easy to compute that
	\begin{equation*}
	1 = 
	\lim_{\eps \to 0} \biggl(\int_{y_\eps}^{y^+} (1-v_\eps ) v_\eps' \,\dd x + \int_{y_\eps}^{y^-} (1-v_\eps ) v_\eps' \,\dd x \biggr) \leq \liminf_{\eps \to 0} \int_{B_\delta (y_0)} (1-v_\eps ) \abs{v_\eps'} \, \dd x \,.
	\end{equation*}
	Therefore, by Young's inequality we obtain
	\begin{align*}
	\inf_{B_\delta (y_0)} \sqrt{fg} &\leq \liminf_{\eps \to 0} \int_{B_\delta (y_0)} (1-v_\eps ) \abs{v_\eps'} \sqrt{fg} \,\dd x \\
	&\leq \liminf_{\eps\to 0} \int_{B_\delta (y_0)} \bigg[ \frac{1}{4 \eps} \bigl(1-v_\eps \bigr)^2 f + \eps \abs{v_\eps'}^2 g \bigg] \,\dd x \,.
	\end{align*}

	For each element in any discrete set $\{y_1, \dotsc, y_N\} \subset S_u$ (with $N  \leq \# S_u$) we can repeat the above argument for all $\delta > 0$ such that $B_\delta(y_k) \cap B_\delta(y_\ell) = \emptyset$ for $k \neq \ell$, in order to obtain
	\begin{equation}\label{e.smth}
	N \inf_{I} \sqrt{fg} \leq \sum_{i=1}^N \inf_{B_{\delta}(y_i)} \sqrt{fg} \leq \liminf_{\eps\to 0} \int_{I} \bigg[ \frac{1}{4 \eps} (1- v_\eps)^2 f + \eps \abs{v'_\eps}^2 g \bigg] \,\dd x\,.
	\end{equation}
	Because of \eqref{eq:liminf-10}, the right-hand side of~\eqref{e.smth} is uniformly bounded. Therefore,~$\# S_u$ must be finite and we can conclude~\eqref{eq:liminf-00} by taking the limit as $\delta \to 0$.
\end{proof}

The $\limsup$-inequality is first shown for a certain class of functions which are dense in the set $\GSBV^2(\om) \cap L^1(\om)$ (see \cite{CorToa1999}).
\begin{lemma}
	\label{lemma:limsup-10}
	Let $u\in \SBV^2(\om)$ be such that
	\begin{enumerate}
		\everymath{\displaystyle}
		\item
		$\overline{S_u}$ is the intersection of $\om$ with a finite number of pairwise disjoint $(n-1)$\nobreakdash-simplexes;
		\item
		$\HH^{n-1}\bigl(\overline{S_u} \setminus S_{u} \bigr) = 0$;
		\item
		$u\in W^{k, \infty} (\om \setminus \overline{S_u})$ for all $k\in \N$.
	\end{enumerate}
	Then, there exists a sequence $(u_\eps, v_\eps)$ converging to $(u,1)$ in $L^1(\om)$ as $\eps \to 0$ such that
	\begin{equation}\label{e:sup_inequality}
	\limsup_{\eps \to 0} \F_\eps (u_\eps, v_\eps) \leq \F (u,v) \,.
	\end{equation}
	\end{lemma}
\begin{proof}
	Throughout the proof, $C>0$ denotes an arbitrary constant independent of $\eps > 0$, which may vary from line to line, and the limits are considered as $\eps \to 0$.

	For the construction of a recovery sequence of $u$, we choose a smooth cut off function $\phi\colon \R \to \R$ with $\phi =1$ on $B_\half (0)$ and $\phi = 0$ on $\R \setminus B_1 (0)$. For all $x\in \om$, define $\tau(x) = \dist(x,S_u)$ and $\phi_\eps (x) = \phi (\frac{\tau(x)}{\delta_\eps})$  for all $\eps > 0$, where $\delta_\eps \coloneq \sqrt{\eps \eta_\eps}$. In this way, we have $\frac{\delta_\eps}{\eps} \to 0$ and $\frac{\eta_\eps}{\delta_\eps} \to 0$. Let us consider the functions $u_\eps = (1 - \phi_\eps) u$ on $\om$. Then, we have $u_\eps \in H^1(\om)$, $u_\eps = u$ on $\om \setminus B_{\delta_\eps} (S_u)$ and $u_\eps \to u$ in $L^1(\om)$.

	In order to construct the recovery sequence corresponding to $v=1$ a.e., we define $\sigma \colon [0,\infty) \to [0,1]$ by $\sigma(t) = 1 - \exp (-\frac{t}{2})$, which solves the initial value problem
	\begin{align*}
	\sigma' &= \half(1-\sigma) \,, \qquad \sigma(0) = 0 \,.
	\end{align*}
	We note that $\sigma$ is a strictly increasing, Lipschitz continuous function and $\sigma(t) \to 1$ as $t\to \infty$. For simplicity of notation, we set
	\begin{equation*}
	\varphi (\zeta , x) \coloneq  \left(\frac{\zeta^\top A \zeta}{\sqrt{a}}\right)^{\frac12} \quad \text{and} \quad
	\tilde\tau ( x)   = \frac{\tau}{ \varphi(\nabla \tau, x) } \quad \tforall \zeta \in \R^n, \, x\in \om\,.
	\end{equation*}
	We notice that by the properties of $A$ and by Section~3.2.34 in \cite{Fed1969} we can define $0<d \coloneq \inf_{ x \in \om} \varphi(\nabla\tau, x)$ and  $\infty > D \coloneq \sup_{x\in \om} \varphi(\nabla\tau, x) $.
	Furthermore, we set $\tilde \delta_\eps \coloneq \frac{\delta_\eps}{\eps d}$ for all $\eps >0$ and
	\begin{equation*}
	\rho_\eps \coloneq D\eps\Biggl(\tilde \delta_\eps - 2 \ln \biggl(\frac{\eps}{1+\eps} \biggr) \Biggr) \,,
	\end{equation*}
	so that $\rho_\eps\to 0$ as $\eps \to 0$.
	Now we define, for every $t >0$ and for every $x\in \om$,
	\begin{equation*}
	\sigma_\eps(t) \coloneq \left\{
	\begin{aligned}
	&0 && \text{for } t\in [0, \tilde \delta_\eps) \\
	&\min \bigl\{ 1, (1 + \eps) \sigma(t - \tilde \delta_\eps) \bigr\} && \text{otherwise}
	\end{aligned}
	\right. \quad \text{and} \quad v_\eps(x) \coloneq \sigma_\eps\biggl( \frac{\tilde\tau(x)}{\eps} \biggr) \,.
	\end{equation*}

	Now, the sequence $(u_\eps, v_\eps)$ will be used as the recovery sequence for $(u,1)$.
	It is easy to check that, for sufficiently small $\eps >0$, there holds $v_\eps = 1$ on~$\om \setminus B_{\rho_\eps}(S_u)$ and $v_\eps =0$ on $B_{\delta_\eps} (S_u)$.
	Moreover, $\nabla v_\eps = 0$ in $B_{\delta_\eps}(S_u)$ and in $\om \setminus B_{\rho_\eps} (S_u)$, so that
	\begin{align}
	\label{eq:h1-limsup-1}
	\F_\eps (u_\eps,v_\eps) = &  \int_{\om} b \abs{u_\eps}^2 \,\dd x + \int_{\om \setminus B_{\delta_\eps}(S_u)} v^2_\eps \nabla u_\eps^\top A \nabla u_\eps  \,\dd x \\
	& + \eta_\eps \int_{\om} \nabla u_\eps^\top A \nabla u_\eps  \,\dd x + \frac{1}{4 \eps} \int_{ B_{\delta_\eps}(S_u)} \sqrt{a} \,\dd x \nonumber \\
	& + \int_{B_{\rho_\eps} (S_u)\setminus B_{\delta_\eps} (S_u)}  \biggl[ \frac{1}{4 \eps} (1 - v_\eps)^2 + \eps \varphi^2 (\nabla v_\eps, x)\biggr] \sqrt{a} \,\dd x \,. \nonumber
	\end{align}

	Let us now estimate the integrals on the right-hand side of \eqref{eq:h1-limsup-1}, separately. Since $u_\eps \to u$ in $L^1(\om)$, we have
	\begin{equation}
	\label{eq:limsup-0}
	\int_\om b \abs{u_\eps}^2\,\dd x \to \int_\om b\abs{u}^2 \,\dd x \quad\text{as } \eps \to 0 \,.
	\end{equation}
	As shown in \cite{AmbTor1992,Foc2001,DalIur2013}, we observe that
	\begin{align}
	\label{eq:h1-limsup-2}
	\int_{\om \setminus B_{\delta_\eps}(S_u)} v_\eps^2 \nabla u_\eps^\top A \nabla u_\eps  \,\dd x &\leq \int_{\om} \nabla u^\top A \nabla u \,\dd x \,,\\
	%
	\label{eq:h1-limsup-4}
	\eta_\eps \int_{\om} \nabla u^\top_\eps A \nabla u_\eps  \,\dd x &\to 0 \quad \text{as } \eps \to 0 \,,\\
	%
	\label{eq:h1-limsup-5}
	\frac{1}{4 \eps} \int_{ B_{\delta_\eps}(S_u)}  \sqrt{a} \,\dd x
	&\to 0 \quad\text{as } \eps \to 0\,.
	\end{align}

	Concerning the last term in~\eqref{eq:h1-limsup-1}, we introduce the notation
	\begin{equation*}
	\K_\eps (v_\eps) \coloneq  \int_{B_{\rho_\eps}(S_u)\setminus B_{\delta_\eps}(S_u)} \biggl[  \frac{1}{4 \eps} (1 - v_\eps)^2 + \eps  \varphi^2 (\nabla v_\eps, x) \biggr] \sqrt{a} \,\dd x \,.
	\end{equation*}
	Precisely, we need to show that
	\begin{equation}
	\label{eq:h1-limsup-key}
	\limsup_{\eps \to 0}  \K_\eps (v_\eps) \leq \int_{S_u} \sqrt{\nabla v^\top A \nabla v \sqrt{a}} \,\dd \HH^{n-1} \,.
	\end{equation}
	This inequality, together with \eqref{eq:limsup-0}--\eqref{eq:h1-limsup-5}, allows us to conclude the assertion~\eqref{e:sup_inequality} by taking the $\limsup$ in \eqref{eq:h1-limsup-1}.

	By the assumption on $\overline{S_u}$, it holds $\overline{S_u} = \bigcup_{i=1}^N \overline{S_u^i}$ for some $N \in \N$ and for some pairwise disjoint $(n-1)$\nobreakdash-simplexes $\overline{S_u^1}, \dotsc, \overline{S_u^N}$, so that,
	for sufficiently small $\eps >0$,
	we can rewrite~$\K_\eps (v_\eps)$ as
	\begin{equation*}
	\K_\eps (v_\eps) =  \sum_{i=1}^N \int_{B_{\rho_\eps}(S_u^i)\setminus B_{\delta_\eps}(S_u^i)}  \biggl[ \frac{1}{4\eps} (1 - v_\eps)^2 + \eps \varphi^2 (\nabla\tau, x) \biggr] \sqrt{a} \,\dd x \,.
	\end{equation*}
	Hence, without loss of generality, we may assume that $\overline{S_u}$ itself is an $(n-1)$\nobreakdash-simplex. We consider the $(n-1)$-dimensional hyperplane $\nu_u^\perp$  which contains $S_u$.

	As illustrated in Figure~\ref{fig:split_integration}, we split the integration domain for~$\K_\eps$ in several parts, namely,
	\begin{equation*}
	S_u^{\pm\perp} \coloneq \{x \in \om : x = y \pm t \nu_u \text{ for some } y \in S_u \text{ and } t>0  \}\,, \quad S_u^\perp \coloneq S_u^{-\perp} \cup S_u^{+\perp},
	\end{equation*}
	and we consider
	\begin{equation}
	\label{eq:split-G}
	\K_\eps (v_\eps) = \K_\eps\rvert_{S_u^{+\perp}}  (v_\eps) + \K_\eps\rvert_{S_u^{-\perp}}  (v_\eps) + \K_\eps \rvert_{\om \setminus S_u^{\perp}}  (v_\eps) \,,
	\end{equation}
	where, for $U \subset \om$, we set
	\begin{equation}
	\label{eq:defG}
	\K_\eps\rvert_{U}  (v_\eps) \coloneq \int_{U \cap B_{\rho_\eps}(S_u)\setminus B_{\delta_\eps}(S_u)}  \biggl[ \frac{1}{4 \eps} (1-v_\eps)^2 + \eps \varphi^2 (\nabla v_\eps, x)\biggr] \sqrt{a} \,\dd x \,.
	\end{equation}

	\begin{figure}
		\centering
		\footnotesize
		\begin{tikzpicture}
		\newcommand{\deltaeps}{0.6}
		\newcommand{\rhoeps}{1.4}
		\draw [thick] (-2,0) to (2,0) node [below=7,left=8] {$S_u$};
		\draw [domain=90:270] plot ({\deltaeps * cos(\x)-2}, {\deltaeps * sin(\x)});
		\draw [domain=-90:90] plot ({\deltaeps * cos(\x)+2}, {\deltaeps * sin(\x)}) node [below=32, right=12] {$\partial B_{\delta_\eps} (S_u)$};
		\draw [domain=90:270] plot ({\rhoeps * cos(\x)-2}, {\rhoeps * sin(\x)});
		\draw [domain=-90:90] plot ({\rhoeps * cos(\x)+2}, {\rhoeps * sin(\x)}) node [below=72, right=25] {$\partial B_{\rho_\eps} (S_u)$};
		\draw (-2,\deltaeps) to (2,\deltaeps);
		\draw (-2,\rhoeps) to (2,\rhoeps);
		\draw (-2,-\deltaeps) to (2,-\deltaeps);
		\draw (-2,-\rhoeps) to (2,-\rhoeps);
		\draw [dashed] (-2,-\rhoeps - 0.6) to (-2, \rhoeps + 0.6);
		\draw [dashed] (2,-\rhoeps - 0.6) to (2, \rhoeps + 0.6);
		\draw [->, >=stealth] (-1,0) to (-1,0.5) node [below=10, right]{$\nu_u$};
		\draw node [left=40, above=45] {$S_u^{+\perp}$};
		\draw node [left=40, below=45] {$S_u^{-\perp}$};
		\draw [->, >=stealth] ({\deltaeps * cos(60) + 2}, {\deltaeps * sin(60)}) to ({(\deltaeps + 0.5) * cos(60) + 2}, {(\deltaeps + 0.5) * sin(60)}) node [right=5, below=3] {$\nabla \tau $};
		\end{tikzpicture}
		\caption{Splitting of the integration domain for $\K_\eps$.}

		\label{fig:split_integration}
	\end{figure}
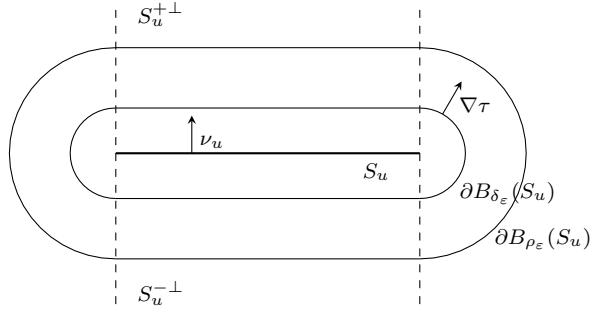

	First of all, note that, for all $x\in \om \setminus \overline{S_u}$, we have
	\begin{equation}
	\label{eq:gradient-v}
	\nabla v_\eps(x) =  \frac{1}{\eps} \sigma_\eps' \biggl( \frac{\tilde\tau(x)}{\eps} \biggr)   \biggl( \frac{\nabla \tau(x)}{\varphi(\nabla \tau(x), x)} - \frac{\tau(x) \nabla\bigl[ x\mapsto \varphi \bigl( \nabla \tau(x),x \bigr) \bigr] }{\varphi^2 (\nabla \tau(x),x)} \biggr) \,.
	\end{equation}
	In $S_u^{+\perp}$ we have that $\nabla \tau (x) = \nu_u$ is constant, and $x \mapsto \varphi (\nu_u,x)$ is Lipschitz continuous. Hence,~\eqref{eq:gradient-v} yields
	\begin{equation*}
	\varphi^2(\nabla v_\eps , x) \leq \frac{1}{\eps^2}\biggabs{\sigma_\eps' \biggl( \frac{\tilde\tau (x)}{\eps} \biggr)}^2 \bigl( 1 + C \tau (x) \bigr)^2 \,,
	\end{equation*}
	and from \eqref{eq:defG} we can estimate
	\begin{multline*}
	\K_\eps\rvert_{S_u^{+ \perp}} (v_\eps)
	\leq (1 + C\rho_\eps)^2 \int_{S_u^{+ \perp} \cap B_{\rho_\eps}(S_u)\setminus B_{\delta_\eps}(S_u)} \Biggl[ \frac{1}{4\eps} \Biggl( 1 - \sigma_\eps \biggl( \frac{\tilde \tau(x) }{\eps} \biggr) \biggr)^2 \\
	+ \frac{1}{\eps} \biggabs{\sigma_\eps' \biggl( \frac{\tilde\tau(x)}{\eps}  \Biggr)}^2 \Biggr] \sqrt{a} \,\dd x \,.
	\end{multline*}
	Together with the Coarea formula (see, e.g., Theorem 2.93~in \cite{AmbFusPal2000}), we obtain
	\begin{align}
	\label{eq:mess}
	\K_\eps\rvert_{S_u^{+ \perp}} (v_\eps)
	\leq (1+ C\rho_\eps)^2 \int_{\delta_\eps}^{\rho_\eps} & \int_{S_u^{+\perp} \cap \partial B_t (S_u) } \Biggl[ \frac{1}{4\eps}   \biggl(1 - \sigma_\eps \biggl( \frac{t}{\eps \varphi \bigl(\nu_u, x \bigr)} \biggr) \biggr)^2 \\
	& \qquad + \frac{1}{\eps}  \biggabs{\sigma_\eps' \biggl( \frac{t}{\eps \varphi \bigl(\nu_u ,x \bigr)}\biggr)}^2  \Biggr] \sqrt{a} \,\dd \HH^{n-1}\,\dd t \,. \nonumber
	\end{align}
	We apply the coordinate transformation $x \mapsto x + t \nu_u$, which maps $S_u$ to $S_u^{+\perp} \cap \partial B_t (S_u)$, to the inner integral of \eqref{eq:mess}, obtaining
	\begin{align}
	\label{eq:mess2}
	\K_\eps\rvert_{S_u^{+ \perp}} (v_\eps) \leq (1+ C\rho_\eps)^3  \int_{S_u} & \int_{\delta_\eps}^{\rho_\eps}  \Biggl[ \frac{1}{4\eps} \Biggl( 1 - \sigma_\eps \biggl( \frac{t}{\eps \varphi (\nu_u, x + t \nu_u )} \biggr) \Biggr)^2 \\
	& + \frac{1}{\eps}  \biggabs{\sigma_\eps' \biggl( \frac{t}{\eps \varphi (\nu_u , x + t \nu_u )}\biggr)}^2  \Biggr] \sqrt{a} \,\dd t \,\dd \HH^{n-1} \,, \nonumber
	\end{align}
	where we additionally used the fact that~$\sqrt{a}$ is Lipschitz and  bounded away from zero.

	Note that, by construction,
	\[
	\sigma'_\eps (t)  = \frac{1+\eps}{2} \exp\biggl(\frac{\tilde\delta_\eps-t}{2}\biggr) \quad \text{for }\tilde \delta_\eps < t < \tilde\delta_\eps -2\ln \biggl(\frac{\eps}{1+\eps}\biggr)
	\]
	and $\sigma'_\eps = 0$ otherwise. Thus, $\sigma'_\eps$ is decreasing in $(\tilde \delta_\eps , \infty)$ with supremum $(1 + \eps)/2$. Hence, with $\gamma_\eps\coloneq \tilde \delta_\eps \eps (\varphi (\nu_u,x) + C\rho_\eps)$ we can compute
	\begin{align*}
	\int_{\delta_\eps}^{\rho_\eps} \frac{1}{\eps} &  \biggabs{\sigma_\eps' \biggl( \frac{t}{\eps \varphi (\nu_u , x + t \nu_u )}\biggr)}^2   \,\dd t \\
	& \leq \int_{\gamma_\eps}^{\rho_\eps} \frac{1}{\eps}  \biggabs{\sigma_\eps' \biggl( \frac{t}{\eps( \varphi (\nu_u , x) + C \rho_\eps)}\biggr)}^2 \,\dd t + \int_{\delta_\eps}^{\gamma_\eps} \frac{(1+\eps)^2}{4\eps}  \, \dd t \\
	& \leq \int_{\delta_\eps}^{\rho_\eps} \frac{1}{\eps}  \biggabs{\sigma_\eps' \biggl( \frac{t}{\eps( \varphi (\nu_u , x) + C \rho_\eps)}\biggr)}^2 \,\dd t + C\tilde \delta_\eps\,,
	\end{align*}
	for a.e. $x\in \om$. Since $\sigma_\eps$ is increasing
	and $x \mapsto \varphi (\nu_u, x)$ is Lipschitz continuous on~$S_u$, we can estimate
	\begin{multline*}
	\!\!\!\!  \int_{\delta_\eps}^{\rho_\eps} \!\! \frac{1}{4\eps} \Biggl( 1 - \sigma_\eps \biggl( \frac{t}{\eps \varphi (\nu_u, x + t \nu_u )} \biggr) \Biggr)^2 \dd t
	\leq \int_{\delta_\eps}^{\rho_\eps} \!\! \frac{1}{4 \eps} \Biggl( 1 - \sigma_\eps  \biggl( \frac{t}{\eps (\varphi (\nu_u, x ) + C  \rho_\eps)} \biggr)\Biggr)^2 \dd t .
	\end{multline*}
	Therefore, inserting the two previous estimates in~\eqref{eq:mess2} we obtain
	\begin{align*}
	\K_\eps\rvert_{S_u^{+ \perp}} (v_\eps) \leq (1+ C\rho_\eps)^3 \int_{S_u} & \int_{\delta_\eps}^{\rho_\eps}  \Biggl[ \frac{1}{4\eps} \Biggl( 1 - \sigma_\eps\biggl( \frac{t}{\eps (\varphi (\nu_u, x ) + C  \rho_\eps)} \biggr)\Biggr)^2 \\
	&\!\!\! + \frac{1}{\eps}  \biggabs{\sigma_\eps' \biggl( \frac{t}{\eps( \varphi (\nu_u , x) + C \rho_\eps)}\biggr)}^2  \Biggr] \sqrt{a} \,\dd t \,\dd \HH^{n-1} + C\tilde\delta_\eps \,.
	\end{align*}

	We introduce another change of variables, namely $t \mapsto t  \eps (\varphi(\nu_u, x) + C \rho_\eps)$, so that
	\begin{align}
	\label{eq:mess3}
	\K_\eps\rvert_{S_u^{+ \perp}} (v_\eps) \leq (1+ C\rho_\eps)^4  & \int_{\tilde \delta_\eps}^{\tilde \delta_\eps - 2\ln\bigl(\frac{\eps}{1 + \eps} \bigr)} \bigg[  \frac{1}{4}\bigl( 1 - \sigma_\eps(t) \bigr)^2  + \bigabs{\sigma_\eps'(t)}^2 \bigg] \,\dd t \\
	&\qquad \times \int_{S_u}  \sqrt{\nu_u^\top A \nu_u \sqrt{a}} \,\dd \HH^{n-1} + C\tilde \delta_\eps \,. \nonumber
	\end{align}
	Using the explicit form of  $\sigma_\eps$, we compute the first integral on the right-hand side of~\eqref{eq:mess3} as
	\begin{multline*}
	\int_{\tilde \delta_\eps}^{\tilde \delta_\eps-2\ln\bigl(\frac{\eps}{1+\eps} \bigr)} \bigg[ \frac{1}{4}\bigl( 1 - \sigma_\eps(t) \bigr)^2  + \bigabs{\sigma_\eps'(t)}^2 \bigg]  \,\dd t \\
	\begin{aligned}
	&=\int_{0}^{-2\ln\bigl(\frac{\eps}{1 + \eps} \bigr)} \bigg[ \frac{1}{4} \bigl(1- (1 + \eps) \sigma(t) \bigr)^2 + \frac{1}{4} \bigl(1 - \sigma(t) \bigr)^2 \bigg] \,\dd t \\
	&\leq \half \int_{0}^{\infty}  \bigl( 1-  \sigma(t) \bigr)^2  \,\dd t
	= \int_{0}^{\infty}    \bigl( 1-  \sigma(t) \bigr) \sigma'(t) \,\dd t
	= \int_{0}^{1}    (1 - t) \,\dd t = \half \,.
	\end{aligned}
	\end{multline*}
	Hence, taking the limit in \eqref{eq:mess3} as $\eps \to 0$, we deduce
	\begin{equation*}
	\limsup_{\eps \to 0} \K_\eps\rvert_{S_u^{+ \perp}} (v_\eps) \leq \half \int_{S_u} \sqrt{\nu_u^\top A \nu_u \sqrt{a}} \,\dd \HH^{n-1}\,.
	\end{equation*}

	Repeating all the arguments above for $\K_\eps\rvert_{S^{-\perp}_u} (v_\eps)$ with $\nabla \tau (x) = -\nu_u$ on $S_u^{-\perp}$, we infer
	\begin{equation}
	\label{eq:mess5}
	\limsup_{\eps \to 0} \Bigl( \K_\eps\rvert_{S_u^{- \perp}} (v_\eps) + \K_\eps\rvert_{S_u^{+ \perp}} (v_\eps) \Bigr)
	\leq \int_{S_u} \sqrt{\nu_u^\top A \nu_u \sqrt{a}} \,\dd \HH^{n-1} \,.
	\end{equation}

	Finally, we show that $\K_\eps \rvert_{\om \setminus S_u^\perp} \to 0$ as $\eps \to 0$.
	For $x \in B_{\rho_{\eps}}(S_{u}) \setminus \overline{S_u}$, we claim that
	\begin{equation}
	\label{eq:bounded-gradient-phi}
	\Bigabs{ \nabla \bigl[ x \mapsto \varphi \bigl(\tau(x), x \bigr) \bigr] } \leq \frac{C}{\tau(x)} \,.
	\end{equation}
	Indeed, let $x, y \in B_{\rho_{\eps}}(S_u) \setminus \overline{S_u}$. We set $\overline \tau \coloneq \min\{\tau(x),\tau(y)\}$, $\overline x \coloneq \pi_{B_{\overline \tau} (S_u)}(x) =  \pi_{S_u} (x) + \overline \tau \nabla \tau (x)$ and $\overline y \coloneq \pi_{B_{\overline \tau} (S_u)}(y) = \pi_{S_u} (y) + \overline \tau \nabla \tau (y)$, where $\pi_{E}$ denotes the projection onto $E \subset \R^n$. Since the projection on a convex set  is Lipschitz continuous  with Lipschitz constant equal to one, we have that $\abs{\overline x-\overline y}\leq \abs{x-y}$ and
	\begin{equation}
	\label{eq:gradient-Lipschitz}
	\bigabs{\nabla \tau (x) - \nabla \tau(y)}
	= \frac{1}{\overline \tau} \Bigabs{\overline x - \pi_{S_u} ( x) - \bigl(\overline y - \pi_{S_u} (  y) \bigr)}
	\leq \frac{2}{\overline \tau} \abs{x-y} \,.
	\end{equation}
	Together with the positive definiteness of $A$, for $x, y \in B_{\rho_\eps} (S_u) \setminus S_u$ and $\eps$ sufficiently small we obtain
	\begin{align*}
	\Bigabs{ \varphi \bigl(\nabla \tau(x), x \bigr) - \varphi \bigl(\nabla \tau (y),y \bigr) }
	&\leq C \bigabs{ \nabla \tau (x) - \nabla \tau (y) }  + C \abs{x-y} \\
	&\leq \frac{C}{\min\{\tau(x), \tau(y)\}} \abs{x-y}  \,,
	\end{align*}
	which yields \eqref{eq:bounded-gradient-phi}.

	From \eqref{eq:gradient-v} we obtain
	\begin{equation*}
	\varphi^2(\nabla v_\eps(x) , x) \leq \frac{C}{\eps^2}\biggabs{\sigma_\eps' \biggl( \frac{\tilde\tau (x)}{\eps} \biggr)}^2 \quad\tforall x\in \om\setminus S_u^\perp  \,.
	\end{equation*}
	We plug the above inequality into the expression of $\K_\eps\rvert_{\om \setminus S_u^\perp} (v_\eps)$ and  apply again the Coarea formula, so that
	\begin{align}
	\label{eq:mess4}
	\K_\eps\rvert_{\om \setminus S_u^\perp} (v_\eps)
	\leq C \int_{\delta_\eps}^{\rho_\eps} \int_{\partial B_t (S_u) \setminus S_u^\perp} & \Biggl[ \frac{1}{4 \eps} \biggl(1 - \sigma_\eps \biggl( \frac{t}{\eps \varphi \bigl(\nabla\tau, x \bigr)} \biggr) \biggr)^2\\
	& + \frac{1}{\eps}  \biggabs{\sigma_\eps' \biggl( \frac{t}{\eps \varphi \bigl(\nabla\tau ,x \bigr)}\biggr)}^2  \Biggr] \sqrt{a}  \,\dd \HH^{n-1} \,\dd t \,.\nonumber
	\end{align}
	Next, we use the coordinate transformation $x \mapsto x + (t-\delta_\eps) \nabla \tau (x)$, which maps~$\partial B_{\delta_\eps} (S_u)$ onto~$\partial B_{t} (S_u)$. Note that $\nabla \tau(x) = \nabla \tau ( x+t\nabla \tau (x))$ and, from \eqref{eq:gradient-Lipschitz}, we infer that $\abs{\nabla^2 \tau } \leq \frac{C}{\delta_\eps}$ on $\partial B_{\delta_\eps} (S_u) \setminus S_u^\perp$, so that the Coarea factor is bounded by~$C{\rho_\eps}/{\delta_\eps}$. Hence, from \eqref{eq:mess4} we deduce
	\begin{multline*}
	\K_\eps\rvert_{\om \setminus S_u^\perp} (v_\eps) \leq \frac{C \rho_\eps}{\delta_\eps} \int_{\partial B_{\delta_\eps} (S_u) \setminus S_u^\perp} \int_{0}^{\rho_\eps - \delta_\eps}  \Biggl[ \frac{1}{4 \eps} \biggl(1 - \sigma_\eps  \biggl( \frac{t+\delta_\eps}{\eps \varphi (\nabla \tau, x + t \nabla \tau )} \biggr) \biggr)^2 \\
	+ \frac{1}{\eps} \biggabs{\sigma_\eps' \biggl( \frac{t+\delta_\eps}{\eps \varphi (\nabla \tau, x + t \nabla \tau  )}\biggr)}^2  \Biggr]   \sqrt{a} \,\dd t \,\dd \HH^{n-1} \,,
	\end{multline*}
	where we again use the Lipschitz continuity and the uniform strictly positive boundedness of~$a$, and additionally shift the integration domain with respect to~$t$.
	Repeating the same arguments used for the estimate of $\K_\eps \rvert_{S_u^{+\perp}}$, we obtain
	\begin{align*}
	\K_\eps\rvert_{\om \setminus S_u^\perp} (v_\eps)
	&\leq \frac{C \rho_\eps}{\delta_\eps}  \int_{{}_{\scriptstyle \partial B_{\delta_\eps} (S_u) \setminus S_u^\perp}} \!\!\!\!\!\!\!\!\!\!\!\!\!\!\!\!\!\!\!\!\! \sqrt{\nabla\tau^\top A \nabla \tau\sqrt{a}} \,\dd \HH^{n-1} \leq \frac{C \rho_\eps}{\delta_\eps} \HH^{n-1} \bigl( \partial B_{\delta_\eps} (S_u) \setminus S_u^\perp \bigr) \,.
	\end{align*}
	It is easy to check that $ \partial B_{\delta_\eps} (S_u) \setminus S_u^\perp \subset \partial B_{\delta_\eps} (\partial S_u)$, where $\partial S_u$ denotes the relative boundary of $S_u$ in the hyperplane $\nu_u^\perp$. Hence,
	\begin{equation}
	\label{eq:lastmess}
	\K_\eps\rvert_{\om \setminus S_u^\perp} (v_\eps) \leq \frac{C \rho_\eps}{\delta_\eps} \HH^{n-1} \bigl( \partial B_{\delta_\eps} (\partial S_u) \bigr) \leq C \rho_\eps \to 0 \quad \text{as } \eps \to 0\,.
	\end{equation}
	Summing up \eqref{eq:split-G}, \eqref{eq:mess5}, and \eqref{eq:lastmess}, we obtain the desired estimate~\eqref{eq:h1-limsup-key}.
\end{proof}\textsl{\textsl{}}

We now conclude the proof of Theorem~\ref{thm:h1-phase-field}.

\begin{proof} We provide a proof which folds for a generic dimension~$n$.

	We first show the $\liminf$-inequality. Let $(u_\eps,v_\eps)$ be a sequence converging to $(u,v)$ in $L^1(\om)$. We assume, without loss of generality, that
	\begin{equation*}
	\liminf_{\eps\to 0} \F_\eps(u_\eps,v_\eps) = \lim_{\eps\to 0} \F_\eps(u_\eps,v_\eps) < +\infty \,.
	\end{equation*}
	Since the norm is lower semicontinuous, we clearly have
	\begin{equation}
	\label{eq:liminf-20}
	\int_\om b \abs{u}^2 \,\dd x   \leq  \liminf_{\eps\to 0} \int_\om b \abs{u_\eps}^2 \,\dd x \,.
	\end{equation}
	Following the proof of Lemma~3.2 in \cite{Foc2001}, we obtain
	\begin{equation}
	\label{eq:liminf-21}
	\frac{\mu}{2} \int_\om \nabla u^\top A \nabla u \, \dd x \leq \liminf_{\eps \to 0} \, \frac{\mu}{2} \int_\om v_{\eps} \nabla u_\eps^\top A \nabla u_\eps \, \dd x \,,
	\end{equation}
	and by a slicing argument (see also \cite{Bra1998}) we obtain from Lemma~\ref{lemma:liminf-10}
	\begin{equation}
	\label{eq:liminf-22}
	\kappa \int_{S_u} \sqrt{\nu_u^\top A \nu_u \sqrt{a}} \,\dd\HH^{n-1} \leq \liminf_{\eps \to 0} \kappa \int_\om \bigg[ \frac{1}{4\eps}  (1 - v)^2 \sqrt{a} +  \eps \nabla v^\top A \nabla v \bigg] \,\dd x \,.
	\end{equation}
	Combining the inequalities \eqref{eq:liminf-20}--\eqref{eq:liminf-22} we deduce the required $\liminf$\nobreakdash-inequality.

	The $\Glimsup$\nobreakdash-inequality immediately follows from Lemma~\ref{lemma:limsup-10} using the density result in Theorem~3.1 in \cite{CorToa1999}.
\end{proof}

%% file: main.bbl
\begin{thebibliography}{10}

\bibitem{AlmBel2019}
S.~Almi and S.~Belz.
\newblock Consistent finite-dimensional approximation of phase-field models of
  fracture.
\newblock {\em Ann. Mat. Pura Appl. (4)}, 198(4):1191--1225, 2019.

\bibitem{AlmBelNeg2019}
S.~Almi, S.~Belz, and M.~Negri.
\newblock Convergence of discrete and continuous unilateral flows for
  {A}mbrosio--{T}ortorelli energies and application to mechanics.
\newblock {\em ESAIM Math. Model. Numer. Anal.}, 53(2):659--699, 2019.

\bibitem{AlmNeg2019}
S.~Almi and M.~Negri.
\newblock Analysis of {S}taggered {E}volutions for {N}onlinear {E}nergies in
  {P}hase {F}ield {F}racture.
\newblock {\em Arch. Ration. Mech. Anal.}, 236(1):189--252, 2020.

\bibitem{AmbCosDal1997}
L.~Ambrosio, A.~Coscia, and G.~Dal~Maso.
\newblock Fine properties of functions with bounded deformation.
\newblock {\em Arch. Rational Mech. Anal.}, 139(3):201--238, 1997.

\bibitem{AmbFusPal2000}
L.~Ambrosio, N.~Fusco, and D.~Pallara.
\newblock {\em Functions of bounded variation and free discontinuity problems}.
\newblock Oxford Mathematical Monographs. The Clarendon Press, Oxford
  University Press, New York, 2000.

\bibitem{AmbTor1990}
L.~Ambrosio and V.~M. Tortorelli.
\newblock Approximation of functionals depending on jumps by elliptic
  functionals via {$\Gamma$}-convergence.
\newblock {\em Comm. Pure Appl. Math.}, 43(8):999--1036, 1990.

\bibitem{AmbTor1992}
L.~Ambrosio and V.~M. Tortorelli.
\newblock On the approximation of free discontinuity problems.
\newblock {\em Boll. Un. Mat. Ital. B (7)}, 6(1):105--123, 1992.

\bibitem{ArtForMicPer2015b}
M.~Artina, M.~Fornasier, S.~Micheletti, and S.~Perotto.
\newblock Anisotropic adaptive meshes for brittle fractures: parameter
  sensitivity.
\newblock In {\em Numerical mathematics and advanced applications---{ENUMATH}
  2013}, volume 103 of {\em Lect. Notes Comput. Sci. Eng.}, pages 293--301.
  Springer, Cham, 2015.

\bibitem{ArtForMicPer2015}
M.~Artina, M.~Fornasier, S.~Micheletti, and S.~Perotto.
\newblock Anisotropic mesh adaptation for crack detection in brittle materials.
\newblock {\em SIAM J. Sci. Comput.}, 37(4):B633--B659, 2015.

\bibitem{ArtForMicPer2015a}
M.~Artina, M.~Fornasier, S.~Micheletti, and S.~Perotto.
\newblock The benefits of anisotropic mesh adaptation for brittle fractures
  under plane-strain conditions.
\newblock In {\em New challenges in grid generation and adaptivity for
  scientific computing}, volume~5 of {\em SEMA SIMAI Springer Ser.}, pages
  43--67. Springer, Cham, 2015.

\bibitem{Bab2006}
J.-F. Babadjian.
\newblock Quasistatic evolution of a brittle thin film.
\newblock {\em Calc. Var. Partial Dif.}, 26(1):69--118, 2006.

\bibitem{Bab2008}
J.-F. Babadjian.
\newblock Lower semicontinuity of quasi-convex bulk energies in
  {$\mathrm{SBV}$} and integral representation in dimension reduction.
\newblock {\em SIAM J. Math. Anal.}, 39(6):1921--1950, 2008.

\bibitem{BabHen2016}
J.-F. Babadjian and D.~Henao.
\newblock Reduced models for linearly elastic thin films allowing for fracture,
  debonding or delamination.
\newblock {\em Interfaces Free Bound.}, 18(4):545--578, 2016.

\bibitem{BelBre2019}
S.~Belz and K.~Bredies.
\newblock Approximation of the {M}umford-{S}hah functional by functions of
  bounded variation.
\newblock submitted, arXiv:1903.02349 [math.AP], 2019.

\bibitem{Bou2007a}
B.~Bourdin.
\newblock Numerical implementation of the variational formulation for
  quasi-static brittle fracture.
\newblock {\em Interfaces Free Bound.}, 9(3):411--430, 2007.

\bibitem{BouFraMar2000}
B.~Bourdin, G.~A. Francfort, and J.-J. Marigo.
\newblock Numerical experiments in revisited brittle fracture.
\newblock {\em J. Mech. Phys. Solids}, 48(4):797--826, 2000.

\bibitem{Bra1998}
A.~Braides.
\newblock {\em Approximation of free-discontinuity problems}, volume 1694 of
  {\em Lecture Notes in Mathematics}.
\newblock Springer-Verlag, Berlin, 1998.

\bibitem{BraFon2001}
A.~Braides and I.~Fonseca.
\newblock Brittle thin films.
\newblock {\em Appl. Math. Optim.}, 44(3):299--323, 2001.

\bibitem{BurOrtSue2010}
S.~Burke, C.~Ortner, and E.~S\"{u}li.
\newblock An adaptive finite element approximation of a variational model of
  brittle fracture.
\newblock {\em SIAM J. Numer. Anal.}, 48(3):980--1012, 2010.

\bibitem{BurOrtSue2011}
S.~Burke, C.~Ortner, and E.~S\"{u}li.
\newblock Adaptive finite element approximation of the {F}rancfort-{M}arigo
  model of brittle fracture.
\newblock In {\em Approximation and computation}, volume~42 of {\em Springer
  Optim. Appl.}, pages 297--310. Springer, New York, 2011.

\bibitem{ChaCri2019}
A.~Chambolle and V.~Crismale.
\newblock A density result in {$GSBD^p$} with applications to the approximation
  of brittle fracture energies.
\newblock {\em Arch. Ration. Mech. Anal.}, 232(3):1329--1378, 2019.

\bibitem{Cia1997}
P.~Ciarlet.
\newblock {\em Mathematial Elasticity; Volume II: Theory of Plates}, volume~27
  of {\em Studies in Mathematics and its Applications}.
\newblock Elsevier, Amsterdam, 1997.

\bibitem{Cia2000}
P.~Ciarlet.
\newblock {\em Mathematial Elasticity; Volume III: Theory of Shells}, volume~29
  of {\em Studies in Mathematics and its Applications}.
\newblock Elsevier, Amsterdam, 2000.

\bibitem{CiaRav1973}
P.~G. Ciarlet and P.-A. Raviart.
\newblock Maximum principle and uniform convergence for the finite element
  method.
\newblock {\em Comput. Methods Appl. Mech. Engrg.}, 2:17--31, 1973.

\bibitem{Cle1975}
P.~Cl\'ement.
\newblock Approximation by finite element functions using local regularization.
\newblock {\em RAIRO Anal. Numér.}, 9(R2):77--84, 1975.

\bibitem{CorToa1999}
G.~Cortesani and R.~Toader.
\newblock A density result in {SBV} with respect to non-isotropic energies.
\newblock {\em Nonlinear Anal.}, 38(5, Ser. B: Real World Appl.):585--604,
  1999.

\bibitem{Dal1993}
G.~Dal~Maso.
\newblock {\em An introduction to {$\Gamma$}-convergence}, volume~8 of {\em
  Progress in Nonlinear Differential Equations and their Applications}.
\newblock Birkh\"auser Boston, Inc., Boston, MA, 1993.

\bibitem{Dal2013}
G.~Dal~Maso.
\newblock Generalised functions of bounded deformation.
\newblock {\em J. Eur. Math. Soc. (JEMS)}, 15(5):1943--1997, 2013.

\bibitem{DalFraToa2005}
G.~Dal~Maso, G.~A. Francfort, and R.~Toader.
\newblock Quasistatic crack growth in nonlinear elasticity.
\newblock {\em Arch. Rational Mech. Anal.}, 176(2):165--225, 2005.

\bibitem{DalIur2013}
G.~Dal~Maso and F.~Iurlano.
\newblock Fracture models as {$\Gamma$}-limits of damage models.
\newblock {\em Commun. Pure Appl. Anal.}, 12(4):1657--1686, 2013.

\bibitem{DalToa2002}
G.~Dal~Maso and R.~Toader.
\newblock A model for the quasi-static growth of brittle fractures: existence
  and approximation results.
\newblock {\em Arch. Ration. Mech. Anal.}, 162(2):101--135, 2002.

\bibitem{EvaGar1992}
L.~C. Evans and R.~F. Gariepy.
\newblock {\em Measure theory and fine properties of functions}.
\newblock Studies in Advanced Mathematics. CRC Press, Boca Raton, FL, 1992.

\bibitem{Fed1969}
H.~Federer.
\newblock {\em Geometric measure theory}.
\newblock Die Grundlehren der mathematischen Wissenschaften, Band 153.
  Springer-Verlag New York Inc., New York, 1969.

\bibitem{Foc2001}
M.~Focardi.
\newblock On the variational approximation of free-discontinuity problems in
  the vectorial case.
\newblock {\em Math. Models Methods Appl. Sci.}, 11(4):663--684, 2001.

\bibitem{Foe1907}
A.~F\"{o}ppl.
\newblock {\em Vorlesungen \"{u}ber technische Mechanik}, volume~5.
\newblock B.G. Teubner, Leipzig, 1907.

\bibitem{WCCM}
L.~Formaggia, S.~Micheletti, and S.~Perotto.
\newblock Anisotropic mesh adaption with application to {CFD} problems.
\newblock In H.~Mang, F.~Rammerstorfer, and J.~Eberhardsteiner, editors, {\em
  Proceedings of WCCM V, Fifth World Congress on Computational Mechanics},
  pages 1481--1493, 2002.

\bibitem{ForMicPer2004}
L.~Formaggia, S.~Micheletti, and S.~Perotto.
\newblock Anisotropic mesh adaption in computational fluid dynamics:
  application to the advection-diffusion-reaction and the {S}tokes problems.
\newblock {\em Appl. Numer. Math.}, 51(4):511--533, 2004.

\bibitem{ForPer2001}
L.~Formaggia and S.~Perotto.
\newblock New anisotropic a priori error estimates.
\newblock {\em Numer. Math.}, 89(4):641--667, 2001.

\bibitem{ForPer2003}
L.~Formaggia and S.~Perotto.
\newblock Anisotropic error estimates for elliptic problems.
\newblock {\em Numer. Math.}, 94(1):67--92, 2003.

\bibitem{FraMar1998}
G.~A. Francfort and J.-J. Marigo.
\newblock Revisiting brittle fracture as an energy minimization problem.
\newblock {\em J. Mech. Phys. Solids}, 46(8):1319--1342, 1998.

\bibitem{FriJamMorMue2003}
G.~Friesecke, R.~D. James, M.~G. Mora, and S.~M\"{u}ller.
\newblock Derivation of nonlinear bending theory for shells from
  three-dimensional nonlinear elasticity by {G}amma-convergence.
\newblock {\em C. R. Math. Acad. Sci. Paris}, 336(8):697--702, 2003.

\bibitem{FriJamMue2002}
G.~Friesecke, R.~D. James, and S.~M\"{u}ller.
\newblock A theorem on geometric rigidity and the derivation of nonlinear plate
  theory from three-dimensional elasticity.
\newblock {\em Comm. Pure Appl. Math.}, 55(11):1461--1506, 2002.

\bibitem{FriJamMue2006}
G.~Friesecke, R.~D. James, and S.~M\"{u}ller.
\newblock A hierarchy of plate models derived from nonlinear elasticity by
  gamma-convergence.
\newblock {\em Arch. Ration. Mech. Anal.}, 180(2):183--236, 2006.

\bibitem{Gen1999}
K.~Genevey.
\newblock Justification of two-dimensional linear shell models by the use of
  {$\Gamma$}-convergence theory.
\newblock In {\em Plates and shells ({Q}u\'ebec, {QC}, 1996)}, volume~21 of
  {\em CRM Proc. Lecture Notes}, pages 185--197. Amer. Math. Soc., Providence,
  RI, 1999.

\bibitem{GeoBor1998}
P.-L. George and H.~Borouchaki.
\newblock {\em Delaunay triangulation and meshing}.
\newblock Editions Herm\`es, Paris, 1998.
\newblock Application to finite elements, Translated from the 1997 French
  original by the authors, P. J. Frey and Scott A. Canann.

\bibitem{Gia2005}
A.~Giacomini.
\newblock Ambrosio-{T}ortorelli approximation of quasi-static evolution of
  brittle fractures.
\newblock {\em Calc. Var. Partial Differential Equations}, 22(2):129--172,
  2005.

\bibitem{Hec2012}
F.~Hecht.
\newblock New development in freefem++.
\newblock {\em J. Numer. Math.}, 20(3-4):251--265, 2012.

\bibitem{Iur2013}
F.~Iurlano.
\newblock Fracture and plastic models as {$\Gamma$}-limits of damage models
  under different regimes.
\newblock {\em Adv. Calc. Var.}, 6(2):165--189, 2013.

\bibitem{KikOde1988}
N.~Kikuchi and J.~T. Oden.
\newblock {\em Contact problems in elasticity: a study of variational
  inequalities and finite element methods}, volume~8 of {\em SIAM Studies in
  Applied Mathematics}.
\newblock Society for Industrial and Applied Mathematics (SIAM), Philadelphia,
  PA, 1988.

\bibitem{Kir1850}
G.~Kirchhoff.
\newblock \"{U}ber das {G}leichgewicht und die {B}ewegung einer elastischen
  {S}chei\-be.
\newblock {\em J. Reine Angew. Math.}, 40:51--88, 1850.

\bibitem{KneNeg2017}
D.~Knees and M.~Negri.
\newblock Convergence of alternate minimization schemes for phase-field
  fracture and damage.
\newblock {\em Math. Models Methods Appl. Sci.}, 27(9):1743--1794, 2017.

\bibitem{KneRosZan2013}
D.~Knees, R.~Rossi, and C.~Zanini.
\newblock A vanishing viscosity approach to a rate-independent damage model.
\newblock {\em Math. Models Methods Appl. Sci.}, 23(4):565--616, 2013.

\bibitem{Lov1888/89}
A.~E.~H. Love.
\newblock On the {E}quilibrium of a {T}hin {E}lastic {S}pherical {B}owl.
\newblock {\em Proc. Lond. Math. Soc.}, 20:89--102, 1888/89.

\bibitem{MicPer2008}
S.~Micheletti and S.~Perotto.
\newblock Output functional control for nonlinear equations driven by
  anisotropic mesh adaption: the {N}avier-{S}tokes equations.
\newblock {\em SIAM J. Sci. Comput.}, 30(6):2817--2854, 2008.

\bibitem{MicPer2011}
S.~Micheletti and S.~Perotto.
\newblock The effect of anisotropic mesh adaptation on {PDE}-constrained
  optimal control problems.
\newblock {\em SIAM J. Control Optim.}, 49(4):1793--1828, 2011.

\bibitem{Neg2019}
M.~Negri.
\newblock A unilateral {$L^2$}-gradient flow and its quasi-static limit in
  phase-field fracture by an alternate minimizing movement.
\newblock {\em Adv. Calc. Var.}, 12(1):1--29, 2019.

\bibitem{ScoZha1990}
L.~R. Scott and S.~Zhang.
\newblock Finite element interpolation of nonsmooth functions satisfying
  boundary conditions.
\newblock {\em Math. Comp.}, 54(190):483--493, 1990.

\bibitem{StrFix2008}
G.~Strang and G.~J. Fix.
\newblock {\em An Analysis of the Finite Element Method}.
\newblock Cambridge Press, Wellesley, 2nd edition, 2008.

\bibitem{Ver1999}
R.~Verf\"{u}rth.
\newblock Error estimates for some quasi-interpolation operators.
\newblock {\em M2AN Math. Model. Numer. Anal.}, 33(4):695--713, 1999.

\bibitem{Kar1910}
T.~von K\'{a}rm\'{a}n.
\newblock Festigkeitsprobleme im maschinenbau.
\newblock In {\em Encyklopädie der mathematischen Wissenschaften mit
  Einschluss ihrer Anwedungen}, volume IV, 4, pages 314--385. B.G. Teubner,
  Leipzig, 1910.

\bibitem{WaeBie2006}
A.~W{\"a}chter and L.~T. Biegler.
\newblock On the implementation of an interior-point filter line-search
  algorithm for large-scale nonlinear programming.
\newblock {\em Mathematical Programming}, 106(1):25--57, Mar 2006.

\bibitem{ZZ87}
O.~C. Zienkiewicz and J.~Z. Zhu.
\newblock A simple error estimator and adaptive procedure for practical
  engineering analysis.
\newblock {\em Int. J. Numer. Meth. Engng}, 24:337--357, 1987.

\bibitem{ZZ92}
O.~C. Zienkiewicz and J.~Z. Zhu.
\newblock The superconvergent patch recovery and a posteriori error estimates.
  ii: Error estimates and adaptivity.
\newblock {\em Int. J. Numer. Meth. Engng}, 33:1365--1382, 1992.

\end{thebibliography}
